\documentclass[11pt,a4paper,reqno]{amsart}

\usepackage[top=2.5cm,bottom=2.5cm,left=2.5cm,right=2.5cm]{geometry}
\usepackage{fancyhdr}
\usepackage{nccmath}

\usepackage{amsmath}
\usepackage{amssymb}
\usepackage{amsfonts}
\usepackage{amsthm}
\usepackage{amscd}
\usepackage{mathrsfs}
\usepackage{bm}
\usepackage{latexsym}
\usepackage{colonequals}

\usepackage{graphicx}
\usepackage{float}
\usepackage{tikz-cd}
\usepackage[all]{xy}
\usepackage{pstricks}

\usepackage{indentfirst}
\usepackage{color}
\usepackage{pifont}
\usepackage{aliascnt}
\usepackage{appendix}
\usepackage{todonotes}
\usepackage{ulem}
\usepackage{setspace}

\let\oldthebibliography\thebibliography
\let\endoldthebibliography\endthebibliography
\renewenvironment{thebibliography}[1]{%
    \oldthebibliography{#1}%
    \onehalfspacing  
    \setlength{\parskip}{0.2\baselineskip}
}{%
    \endoldthebibliography
}

\usepackage[CJKbookmarks=true,colorlinks,linkcolor=red,anchorcolor=blue,citecolor=blue]{hyperref}
 \usepackage[hyperpageref]{backref}
 
\usepackage{xcolor}
\hypersetup{
	colorlinks,
	linkcolor={red},
	citecolor={blue},
	urlcolor={blue}
}
\usepackage{cleveref}

\newtheorem{thm}{Theorem}[section]
\newtheorem{example}[thm]{{Example}}
\newtheorem{corollary}[thm]{{Corollary}}
\newtheorem{lemma}[thm]{{Lemma}}
\newtheorem{definition}[thm]{{Definition}}
\newtheorem{remark}[thm]{{Remark}}

\newtheorem{proposition}[thm]{Proposition}

\newtheorem{question}[thm]{{Question}}

\theoremstyle{definition}

\newtheorem*{acknowledgement*}{Acknowledgements}
\newtheorem*{theorem*}{Theorem}
\theoremstyle{remark}

\newcommand{\p}{\partial}
\newcommand{\C}{\mathbb C}
\newcommand{\A}{\mathcal A}

\newcommand{\End}{\operatorname{End}}
\newcommand{\cc}{\frac{1}{2}}
\newcommand{\st}{{\star_{h_0}}}
\newcommand{\E}{{\mathcal E}}
\newcommand{\Ch}{D_{h_0}}
\newcommand{\bg}{g^{\star_{h_0}}}
\newcommand{\id}{\operatorname{id}}
\newcommand{\ob}{\operatorname{ob}}
\newcommand{\barHb}{{\mathbb H^1(\overline{X_0}, (\End \overline E^\vee,\operatorname{ad}(\theta^\st)))}}
\DeclareMathOperator{\KS}{KS}
\DeclareMathOperator{\Dol}{Dol}
\DeclareMathOperator{\dR}{dR}
\newcommand{\NHC}{\mathrm{NHC}}

\newcommand{\Betti}{\mathrm B}

\newcommand{\Ext}{\operatorname{Ext}}

\title[]{Higher order isomonodromic deformation of Higgs bundles and a characterization of the non-abelian Noether-Lefschetz locus}

\subjclass[2010]{14D22,14C30}
\keywords{}
\author[Tianzhi Hu]{Tianzhi Hu}
\address{ School of Mathematics and Statistics, Wuhan University, Luojiashan, Wuchang, Wuhan, Hubei, 430072, P.R. China}
\email{hutianzhi@whu.edu.cn}

\author{Ruiran Sun}
 \address{School of Mathematical Sciences, Xiamen University, Xiamen 361005, China}
\email{ruiransun@xmu.edu.cn}

\author{Jinbang Yang}
\address{School of Mathematical Sciences, University of Science and Technology of China, Hefei, Anhui 230026, PR China}
\email{yjb@mail.ustc.edu.cn}

\author[Kang Zuo]{Kang Zuo}
\address{ School of Mathematics and Statistics, Wuhan University, Luojiashan, Wuchang, Wuhan, Hubei, 430072, P.R. China; Institut f\"ur Mathematik, Universit\"at Mainz, Mainz, Germany, 55099}
\email{zuok@uni-mainz.de}
\begin{document}

\begin{abstract}The purpose of this paper is to establish a local theory of the non-abelian Noether–\\Lefschetz locus.  Given a family of projective manifolds over a complex variety $S$, the isomonodromic deformation of the initial $\C$-PVHS defines a holomorphic family of flat bundles and defines a real analytic family of Higgs bundles by the non-abelian Hodge correspondence. The non-abelian Noether–Lefschetz locus exactly consists of those points in $S$ on which the isomonodromic deformed Higgs bundle underlies a graded structure. Esnault-Kerz ask whether the non-abelian Noether–Lefschetz locus is precisely the maximal complex analytic subvariety on which the real analytic isomonodromic deformation of Higgs bundles becomes holomorphic.

Our main result gives an affirmative answer to this question. The proof is based on the deformation equation of the harmonic metric solved by the non-abelian Hodge correspondence, and we use it to study higher order deformation class of the isomonodromic deformation of a graded Higgs bundle, which is expressed in terms of the differential graded Lie algebra of the joint real analytic deformation. We introduce a sequence of obstruction classes measuring the failure of holomorphicity and show that their vanishing forces the graded structure to lift to arbitrary finite order. This yields a local characterization of the non-abelian Noether–Lefschetz locus in terms of the holomorphicity of the isomonodromic deformation of Higgs bundles.
\end{abstract}
\maketitle

\setcounter{tocdepth}{1}
\hypersetup{linkcolor=black}
\tableofcontents
\hypersetup{linkcolor=red}

\section{Introduction} \label{sec_intro}

The classical Noether--Lefschetz theorem concerns the variation of algebraic
cycles in a family of smooth projective varieties.  For instance, if
$\mathcal X\to B$ is a family of smooth hypersurfaces in $\mathbb P^3$, the
Noether--Lefschetz locus consists of those points $t\in B$ for which the
Picard rank of $X_t$ jumps.  Equivalently, it is the locus where some
flat cohomology class remains of Hodge type $(1,1)$.  More generally, for a
polarized $\mathbb Q$-variation of Hodge structures
$(\mathbb V,F^\bullet,\nabla,Q)$ of even weight $2k$ over $B$, and for a flat (multi-valued) 
section $\gamma$ with $\gamma(0)\in F^k\mathcal V:=F^k(\mathbb V\otimes\mathcal O_B)$, the Hodge locus is
\[
  \mathrm{HL}_\gamma
  :=
  \{\, t\in B \mid \gamma(t)\in F^k\mathcal V \,\}.
\]

We have the following two different viewpoints to study the Hodge locus:
\begin{enumerate}
\item The first one is a {\it global criterion}, known as Deligne-Griffiths fixed part theorem: let $\gamma$ be a flat section over $B$ with $\gamma(0)\in F^k\mathcal V$ and suppose $B$ is quasi-projective and the monodromy orbit of $\gamma$ is finite, then $\gamma$ is always of type $(k,k)$.
\item The second one is a {\it local study}, stating  that its Zariski tangent space is controlled by the
Higgs field: (cf. \cite[Section 5.3.2.]{VoisinII})
\[
  T^{\mathrm{Zar}}_0 \mathrm{HL}_\gamma
  =
  \ker\bigl((\theta\circ \gamma)(0):T_0B\to E\bigr),
\]
where $(E:=\operatorname{gr}_F \mathcal V,\theta:=\operatorname{gr}_F\nabla)$ is the associated Higgs bundle.  
\end{enumerate}
The non-abelian analogue of the first one has been extensively studied in the literature and will be discussed later. The non-abelian analogue of the second one is the starting point of our research.
\medskip

Let
\[
  f:X\longrightarrow S
\]
be a smooth proper family of smooth projective varieties, and fix a point
$0\in S$.  Following Simpson \cite{Simp94I,Simp94II}, one has the relative
de Rham and Dolbeault moduli spaces
\[
\begin{tikzcd}
M_{\dR}(X/S) \arrow[rd] & & M_{\Dol}(X/S) \arrow[ld] \\
& S &
\end{tikzcd}
\]
whose fibers over $s\in S$ are respectively the moduli spaces of flat bundles
and Higgs bundles on $X_s$.  The non-abelian Hodge correspondence, developed primarily by Donaldson \cite{Don} and Uhlenbeck-Yau \cite{UY} in the vector bundle setting and by Hitchin \cite{Hit}, Corlette \cite{Corl} and Simpson \cite{Simp88,Simp92} in the Higgs bundle setting,  gives a
{\it real analytic} isomorphism
\[
  \NHC:M_{\dR}(X/S)\xrightarrow{\ \sim\ } M_{\Dol}(X/S)
\]
covering the identity of $S$.  We shall also use that this correspondence is
real analytic in families; see Theorem~\ref{R-analyticity}.

\medskip
Now let $(\mathbb V,\mathcal F^\bullet,\nabla,Q)$ be a polarized
$\mathbb C$-variation of Hodge structures on the central fiber $X_0$ over $0\in S$. 
For every nearby fiber $X_s$, the underlying flat bundle
$(\mathbb V\otimes \mathcal O_{X_0},\nabla)$ admits an {\it isomonodromic deformation}
$(\mathbb V\otimes \mathcal O_{X_s},\nabla_s)$.  Equivalently, this yields a holomorphic section
\[
  \sigma_{\dR}:S\longrightarrow M_{\dR}(X/S).
\]

Simpson introduced the corresponding {\it non-abelian Noether--Lefschetz locus}
\begin{align}\label{eq_NLlocus}
  \mathcal{NL}
  :=
  \bigl\{
    s\in S
    \mid
    (\mathbb V\otimes \mathcal O_{X_s},\nabla_s)
    \text{ underlies a polarized } \mathbb C\text{-VHS}
  \bigr\}.
\end{align}

\medskip
The non-abelian Deligne's fixed part theorem, proved in \cite{JostZuo,KP,EK}, gives a {\it global geometric criterion} that the isomonodromic deformation of a $\C$-PVHS again underlies a $\C$-PVHS. More precisely, when the base is quasi-projective and the monodromy orbit of the isomonodromic deformation is finite, the isomonodromic deformed local system always underlies a $\C$-PVHS. 

\medskip
We now turn to the {\it local study} of $\mathcal{NL}$. Composing $\sigma_{\dR}$ with the relative non-abelian Hodge correspondence gives a
real analytic section
\[
  \sigma_{\Dol}:=\NHC\circ \sigma_{\dR}
  :
  S\longrightarrow M_{\Dol}(X/S).
\]
If $(E,\theta)$ is the graded Higgs bundle associated with the initial
variation of Hodge structures on $X_0$, then $\sigma_{\Dol}(0)=[(E,\theta)]$.
By non-abelian Hodge theory, a flat bundle underlies a polarized
$\mathbb C$-VHS precisely when the corresponding Higgs bundle is graded, or
equivalently fixed by the natural $\mathbb C^*$-action on the Dolbeault
moduli space \cite{Simp92,Simp97,Simp10}.  Hence, if we let
\[
  \mathcal{GR}
  :=
  \bigl\{
    s\in S
    \mid
    \sigma_{\Dol}(s)
    \text{ is represented by a graded Higgs bundle}
  \bigr\},
\]
then
\[
  \mathcal{NL}=\mathcal{GR}.
\]
Thus $\mathcal{NL}$ can be studied purely in terms of Higgs bundles.
Simpson proved that $\mathcal{NL}$ is a complex analytic subvariety of $S$
and that the restriction
\[
  \sigma_{\Dol}|_{\mathcal{NL}}
  :
  \mathcal{NL}\longrightarrow M_{\Dol}(X/S)
\]
is holomorphic \cite[Theorem 12.1]{Simp97}.  Esnault and Kerz asked whether this property
characterizes the non-abelian Noether--Lefschetz locus:

\begin{question}[Esnault--Kerz]\label{EsnKer}
Let $(\mathbb V,\mathcal F^\bullet,\nabla,Q)$ be a polarized
$\mathbb C$-VHS on $X_0$, and let $U\subset S$ be a closed complex analytic
subvariety passing through $0$.  Suppose that
\[
  \sigma_{\Dol}|_U:U\longrightarrow M_{\Dol}(X/S)
\]
is holomorphic.  Must one have
\[
  U\subset \mathcal{NL}?
\]
Equivalently, is $\mathcal{NL}$ the maximal complex analytic subvariety of
$S$ on which the real analytic section $\sigma_{\Dol}$ becomes holomorphic?
\end{question}

It is worth noting that Question~\ref{EsnKer}, stated in terms of holomorphicity, is a {\it local characterization} of the non-abelian Noether-Lefschetz locus. 
The purpose of this paper is to give an affirmative answer to Question~\ref{EsnKer}.  Our main
theorem is the following.

\begin{thm}\label{thm_main}
Let $(\mathbb V,\mathcal F^\bullet,\nabla,Q)$ be a polarized
$\mathbb C$-VHS on $X_0$.  Let $U\subset S$ be a closed complex analytic
subvariety such that
\[
  \sigma_{\Dol}|_U:U\longrightarrow M_{\Dol}(X/S)
\]
is holomorphic.  Then
\[
  U\subset \mathcal{NL}.
\]
\end{thm}

The first-order case already illustrates the mechanism.  Let
\[
  \tau_0:T_0S\longrightarrow H^1(X_0,T_{X_0})
\]
be the Kodaira--Spencer map of the family $X\to S$.  The Higgs field
$\theta$ induces a morphism of deformation complexes, denoted as $\theta:(T_{X_0},0)\to (\End\E,\operatorname{ad}(\theta))$, hence a map
\begin{align}\label{eq_nabHiggs}
  \theta_{*}:H^1(X_0,T_{X_0})
  \longrightarrow
  \mathbb{H}^1\bigl(X_0,(\End E,\operatorname{ad}(\theta))\bigr).
\end{align}
The non-abelian analogue of the classical Zariski tangent space formula for Hodge
loci is
\[
  T^{\mathrm{Zar}}_0\mathcal{NL}
  =
  \ker
  \Bigl(
    \theta_{*}\circ \tau_0:
    T_0S
    \longrightarrow
    \mathbb{H}^1\bigl(X_0,(\End E,\operatorname{ad}(\theta))\bigr)
  \Bigr),
\]
which is proved in Theorem~\ref{thm_Zar_NL}.

Indeed, as observed in \cite{HSZ,CTW}, the section  $\sigma_{\Dol}$ is real analytic rather than holomorphic. We have the following result on the first order derivative of $\sigma_{\Dol}$
\begin{theorem*}[{\cite[Theorem A]{HSZ}}]
For $v\in T_0S,$ the
$(1,0)$-part derivative $\Pi^{1,0}\sigma_{\Dol,*}(v)$ can be expressed explicitly in terms of the first variation of
the harmonic metric, while the $(0,1)$-part derivative $\Pi^{0,1}\sigma_{\Dol,*}(v)$ is exactly the complex conjugation of 
$\theta_{*}\circ \tau_0(v)$ (see also Proposition~\ref{prop_firstob}). 
\end{theorem*}
  It turns out that first-order holomorphicity of
$\sigma_{\Dol}$ along a tangent vector is equivalent to first-order
liftability of the graded structure. Thus Theorem~\ref{thm_main} is a higher-order generalization of this equivalence. 

 \medskip
 The first difficulty in proving Theorem~\ref{thm_main} is that higher-order holomorphicity of a real analytic map involves many conditions, such as the vanishing of derivatives in the directions \(t\bar t\), \(\bar t^2\), and so on. We divide all non-holomorphic derivatives into two types:
\begin{enumerate}
    \item purely anti-holomorphic derivatives, i.e., those with respect to \(\bar t, \bar t^2, \cdots\);
    \item mixed derivatives, i.e., those with respect to \(t\bar t, t\bar t^2, t^2\bar t, \cdots\).
\end{enumerate}
Rather than attempting to handle all these derivatives simultaneously, we isolate the purely anti-holomorphic ones. It is crucial that, under the setup of Question~\ref{EsnKer}, these purely anti-holomorphic derivatives are essentially related to the liftability of the initial graded structure. This leads to a sequence of partial obstruction classes whose vanishing detects whether the isomonodromic deformation becomes increasingly holomorphic in the \(\bar t\)-directions.

\medskip
We now describe the proof of Theorem~\ref{thm_main}. The first step is the truncation argument. Let
\[
  A_n=\mathbb C[t]/(t^{n+1})
\]
and let $X_n\to \operatorname{Spec}A_n$ be an $n$-th order deformation of
$X_0$ obtained by restricting the family $X\to S$ to an $n$-jet in the base.
We regard $X_n$ as a deformation of the complex structure on the fixed
underlying differentiable manifold of $X_0$.  

Let $(E,\theta)$ a stable graded Higgs bundle  on $X_0$ and let $\E$ be the smooth model of the holomorphic bundle $E$.  We may use the Dolbeault operator $\bar\p$ of $\E$ on $X_0$ to represent the complex structure of $E$. Thus we have $(\E,\bar\partial,\theta)=(E,\theta)$. 
By truncating its isomonodromic deformation $\sigma_{\Dol}$, we have a real analytic deformation on $X_n$, denoted by 
\[
  (\E,\bar\partial_t,\theta_t)
\]
where $\bar\partial_t$ defines a family of complex structures on $\E$ over $X_n$ and
$\theta_t$ is a family of Higgs fields, both depending on 
$t$ and $\bar t$ (Definition~\ref{def_real_ana_deform}).  Such a deformation is holomorphic if, it depends only on $t$ (Definition~\ref{def_holo_deform}).  

Secondly, to select partial obstructions of holomorphicity, we introduce the following notion. For an integer $k=1,2,\cdots,n$, and an ideal
$(t,\bar t^{k+1})\subset \mathbb C[t,\bar t]/(t,\bar t)^{n+1}$, we say that the
deformation is \textit{modulo-$(t,\bar t^{k+1})$-holomorphic} if it depends only on $A_n+(t,\bar t^{k+1})$, i.e. after modulo the ideal $(t,\bar t^{k+1})$, it depends only on $t$ (Definition~\ref{def_modulo_hol}). If $(\E,\bar\p_t,\theta_t)$ is modulo-$(t,\bar t^k)$-holomorphic,
we can find an obstruction class
\[
  \ob_k
  \in
  \mathbb H^1\bigl(
    \overline{X_0},
    (\End \overline E^\vee,\operatorname{ad}(\theta^{\st}))
  \bigr)
\]
whose vanishing is equivalent to modulo-$(t,\bar t^{k+1})$-holomorphicity (proved in Proposition~\ref{prop_existence_ob}).
Consequently, the successive vanishing of
\[
  \ob_1,\ob_2,\ldots,\ob_n
\]
is equivalent to the modulo-$(t)$-holomorphicity, i.e. the disappearance of all pure anti-holomorphic derivatives up to
order $n$.  By definition, $\ob_1$ is exactly given by $\theta_{*}\circ \tau_0$.

 We explain the relationship between the vanishing of $\ob_1,\ob_2,\ldots,\ob_n$ and the liftability of the initial graded structure, which is the main step in our proof. Let
$h_t$ be the harmonic metric of the deformed Higgs bundle $ (\E,\bar\partial_t,\theta_t)$.  After choosing
the fixed smooth bundle $\E$, we may write
\[
  h_t
  =
  h_0\left(
    \id
    +\sum_{i=1}^n t^i g_i
    +\sum_{i=1}^n \bar t^i g_i^{\st}
    +\text{mixed terms}
  \right),
\]
where $h_0$ is the Hodge metric of the initial graded Higgs bundle and
$g_i\in \A^0(\End \E)$. These $g_i$ can be found by solving the harmonic metric equation for the isomonodromic deformation. The harmonic-map formula
\[
  \theta_t
  =
  \left(
    -\frac12 h_t^{-1}dh_t
  \right)^{1,0}
\]
shows that the Taylor coefficients $t,t^2,\cdots,t^n$ and $\bar t,\bar t^2,\cdots,\bar t^n$ of the deformed Higgs field $\theta_t$ and
Dolbeault operator $\bar\p_t$ are determined by the coefficients $g_i$ and
$g_i^{\st}$ of the metric (for detailed expressions, see Proposition~\ref{prop_Taylor_Exp_DolHig}).  In particular, the holomorphic deformation is
controlled by the $g_i$, while the obstruction classes $\ob_1,\ldots,\ob_n$ are controlled by the
adjoint coefficients $g_i^{\st}$.  Since all of them are determined by the same endomorphisms $g_i$ (or their adjoints), the vanishing of these obstruction classes
imposes certain restrictions on the holomorphic part of the deformation, which may help to confirm Question~\ref{EsnKer}.

A technical issue arises in our proof: we need to express $\ob_1,\ldots,\ob_n$ in terms of $\bg_i$, but there is no clean algorithm to do so. Instead, we employ a gauge transformation. We remark that any two deformations differ by a gauge transformation should be viewed as a same one. With this in mind, the proof of Theorem~\ref{thm_main} proceeds in the following steps:
\begin{enumerate}
\item[(i)] The condition of modulo-$(t)$-holomorphicity---equivalently, the vanishing of $\ob_1,\ldots,\ob_n$---can be expressed directly via $n$ gauge equations \eqref{hol} together with the Taylor expansion Proposition~\ref{prop_Taylor_Exp_DolHig} obtained from the deformed harmonic metric. 
\item[(ii)] We prove that the above gauge equations impose concrete and complete restrictions on $g_1,\ldots,g_n$ in the harmonic metric $h_t$: in summary, Proposition~\ref{wtprop}, Proposition~\ref{prop_ob} and Proposition~\ref{1,-1}. 
\item[(iii)] Using these restrictions, along with the initial graded decomposition of $\End \E$, we choose a gauge transformation \eqref{eq_gauge_gr} and \eqref{thmv} in section~\ref{sec_proof_main} such that the transformed Dolbeault operator and Higgs field preserve the graded structure up to order $n$. This establishes that the initial graded structure lifts to $X_n$.
\end{enumerate}
In summary, we prove the following key truncated statement.

\begin{thm}[Truncated version]\label{thm_main_tru}
Let $(\E,\bar\partial,\theta)$ be a graded stable Higgs bundle on $X_0$.
Let
\[
  (\E,\bar\partial_t,\theta_t)
\]
be the isomonodromic deformation of $(\E,\bar\partial,\theta)$ over an
$n$-th order deformation $X_n$ of $X_0$.  If this real analytic deformation
is holomorphic, then it is graded.
\end{thm}
To deduce Theorem~\ref{thm_main} from Theorem~\ref{thm_main_tru}, we work
locally on a resolution of the analytic subvariety $U$.  If
$\sigma_{\Dol}|_U$ is holomorphic, then every formal arc in $U$ gives a
holomorphic truncated isomonodromic deformation.  By
Theorem~\ref{thm_main_tru}, the graded structure lifts along every such
truncation.  Hence the Higgs bundles parametrized by $\sigma_{\Dol}|_U$ are
graded, so $U\subset \mathcal{GR}=\mathcal{NL}$.  This proves the
Esnault--Kerz characterization of the non-abelian Noether--Lefschetz locus.

\bigskip

A related question is: when \(X/S\) is the universal curve \(\mathcal{C}/\mathcal{T}_g\) over the Teichm\"uller space, can \(\sigma_{\mathrm{Dol}}\) be holomorphic over the entire \(\mathcal{T}_g\)? For rank‑2 and rank‑3 non‑unitary Higgs bundles, the answer is negative, as shown in \cite{biswas,HSZII}. However, for non‑unitary Higgs bundles of high rank, \cite{biswas} provides examples where \(\sigma_{\mathrm{Dol}}\) is a family of graded Higgs bundles and is indeed holomorphic over the whole \(\mathcal{T}_g\). 

\bigskip

Furthermore, we point out that the above research method for local deformation theory—namely, characterizing local deformations of complex structures on $X_0$ via the Kodaira-Spencer differential graded Lie algebra and employing the harmonic theory of the Hodge correspondence—has also been applied to study the classical Hodge locus; see \cite{LiuShen26II}.
\bigskip

\begin{acknowledgement*}
We are deeply grateful to Hélène Esnault and Moritz Kerz for communicating their conjecture (Question~\ref{EsnKer}) to us and for suggesting that the obstruction theoretic approach using higher-order deformations could be the right tool to characterise the non-abelian Noether--Lefschetz locus. Their insights and encouragement have been instrumental to the development of this work. We also express our gratitude to Sebastian Heller, Lin Weng and Shing-Tung Yau for their interest in this work, their valuable comments and suggestions. We are grateful to Runze Zhang for discussions on higher-order deformation theory and for providing the reference \cite{Ono}.
\end{acknowledgement*}

\section*{Notations}

\begin{itemize}
\item Unless otherwise stated, all indices $i$ (including $i_1, i_2, \dots$) appearing in this paper are positive integers and we define $|I_N|:=i_1+i_2+\cdots+i_N$.

\item Let $(\E,\bar\p,\theta)$ be a polystable Higgs bundle on $X_0$ carrying a harmonic metric $h_0$ and $\Ch=\Ch^{1,0}+\bar\p$ be the associated Chern connection. For any $g\in\A^0(\End\E)$, we let $\bg:=h_0^{-1}\bar g^T h_0$.

\item  Let $A_n = \mathbb{C}[t]/(t^{n+1})$ and $B_n = \mathbb{C}[t, \bar{t}]/(t^{n+1}, t^n\bar{t}, \dots, \bar{t}^{n+1})$ be the Artin rings of truncated holomorphic and real analytic functions on the complex plane $\mathbb{C}$ at the origin, with maximal ideals $(t)$ and $(t, \bar{t})$, respectively.

\end{itemize}

\newpage

\section{Isomonodromic deformation of a stable Higgs bundle with trivial Chern classes}
\label{sec_isom_defm_Higgs}

In this section, we investigate the isomonodromic deformations of a stable Higgs bundle with trivial Chern classes defined on a smooth projective variety. We adopt the framework of infinitesimal deformations over Artin rings and derive explicit Taylor series expansions for the deformed Dolbeault operator, Higgs field, and associated harmonic metric. The principal result of this section is Proposition~\ref{prop_Taylor_Exp_DolHig}, which demonstrate that the coefficients occurring in these expansions depend solely on the initial Higgs bundle and the infinitesimal deformations of the underlying smooth projective variety. These results will play a fundamental role in the following sections.

\medskip 
\subsection{Deformation theory of a smooth projective variety} \label{sec_def_projvar}

Let $X_0$ be a smooth projective variety. Recall that $A_n = \mathbb{C}[t]/(t^{n+1})$ is the Artin ring of truncated holomorphic on the complex plane at the origin with the maximal ideal $\mathfrak m:=(t)$. Let $$D_{X_0}(A_n)=\{\text{deformations } X_n\to\operatorname{Spec}A_n \text{ of }X_0 \}/\sim$$ 
be the set of isomorphic classes of deformations $X_n\to\operatorname{Spec}A_n$ of $X_0$ to $\operatorname{Spec}A_n$. 
Using the theory of \textbf{differential graded Lie algebra} (see \cite[Remark 14.8-14.9]{Gro}), we view a deformation $X_n\to \operatorname{Spec} A_n$ as a family of complex structures on \textbf{a fixed differential manifold} $X_0$ (forgetting the initial complex structure of $X_0$). This gives 
\begin{align*}D_{X_0}(A_n)\cong\frac{\{\eta\in \A^{0,1}(T_{X_0})\otimes\mathfrak m\mid \bar\p_{T_{X_0}}\eta+\cc[\eta,\eta]=0\}}{\text{gauge equivalence}}.\end{align*}
Precisely, for any $\eta=\sum_{i=1}^n \eta_i t^i\in \A^{0,1}(T_{X_0})\otimes \mathfrak m$ satisfying the above integrability condition, the corresponding deformation is the ringed space $X_n = (X_0^{\rm Top}, \mathcal{O}_{X_n})$ over $\operatorname{Spec} A_n$, where $\mathcal{O}_{X_n} \subset \mathcal{C}^\infty(X_0) \otimes A_n$ is the subsheaf of functions annihilated by the operator $\bar{\partial}_{X_0} + \eta \circ \partial_{X_0}$. The $A_n$-algebra structure on $\mathcal{O}_{X_n}$ induces the structural morphism $X_n \to \operatorname{Spec} A_n$.

Let $B_n = \mathbb{C}[t, \bar{t}]/(t,\bar{t})^{n+1}$ be  the Artin ring of $n$-th order truncated real analytic function germs at the origin of the complex plane.
Let $\mathcal C^\infty(X_n):=\mathcal C^\infty(X_0)\otimes B_n$ be the \textbf{sheaf of smooth functions} on $X_n$. We define the \textbf{holomorphic cotangent bundle} $\Omega^{1}(X_n/A_n)$ as the locally free sheaf of $\mathcal O_{X_n}$-modules locally generated by $df$ for any $f$ being a local holomorphic function of $X_n$. Let $\Omega^{1,0}(X_n/A_n):=\Omega^{1}(X_n/A_n)\otimes_{\mathcal O_{X_n}}\mathcal C^\infty(X_n) $ be the \textbf{smooth $(1,0)$ cotangent bundle}, which is a subsheaf of the \textbf{smooth cotangent bundle} $\mathbb CT_{X_n}^*:=\A^1(X_0)\otimes B_n$. The \textbf{anti-holomorphic cotangent bundle} $\Omega^{0,1}(X_n/A_n)$ is defined to be the complex conjugation of $\Omega^{1,0}(X_n/A_n)\subset \mathbb CT^*_{X_n}$.

\subsection{Deformation of a Higgs bundle and the gauge theory} \label{sec_def_Higg_and_gauge}

Let \((E,\theta)=(\E,\bar\p,\theta)\) be a stable Higgs bundle (we always assume that such bundle have trivial Chern classes) on \(X_0\), where \(\E\) denotes the underlying smooth vector bundle obtained by forgetting the holomorphic structure of \(E\). To study the deformation theory of the triple \((X_0,E,\theta)\), we fix \textbf{the smooth model} \((X_0,\E)\) and equip it with a family of complex structures \((\eta,\bar\p_t)\) together with a family of Higgs fields \(\theta_t\), subject to the following definition.

\begin{definition} \label{def_real_ana_deform} 
For any order $n$ deformation of $X_0$ denoted by $X_n\in D_{X_0}(A_n)$, we define a \textbf{real analytic deformation} of the initial stable Higgs bundle $(\E,\bar\p,\theta)$ on $X_0$ to $X_n$ as a triple $(\E,\bar\p_t,\theta_t)$ with
\begin{align*}
\bar\p_t:&\E\to\E\otimes_{\mathcal C^\infty(X_0)}\Omega^{0,1}(X_n/A_n);\\
\theta_t:&\E\to\E\otimes_{\mathcal C^\infty(X_0)}\Omega^{1,0}(X_n/A_n),
\end{align*}
satisfying the following  conditions:
\begin{enumerate}
\item $\bar\p_t$ is $\mathbb C$-linear and satisfies the Leibniz rule as a $(0,1)$ connection; $\theta_t$ is $\mathcal C^\infty(X_n)$-linear;
\item Modulo $t,\bar t$, the deformation triple reduces to the initial Higgs bundle, i.e. $\bar\p_t\equiv\bar\p$ and $\theta_t\equiv\theta$;
\item $(\bar\p_t,\theta_t)$ satisfies the integrable conditions
\begin{align}\label{eq_compatibility}
\bar\p_t^2=0;\quad\theta_t\wedge\theta_t=0;\quad \bar\p_t\theta_t=0.
\end{align}
\end{enumerate}
\end{definition}

\begin{example}\label{eg_isomo}
Let $(\E,\bar\p_s,\theta_s)$ be the isomonodromic deformation of the initial Higgs bundle $(\E,\bar\p,\theta)$ on $X_0$ to the family $X/S$. By \cite[Theorem 4.23]{CTW} when the fibers of $X/S$ are compact Riemann surfaces, and by Theorem~\ref{R-analyticity} in the general case, $(\E,\bar\p_s,\theta_s)$ is a real analytic deformation of Higgs bundles. We consider any order-$n$ germ of $S$ at $0$, i.e. a morphism $\gamma:\operatorname{Spec}A_n\to S$ mapping $0\in\operatorname{Spec}A_n$ to $0\in S$.
The pull-back of $X/S$ via $\gamma:\operatorname{Spec}A_n\to S$ gives an $X_n\in D_{X_0}(A_n)$. The pull-back of $(\E,\bar\p_s,\theta_s)$ via $\gamma$ gives a real analytic deformation of Higgs bundles.
\end{example}

We can explicitly expand the deformed operators in Definition~\ref{def_real_ana_deform} in terms of the deformation parameters $t,\ \bar t$, as in the following lemma.

\begin{lemma}\label{lem_expanding} 
Let $\eta=\sum\limits_{i=1}^n \eta_i t^i\in \A^{0,1}(T_{X_0})\otimes \mathfrak m$ represent $X_n$. Then there exist $\alpha_i,\beta_i,\varphi_i,\psi_i\in \A^1(\End \E)$ such that \begin{align*}\bar\p_t&=\bar\p-\sum_{i=1}^n\bar t^i\bar\eta_i\circ\bar\p+\sum_{i=1}^n t^i\eta_i\circ\Ch^{1,0}+\sum_{i=1}^nt^i\beta_i+\sum_{i=1}^n\bar t^i\psi_i\pmod{t\bar t};\\
\theta_t&=\theta+\sum_{i=1}^nt^i\alpha_i+\sum_{i=1}^n\bar t^i\varphi_i\pmod{t\bar t},
\end{align*}
where $h_0$ is the harmonic metric of the initial Higgs bundle $(\E,\bar\p,\theta)$ on $X_0$ and $\Ch=\Ch^{1,0}+\bar\p$ is the Chern connection of $(\E,\bar\p,h_0)$.
\end{lemma}
\begin{proof}We consider the $(0,1)$-part of $\Ch$ with respect to the complex structure of $X_n$, denoted by $\pi''_\eta\Ch$ and one may verify directly as operators that  \begin{align}\label{eq_piCh}
\pi''_\eta\Ch\equiv \bar\p-\sum_{i=1}^n\bar t^i\bar\eta_i\circ\bar\p+\sum_{i=1}^n t^i\eta_i\circ\Ch^{1,0}\pmod{t\bar t}.
\end{align}
Then $\bar\p_t-\pi''_\eta\Ch$ is a section of $\End\E\otimes_{\mathcal C^\infty(X_0)} \Omega^{0,1}(X_n/A_n)$. 
This proves our claim.
\end{proof}

\begin{definition} \label{def_holo_deform}
Let $(\E,\bar\p,\theta)$ be a Higgs bundle on $X_0$. 
A \textbf{(holomorphic) deformation} of $(\E,\bar\p,\theta)$ is a relative Higgs bundle $(\E,\bar\p_t,\theta_t)$ over $X_n/A_n$ with central fiber $(\E,\bar\p,\theta)$. 
There is a natural forgetful functor from the category of holomorphic deformations to the category of real analytic deformations, induced by extending the coefficient sheaf from the holomorphic to the smooth setting. 
A real analytic deformation is said to be \textbf{holomorphic} if it lies in the essential image of this functor.
\end{definition}

In the following, we give a criterion of holomorphicity of real analytic deformation. Before that we need to introduce some notations. 
Let $\pi'_{\eta}:\C T^*_{X_n}\to \Omega^{1,0}(X_n/A_n)$ and $\pi''_{\eta}:\C T^*_{X_n}\to \Omega^{0,1}(X_n/A_n)$ be two natural projections. By the trivial extension as smooth forms, we have $\Omega^{1,0}(X_0)\hookrightarrow \C T^*_{X_n}$ and $\Omega^{0,1}(X_0)\hookrightarrow \C T^*_{X_n}$. We define 
\begin{align*}P_\eta':=&\pi'_{\eta}|_{\Omega^{1,0}(X_0)}:\Omega^{1,0}(X_0)\to \Omega^{1,0}(X_n/A_n);\\ P_\eta'':=&\pi''_{\eta}|_{\Omega^{0,1}(X_0)}:\Omega^{0,1}(X_0)\to \Omega^{0,1}(X_n/A_n)
\end{align*}
For any $\alpha\in \Omega^{1,0}(X_0)$, by \cite[P75]{Gro}, we have
\begin{align}\label{eq_Peta}
P_\eta'(\alpha)=\alpha-\sum_{i=1}^n t^i\eta_i(\alpha) \quad \text{and}\quad P_\eta''(\bar\alpha)=\bar\alpha-\sum_{i=1}^n \bar t^i\bar\eta_i(\bar\alpha),
\end{align}
where $\eta(\alpha)\in\Omega^{0,1}(X_0)$ is the contraction. 
We have the following canonical isomorphisms of $\mathcal C^\infty(X_0)$-modules:
\begin{align*}  
B_n \otimes P'_\eta \Omega^{1,0}(X_0)\cong \Omega^{1,0}(X_n/A_n) 
 \quad \text{and}\quad
B_n \otimes P''_\eta  \Omega^{0,1}(X_0) \cong  \Omega^{0,1}(X_n/A_n).
\end{align*}
Henceforth, we will always identify these sheaves via these canonical isomorphisms. Now for a real analytic deformation $(\E,\bar\p_t,\theta_t)$ in Definition~\ref{def_real_ana_deform}, one can view $\bar\p_t-\pi_\eta''\Ch$ as an element in $B_n \otimes P''_\eta  \Omega^{0,1}(X_0) \otimes_{\mathcal C^\infty(X_0)} \End\E$ and $\theta_t$ as an element in $B_n \otimes P'_\eta \Omega^{1,0}(X_0) \otimes_{\mathcal C^\infty(X_0)} \End\E$. 

\begin{proposition} \label{prop_criterion_holo_deform}
 The real analytic deformation $(\E,\bar\p_t,\theta_t)$
 is holomorphic if and only if there exists a gauge transformation $\mathscr U$ such that 
\begin{equation}\label{eq_hol}
\begin{aligned}
 \mathscr U^{-1} \circ \bar\p_t \circ \mathscr U - \pi_\eta''\Ch & \in  A_n \otimes P''_\eta  \Omega^{0,1}(X_0) \otimes_{\mathcal C^\infty(X_0)} \End(\E);
 \\  
\mathscr U^{-1}\circ\theta_t \circ \mathscr U & \in A_n \otimes P'_\eta \Omega^{1,0}(X_0) \otimes_{\mathcal C^\infty(X_0)} \End(\E).
\end{aligned}
\end{equation}

\end{proposition}

\begin{proof}
By \cite{Ono}, the condition \eqref{eq_hol} is equivalent to the holomorphicity of $(\E,\mathscr U^{-1} \circ \bar\p_t \circ \mathscr U ,\mathscr U^{-1} \circ \theta_t \circ \mathscr U )$. Therefore $(\E,\bar\p_t,\theta_t)$ is also a holomorphic deformation because it differs with $(\E,\mathscr U^{-1} \circ \bar\p_t \circ \mathscr U ,\mathscr U^{-1} \circ \theta_t \circ \mathscr U )$ by a gauge equivalence. 
\end{proof}

\begin{definition} \label{def_modulo_hol}
For the ideal $(t,\bar t^{k+1})$ of $B_n$, a real analytic deformation $(\E,\bar\p_t,\theta_t)$
 is call \textbf{modulo-$(t,\bar t^{k+1})$-holomorphic} if there exists a gauge transformation $\mathscr U$ such that 
\begin{equation}\label{eq_hol_b}
\begin{aligned}
 \mathscr U^{-1} \circ \bar\p_t \circ \mathscr U - \pi_\eta''\Ch & \in  (A_n+(t,\bar t^{k+1})) \otimes P''_\eta  \Omega^{0,1}(X_0) \otimes_{\mathcal C^\infty(X_0)} \End(\E);
 \\  
\mathscr U^{-1}\circ\theta_t \circ \mathscr U & \in (A_n+(t,\bar t^{k+1})) \otimes P'_\eta \Omega^{1,0}(X_0) \otimes_{\mathcal C^\infty(X_0)} \End(\E).
\end{aligned}
\end{equation}

\end{definition}

\begin{remark}
\begin{enumerate}
    \item Let $(t,\bar t^{k+1}) \subset (t,\bar t^{k})$ be two ideals of $B_n$. If $(\E,\bar\p_t,\theta_t)$  is modulo-$(t,\bar t^{k+1})$-holomorphic, then it is modulo-$(t,\bar t^{k})$-holomorphic. Moreover, all real analytic deformations are modulo-$(t,\bar t)$-holomorphic.
    \item  By replacing $\mathscr U$ with $\mathscr U \mathscr U_0^{-1}$ in \eqref{eq_hol}, we may assume the gauge transformation $\mathscr U\in\A^0(\End \E)\otimes B_n$ satisfying $\mathscr U \equiv \mathrm{id} \pmod{(t,\bar t)}$, where $\mathscr U_0$ is the constant term of $\mathscr U$. 
\end{enumerate}

\end{remark}

\bigskip

Let $(\E,\bar\p_t,\theta_t)$  be a modulo-$(t)$-holomorphic deformation and let 
\begin{equation} \label{eq_Taylor_U}
\mathscr{U} = \operatorname{id} + \sum_{i=1}^n \bar{t}^i u_i + \sum_{i=0}^{n-1} \sum_{j=1}^{n-i} t^j \bar{t}^i u_{ij}\in\A^0(\End\E)\otimes B_n
\end{equation} 
be the Taylor expansion of a gauge transformation such that 
$\mathscr U^{-1}\circ\bar\p_t\circ \mathscr U$ and $\mathscr U^{-1}\circ\theta_t\circ\mathscr U$ satisfies \eqref{eq_hol_b} for $k=n$. Then we have the following lemma describing this kind of holomorphicity as a systems of equations of $u_i$.

\begin{lemma}[Equations for gauge transformation] \label{lem_gauge} 
Let $(\E,\bar\p_t,\theta_t)$  be a real analytic deformation. Then it is modulo-$(t)$-holomorphic if and only if all terms $u_i$'s in gauge transformation $\mathscr{U}$ given in \eqref{eq_Taylor_U} satisfy the following system of equations for $m=1,\cdots,n$:
\begin{align}\label{hol}
\begin{cases}
\varphi_{m}+\sum\limits_{j=1}^{m-1}\varphi_{j}u_{m-j}+[\theta,u_{m}]=0;\\
\psi_{m}+\sum\limits_{j=1}^{m-1}\psi_{j}u_{m-j}+\bar\p u_{m}-\sum\limits_{j=1}^{m-1}\bar\eta_{j}(\bar\p u_{m-j})=0,
\end{cases}
\end{align}
where $\{\varphi_i,\psi_i\}_{i=1}^n$ are defined in Lemma~\ref{lem_expanding}.
 \end{lemma}

\begin{proof}
By Lemma~\ref{lem_expanding}, we have \begin{align*}\bar\p_t&\equiv\bar\p-\sum_{i=1}^n\bar t^i\bar\eta_i\circ\bar\p+\sum_{i=1}^n\bar t^i\psi_i\pmod{(t,\bar t^{n+1})};\\
\theta_t&\equiv\theta+\sum_{i=1}^n\bar t^i\varphi_i\pmod{(t,\bar t^{n+1})}.
\end{align*}
By the definition of $\mathscr U$ in \eqref{eq_Taylor_U}, we have \eqref{eq_hol_b} for $k=n$. Modulo $(t)$ in these two identities and using \eqref{eq_piCh}, we have 
\begin{align*}\bar{\partial}_t \circ \mathscr{U}&=\mathscr U\circ (\bar\p-\sum_{i=1}^n\bar t^i\bar\eta_i\circ\bar\p)\pmod{(t,\bar t^{n+1})};\\
\theta_t\circ\mathscr U&=\mathscr U\circ\theta\pmod{(t,\bar t^{n+1})}.
\end{align*}After expanding the above expressions and comparing the coefficient of $\bar t^k$, we have \eqref{hol}.
\end{proof}

\subsection{Harmonic metric and the isomonodromic deformation} \label{sec_HarmMet_IsomDef}

Let $(\E,\bar\p_s,\theta_s)$ be the isomonodromic deformation of the initial Higgs bundle $(\E,\bar\p,\theta)$ on $X_0$ to the family $X/S$. In this subsection, we always denote $(\E,\bar\p_t,\theta_t)$ as some truncated real analytic deformation on $X_n$ defined in Example~\ref{eg_isomo}.

Let $h_t\in A^0(\bar E^\vee\otimes E^\vee)\otimes B_n$ be truncated metric of the harmonic metric of $(\E,\bar\p_s,\theta_s)$. There exist $g_i,g_{i\bar j}$ lie in $\A^0(\End E)$ such that 
\begin{align}\label{metric}
h_t=h_0 \cdot \underbrace{\left(\operatorname{id}+\sum_{i=1}^n t^ig_i+\sum_{i=1}^n \bar t^ig_i^\st+\sum_{i=1}^{n-1}\sum_{j=1}^{n-i}t^i\bar t^j g_{i\bar j}\right)}_{g(t,\bar{t})}.
\end{align}
where $\bg_i$ is the Hodge star of $g_i$ defined in Notations. Those $g_i,g_{i\bar j}$ characterize the \textbf{deformation} of the harmonic metric. Let $g(t,\bar t)$ denote the expression in parentheses in the above equation, which is invertible as $g(t,\bar t)-\mathrm{id}$ is nilpotent. We remark $g(t,\bar t)$ must be of this form to ensure that $h_t$ is Hermitian.

\medskip

 View $h_t$ as a harmonic map defined on the universal cover $\widetilde X_n$  of $X_n$. Let $\Psi_t:=-\cc h_t^{-1}dh_t$. 
 Let $(\mathcal V,D)$ on $X_0$ be the associated flat bundle given by the nonabelian Hodge correspondence, where $D$ is the smooth flat connection.
 By a similar argument as in the proof of \cite[Lemma 9.13]{Gui}, substituting the metric \eqref{metric} into $\Psi_t=-\cc h_t^{-1}dh_t$ yields
\begin{equation}\label{PSI}\begin{aligned} 
\Psi_t =& -\cc g(t,\bar t)^{-1} \cdot (h_0^{-1}dh_0) \cdot g(t,\bar t)-\cc g(t,\bar t)^{-1} \cdot D^{\End} \Big(g(t,\bar t)\Big)\\
=&g(t,\bar t)^{-1}(\theta+\theta^\st)g(t,\bar t)-\cc g(t,\bar t)^{-1}D^{\End} \Big(g(t,\bar t)\Big).
\end{aligned}\end{equation}

By \cite[Lemma 9.13]{Gui}, the $(1,0)$-part of $\Psi_t$ with respect to $X_n$ is $\theta_t$ and the $(0,1)$-part of $\Psi_t$ with respect to $X_n$ is $\theta_t^{\star_{h_t}}$.
By comparing the $(1,0)$ and $(0,1)$ parts of both sides of  \eqref{PSI}, we obtain our main result in this section: when the deformation is isomonodromic, it enables us to express the $\varphi_i, \psi_i, \alpha_i, \beta_i$ from Lemma~\ref{lem_expanding} exclusively in terms of the graded stable Higgs bundle on $X_0$ (equivalently, the initial data $(\E, \bar\p, \theta, h_0, \Ch^{1,0})$) and the order $n$ deformation $X_n$ of $X_0$ (equivalently, the series $\eta = \sum\limits_{i=1}^n t^i \eta_i$).

We introduce the following notation, which will greatly simplify the writing of the expressions: (note that the following summation indices $i_1,i_2,\cdots,i_m$ are all positive integers)
\begin{align}\label{eq_altersum_S}
S_k:=S_k(g_1,\cdots,g_k)=\sum_{m=1}^k\sum_{i_1+i_2+\cdots+i_m=k}(-1)^{m-1}g_{i_1}g_{i_2}\cdots g_{i_m}.
\end{align}
and for any positive integer $N\geq 1$,
\begin{align}\label{eq_IN}
|I_N|:=i_1+i_2+\cdots+i_N
\end{align}
\begin{proposition} \label{prop_Taylor_Exp_DolHig}
Suppose the deformation is isomonodromic. 
For any $1\leq i\leq n$, $g_i$ is uniquely determined by $(\E,\bar\p,\theta,h_0,\Ch^{1,0})$ and $\sum\limits_{i=1}^kt^i\eta_i$.
And the explicit formulas for  $\varphi_i, \psi_i, \alpha_i, \beta_i$ (given in Lemma~\ref{lem_expanding}) are:
\begin{flalign} \label{varphi} 
(i) \quad \varphi_i &  = \bar\eta_i(\theta^\st)+\cc[\theta,\bg_i]-\cc\Ch^{1,0}\bg_i +\cc\sum_{ |I_2| =i} S^\st_{i_1}\big\{\Ch^{1,0}\bg_{i_2}-[\theta,\bg_{i_2}]\big\}&\notag\\
& \quad +\cc\sum_{|I_2|= i}\bar\eta_{i_1}\big([\theta^\st,\bg_{i_2}]-\bar\p\bg_{i_2}\big)-\cc\sum_{|I_3| = i}\bar\eta_{i_1}\Big\{S_{i_2}^\st\big([\theta^\st,\bg_{i_3}]-\bar\p\bg_{i_3}\big)\Big\}. &
\end{flalign}

\begin{flalign} \label{psi} 
(ii) \quad \psi_i=&-\cc[\theta^\st,\bg_i]+\cc\bar\p\bg_i+\cc\sum_{ |I_2| =i} S^\st_{i_1}\big\{[\theta^\st,\bg_{i_2}]-\bar\p\bg_{i_2}\big\}  \notag &\\
&+\cc\sum_{|I_2| = i}\bar\eta_{i_1}\big([\theta^\st,\bg_{i_2}]-\bar\p\bg_{i_2}\big)-\cc\sum_{|I_3|= i}\bar\eta_{i_1}\Big\{S_{i_2}^\st\big([\theta^\st,\bg_{i_3}]-\bar\p\bg_{i_3}\big)\Big\}.&
\end{flalign}

\begin{flalign} \label{alpha} 
(iii) \quad \alpha_i= & -\eta_i(\theta)+\cc[\theta,g_i]-\cc\Ch^{1,0}g_i +\cc\sum_{ |I_2| =i} S_{i_1}\big\{\Ch^{1,0}g_{i_m}-[\theta, g_{i_m}]\big\} & \notag \\
&+\cc\sum_{ |I_2|= i} \eta_{i_1}\big(\Ch^{1,0}g_{i_2}-[\theta, g_{i_2}]\big)-\sum_{ |I_3|=i} \eta_{i_1}\Big\{S_{i_2}\big(\Ch^{1,0}g_{i_3}-[\theta, g_{i_3}]\big)\Big\}. &
\end{flalign} 

\begin{flalign} \label{beta} 
(iv) \quad \beta_i=&-\cc[\theta^\st,g_i]+\cc\bar\p g_i+\cc\sum_{ |I_2|=i} S_{i_1}\big\{[\theta^\st, g_{i_2}]-\bar\p g_{i_2}\big\} \notag & \\
&+\cc\sum_{ |I_2| = i} \eta_{i_1}\big(\Ch^{1,0}g_{i_2}-[\theta, g_{i_2}]\big)-\sum_{ |I_3|=i} \eta_{i_1}\Big\{S_{i_2}\big(\Ch^{1,0}g_{i_3}-[\theta, g_{i_3}]\big)\Big\}. &
\end{flalign} 
In those formulas, the operators $\Ch^{1,0},\bar\p$ are connections of $\End\E$ induced by those of $\E$. And we remark that $\eta_i:\A^{1,0}(\End\E)\to \A^{0,1}(\End\E)$ and $\bar\eta_i:\A^{0,1}(\End\E)\to \A^{1,0}(\End\E)$ are contractions.
\end{proposition}

\begin{proof}
In this proof, we always modulo the ideal $(t\bar t)$. Firstly, we prove the above four identities. Let 
\begin{align*}
\pi_\eta':&B_n\otimes \A^1(\End\E)\to  B_n\otimes\Omega^{1,0}(X_n/A_n)\otimes \End\E; \\\pi_\eta'':&B_n\otimes \A^1(\End\E)\to  B_n\otimes\Omega^{0,1}(X_n/A_n)\otimes \End\E
\end{align*}
be two projections according to types. Then we have $\pi'_\eta\Psi_t=\theta_t$ and $\pi''_\eta\Psi_t=\theta_t^{\star_{h_t}}$. Using \eqref{PSI}, we have
{\fontsize{10}{12}\selectfont \begin{equation} \label{eq_theta_t}
\begin{aligned}
\theta_t=&g(t,\bar t)^{-1}(\theta-\eta(\theta)+\bar\eta(\theta^\st))g(t,\bar t)-\cc g(t,\bar t)^{-1}(D'-\eta\circ D'+\bar\eta\circ D'')^{\End}\Big(g(t,\bar t)\Big);\\
\theta_t^{\star_{h_t}}=&g(t,\bar t)^{-1}(\theta^\st+\eta(\theta)-\bar\eta(\theta^\st))g(t,\bar t)-\cc g(t,\bar t)^{-1}(D''+\eta\circ D'-\bar\eta\circ D'')^{\End}\Big(g(t,\bar t)\Big),
\end{aligned}
\end{equation}}
where $D':A^0(X_0,\End\E)\to A^{1,0}(X_0,\End\E)$ and $D'':A^0(X_0,\End\E)\to A^{0,1}(X_0,\End\E)$ with $D=D'+D''$. 
Note that we have $D'=\Ch^{1,0}+\theta$ and $D''=\bar\p+\theta^\st$. By comparing the coefficient of $t^i,\ i=0,1,2,\cdots,n$ in the above identity of $\theta_t$, we have
\begin{align*}\theta+\sum_{i=1}^nt^i\alpha_i =&(\id+\sum_{j=1}^n t^j g_j)^{-1}(\theta-\sum_{l=1}^nt^l\eta_l(\theta))(\id+\sum_{k=1}^n t^k g_k)\\
&-\cc(\id+\sum_{j=1}^n t^j g_j)^{-1}(\Ch^{1,0}+\theta-\sum_{l=1}^nt^l\eta_l\circ\Ch^{1,0}-\sum_{l=1}^nt^l\eta_l(\theta))^{\End}\sum_{k=1}^n t^k g_k.
\end{align*}
Substituting $(\id+\sum_{j=1}^n t^j g_j)^{-1}=\id-\sum_{j=1}^nt^j g_j+\sum_{j_1+j_2\leq n}t^{j_1+j_2}g_{j_1}g_{j_2}-\cdots$ into the above identity, we obtain the expression \eqref{alpha}. By a similar argument, we have \eqref{varphi}. Now we explain how to derive the rest two. By a direct check, we have $\pi_\eta'' D=D''+\eta\circ D'-\bar\eta\circ D''.$ Thus \begin{align*}\bar\p_t=\pi_\eta'' D-\theta_t^{\star_{h_t}}.\end{align*}
By substituting \eqref{eq_theta_t} into the above identity and comparing the coefficient of $t^i,\ i=1,2,\cdots,n$ in the above identity, we have
\begin{align*}\sum_{i=1}^nt^i\beta_i =&\sum_{l=1}^nt^l\eta_l(\theta)-(\id+\sum_{j=1}^n t^j g_j)^{-1}(\theta^\st+\sum_{l=1}^nt^l\eta_l(\theta))(\id+\sum_{k=1}^n t^k g_k)\\
&+\cc(\id+\sum_{j=1}^n t^j g_j)^{-1}(\bar\p+\theta^\st+\sum_{l=1}^nt^l\eta_l\circ\Ch^{1,0}+\sum_{l=1}^nt^l\eta_l(\theta))^{\End}\sum_{k=1}^n t^k g_k.
\end{align*}
This proves the expression \eqref{beta}. By a similar argument, we have \eqref{psi}.

Substituting the above expressions into the integrable conditions \eqref{eq_compatibility}, we get several PDEs on $g_i,\ i=1,2,\cdots,n$. By the nonabelian Hodge correspondence and the uniqueness of the harmonic metric, the above PDEs on $g_i$ are all solvable and uniquely determine $g_i,\ i=1,2,\cdots,n$.
\end{proof}

\newpage

\section{Obstruction classes of holomorphicity} \label{sec_ob_hol}

In this section, we investigate the holomorphicity of the real analytic isomonodromic deformations discussed in the previous section. We introduce a sequence of obstruction classes that measure the failure of an isomonodromic deformation to be holomorphic. These obstruction classes lie in a certain cohomology group, which we explicitly describe using the Dolbeault resolution. Proposition~\ref{prop_ob} is the main result in this section, in which we derive explicit formulas for all higher-order obstruction classes, and show that their vanishing is a necessary condition for the deformation to be holomorphic. 

\medskip

\subsection{Obstruction classes of modulo-$(t)$-holomorphicity} \label{sec_ob_mod_t_hol}
\subsubsection{Obstruction class group}
Recall the computation of the hypercohomology group $$\barHb$$ via the Dolbeault resolution. Firstly, we have $\Omega_{\overline{X_0}}^{1,0}=\Omega_{X_0}^{0,1}$ and $\Omega_{\overline{X_0}}^{0,1}=\Omega_{X_0}^{1,0}$, where $\overline{X_0}$ is the complex manifold conjugate to $X_0$. Hence $(\E,\Ch^{1,0},\theta^\st)$ is a Higgs bundle on $\overline{X_0}$. For any two $\omega_1\in\A^k(\End\E),\ \omega_2\in \A^l(\End\E)$, we define the following Lie brackets.
\begin{align}\label{eq_Liebracket_2}[\omega_1,\omega_2]:=\omega_1\circ\omega_2-(-1)^{kl}\omega_2\circ\omega_1.
\end{align}

We have the following Dolbeault resolution:
\[
\begin{tikzcd}
\vdots & \vdots & \vdots \\
C^{2,0}:=\mathcal A^{2,0}(\End E) 
\arrow[r, "{\operatorname{ad}( \theta^{\star_{h_0}})}"] \arrow[u]
& C^{2,1}:=\mathcal A^{2,0}(\End E \otimes \Omega_{X_0}^{0,1}) 
\arrow[r, "{\operatorname{ad}( \theta^{\star_{h_0}})}"] \arrow[u]
& C^{2,2}:=\mathcal A^{2,0}(\End E \otimes \Omega_{X_0}^{0,2}) 
\arrow[r] \arrow[u] & \cdots \\
C^{1,0}:=\mathcal A^{1,0}(\End E) 
\arrow[r, "{\operatorname{ad}( \theta^{\star_{h_0}})}"] 
\arrow[u, "\Ch^{1,0}"] 
& C^{1,1}:=\mathcal A^{1,0}(\End E \otimes \Omega_{X_0}^{0,1}) 
\arrow[r, "{\operatorname{ad}( \theta^{\star_{h_0}})}"] 
\arrow[u, "\Ch^{1,0}"] 
& C^{1,2}:=\mathcal A^{1,0}(\End E \otimes \Omega_{X_0}^{0,2}) 
\arrow[r] \arrow[u, "\Ch^{1,0}"] & \cdots \\
C^{0,0}:=\mathcal A^{0,0}(\End E) 
\arrow[r, "{\operatorname{ad}( \theta^{\star_{h_0}})}"] 
\arrow[u, "\Ch^{1,0}"] 
& C^{0,1}:=\mathcal A^{0,0}(\End E \otimes \Omega_{X_0}^{0,1}) 
\arrow[r, "{\operatorname{ad}( \theta^{\star_{h_0}})}"] 
\arrow[u, "\Ch^{1,0}"] 
& C^{0,2}:=\mathcal A^{0,0}(\End E \otimes \Omega_{X_0}^{0,2}) 
\arrow[r] \arrow[u, "\Ch^{1,0}"] & \cdots
\end{tikzcd}
\]
which gives the following truncated complex 
\[C^{0,0} \overset{d^{0c}}{\longrightarrow} C^{1,0} \oplus C^{0,1} \overset{d^{1c}}{\longrightarrow} C^{2,0}\oplus C^{1,1}\oplus C^{0,2}{\longrightarrow} \cdots\] 
where 
\begin{align*}
 d^{0c}(g) &= (\Ch^{1,0} g, [\theta^{\star_{h_0}},g])\in C^{1,0}\oplus C^{0,1} \quad \text{for } g \in C^{0,0},\\
 d^{1c}(\varphi, \psi) &= (\Ch^{1,0}\varphi,\Ch^{1,0} \psi +[\theta^{\star_{h_0}},\varphi],[\theta^\st,\psi])\in C^{2,0}\oplus C^{1,1}\oplus C^{0,2} \quad \text{for } (\varphi, \psi) \in C^{1,0} \oplus C^{0,1}.
\end{align*}
Hence
$$\barHb=\frac{\mathrm{Ker\ }d^{1c}}{\mathrm{Im\ }d^{0c}}.$$
We will see in the next subsection, this $\barHb$ is the desired obstruction group.

\subsubsection{Existence of obstruction classes}
Given a real analytic deformation of $(\E,\bar\p,\theta)$ on $X_0$ to $X/S$, let $(\E,\bar\p_t,\theta_t)$ be the truncation to $X_n$ via a holomorphic $n$-jet $c:\operatorname{Spec}A_n\to S,\ c(0)=0$.
\begin{proposition}\label{prop_existence_ob}Suppose $(\E,\bar\p_t,\theta_t)$ is modulo-$(t,\bar t^{k})$-holomorphic with $k<n$ (defined in Definition~\ref{def_modulo_hol}). Then there exists a class $\ob_k \in \barHb$, such that $\ob_k$ vanishes if and only if $(\E,\bar\p_t,\theta_t)$  is modulo-$(t,\bar t^{k+1})$-holomorphic. In particular, $\ob_1,\ob_2,\cdots,\ob_n$ all vanish one-by-one if and only if $(\E,\bar\p_t,\theta_t)$ is modulo-$(t)$-holomorphic.
 
\end{proposition}

\begin{remark}\begin{enumerate}
\item Modulo-$(t)$-holomorphic is strictly weaker than the holomorphicity, implying that further obstructions must exist to achieve full holomorphicity.
\item This existence proposition does not help us express the obstruction class explicitly, even for the obstruction class of the isomonodromic deformation. Hence, in the following subsections we will not use this proposition. Instead, we will only need to use the criterion Lemma~\ref{lem_gauge}, which also detects the obstructions of modulo-$(t)$-holomorphicity.
\end{enumerate}
\end{remark}
\begin{proof}

By definition, 
\((\mathcal E,\bar\partial_t,\theta_t)\) is given by a map
\[
\sigma:\operatorname{Spec}\mathbb C[t,\bar t]/(t,\bar t)^{n+1}
\longrightarrow M_{\mathrm{Dol}}(X/S)=:\mathcal{M}
\]
with $
\sigma(0)=(\mathcal E,\bar\partial,\theta)$ on \(X_0\). 

By the complexification argument in section~\ref{sec_joint_realanalytic}, $\sigma$ uniquely extends to a morphism $g$ such that the following diagram commutes
\[
\begin{tikzcd}
\operatorname{Spec}\mathbb C[t,\bar t]/(t,\bar t)^{n+1} \arrow[r,"\sigma"] \arrow[d,hook] & \mathcal M \arrow[d,hook]\\
\operatorname{Spec}(\mathbb C[t]/(t^{n+1})\otimes \C[\bar t]/(\bar t^{n+1}))\arrow[r, "g"] & M_{\Dol}((X\times\overline X)/(S\times \overline S)) 
\end{tikzcd}
\]

Define the jet spaces of maps sending  \(0\) to $o:=
(X_0,\mathcal E,\bar\partial,\theta)\in \mathcal M$:
\[
\operatorname{Hom}\left(
\operatorname{Spec}\mathbb C[t,\bar t]/(t,\bar t)^{k+1},
\mathcal M
\right)
:= \mathbb{C}J_k\mathcal M,
\]
and
\[
\operatorname{Hom}\left(
\operatorname{Spec}\mathbb C[\bar t]/(\bar t^{k+1}),
M_{\Dol}(\overline X/\overline S)
\right)
:= \overline{ J_k\mathcal M} .
\]

Then $g$ induces a $k$-jet map $\operatorname{Spec}(\C[\bar t]/(\bar t^{k+1}))\to M_{\Dol}(\overline X/\overline S) $ by base changing via $$0\times\operatorname{Spec}(\C[\bar t]/(\bar t^{k+1}))\hookrightarrow  0\times \operatorname{Spec}(\C[\bar t]/(\bar t^{n+1}))\hookrightarrow\operatorname{Spec}(\mathbb C[t]/(t^{n+1})\otimes \C[\bar t]/(\bar t^{n+1})).$$
We denote this $k$-jet map by $p_k([\sigma])\in \overline{ J_k\mathcal M}$.

For the jet \([\sigma]\), we have the following commutative diagram:
\[
\begin{tikzcd}
 & & \mathbb{C}J_k\mathcal M
 \arrow[r, "\Pi_{k-1}^k"]
 \arrow[d, "p_k"']
 & \mathbb{C}J_{k-1}\mathcal M
 \arrow[d, "p_{k-1}"]
 & \\
0 \arrow[r]
& T^{1,0}_{o}M_{\Dol}(\overline X/\overline S) \arrow[r]
& \overline{ J_k\mathcal M}
 \arrow[r, "\pi_{k-1}^k"]
& \overline{ J_{k-1}\mathcal M}
 \arrow[r]
& 0 
\end{tikzcd}
\]

By modulo-$(t,\bar t^{k})$-holomorphicity, \([\sigma]\) maps to zero under
\[
p_{k-1}\circ \Pi_{k-1}^k .
\]
Thus \(p_k([\sigma])\) is given by an element in \(T^{1,0}_{o}M_{\Dol}(\overline X/\overline S)\). Let $\mathbb CT_o\mathcal M=T^{1,0}_{o}\mathcal M\oplus T^{0,1}_{o}\mathcal M$ be the complexification of the real Zariski tangent space of a real analytic variety $\mathcal M$ at $o$. One can prove directly there is natural isomorphism $T^{1,0}_{o}M_{\Dol}(\overline X/\overline S)\cong T^{0,1}_o M_{\Dol}(X/S)$ by using the argument in Proposition~\ref{prop:KS-real-analytic-manifold}. Thus we obtain \(p_k([\sigma])\) is given by an element in $T^{0,1}_o \mathcal M$.
\bigskip

Now we prove that \(p_k([\sigma])\) is given by an element in
\(T^{0,1}_oM_{\mathrm{Dol}}(X_0)\). 
By definition, the composition
\[
\operatorname{Spec}\mathbb C[t,\bar t]/(t,\bar t)^{n+1}
\xrightarrow{\ \sigma\ }
\mathcal M
\xrightarrow{\ \pi_{\mathrm{Dol}}\ }
S
\]
is independent of $\bar t$ and is equal to the holomorphic \(n\)-germ $c$ of \(S\) through $0$. Therefore 
\[
p_k\bigl([\pi_{\mathrm{Dol}}\circ \sigma]\bigr)=0
\qquad
\text{in } \quad\overline{ J_k S}
\]
by holomorphicity, and we obtain
\[
p_k\bigl([\pi_{\mathrm{Dol}}\circ \sigma]\bigr)
=
(\pi_{\mathrm{Dol},*})\bigl(p_k([\sigma])\bigr)
=0 .
\]
Therefore $
p_k([\sigma])\in \ker(\pi_{\mathrm{Dol},*}),$ 
and hence $
p_k([\sigma])\in T^{0,1}_oM_{\mathrm{Dol}}(X_0),$ which gives \[\ob_k:=p_k([\sigma])\in T^{0,1}_oM_{\mathrm{Dol}}(X_0)\cong \barHb.\qedhere \]
\end{proof}

We give the following Lemma about ``harmonicity'', which will be repeatedly used later.
\begin{lemma} \label{vanlem}
Let $(\E, \bar\partial, \theta)$ be a stable Higgs bundle. In particular, it admits a harmonic metric $h_0$ and thus it is a harmonic bundle. 
 Suppose  $g\in \A^0(\End \E)$ satisfies either of the following two equations:
\begin{align}\label{exact-harm}
\Ch^{1,0}\bar\p g+[\theta^\st,[\theta,g]]=0,
\end{align}
or 
\begin{align}\label{exact_harm2}
\bar\p\Ch^{1,0} g+[\theta,[\theta^\st,g]]=0,
\end{align}
Then $g=c\cdot\id$ for some $c\in\C.$
\end{lemma}

\begin{remark}
There is a Hodge theoretic interpretation of ``harmonicity'' in Lemma~\ref{vanlem}:
\begin{enumerate}\item  The class $[([\theta,g],\bar\p g)]\in \mathbb H^1(X_0,(\End E,\operatorname{ad}(\theta)))$ is exact. 
\item  The equation \eqref{exact-harm} implies $[([\theta,g],\bar\p g)]\in \mathbb H^1(X_0,(\End E,\operatorname{ad}(\theta)))$ is harmonic.
\end{enumerate}
By Hodge decomposition theory, an exact and harmonic class must be zero.
\end{remark}

\begin{proof}[Proof Lemma~\ref{vanlem}]
Assuming \eqref{exact-harm}, we have
\begin{align*}\sqrt{-1}\int_{X_0}\operatorname{tr}(g^\st\Ch^{1,0}\bar\p g)+\sqrt{-1}\int_{X_0}\operatorname{tr}(g^\st[\theta^\st,[\theta,g]])=0.
\end{align*}
By the K\"ahler identity (see \cite{Simp92} and \cite[Remark 9.2]{Gui}), we have
 $$\sqrt{-1}\int_{X_0}\operatorname{tr}(g^\st\Ch^{1,0}\bar\p g)=-\sqrt{-1}\int_{X_0}\operatorname{tr}((\bar\p g)^\st\bar\p g)\leq 0.$$ One may verify directly that $$\sqrt{-1}\int_{X_0}\operatorname{tr}(g^\st[\theta^\st,[\theta,g]])=-\sqrt{-1}\int_{X_0}\operatorname{tr}([\theta,g]\wedge [\theta,g]^\st)\leq 0.$$
Thus
we have $\bar\p g=0 $ and $[\theta,g]=0.$ This means that $g\in\mathbb H^0(X_0,(\End E,\operatorname{ad}(\theta)))$ and by the stability $g=c\cdot\id$ for some $c\in\C$.

If $g$ satisfies \eqref{exact_harm2}, one can prove $g\in\mathbb H^0(X_0,(\End E,\operatorname{ad}(\theta)))$ similarly.
\end{proof}

\begin{corollary}\label{coro_harm}
Let $(\E, \bar\partial, \theta)$ be a stable Higgs bundle. Suppose $g,f\in \A^0(\End \E)$ satisfies the following system of equations
\begin{align*}\bar\p g=[\theta^{\star_{h_0}},f];\qquad
[\theta,g]=D_{h_0}^{1,0}f.
\end{align*}
Then $g=c_1\cdot\id$ and  $f=c_2\cdot\id$ for some constant $c_1,c_2\in\mathbb C$. 
\end{corollary}
\begin{proof}Applying $\Ch^{1,0}(-)$ to the first equation and $[\theta^\st,-]$ to the second and summing them, we have
\begin{align*}\Ch^{1,0}\bar\p g+[\theta^\st,[\theta,g]]=0,
\end{align*}
which implies $g=c_1\cdot\id$ by the ``harmonicity'' in Lemma~\ref{vanlem}. Applying $[\theta,-]$ to the first equation and $\bar\p(-)$ to the second and summing them, we have
\begin{align*}\bar\p\Ch^{1,0} f+[\theta,[\theta^\st,f]]=0,
\end{align*}
which implies $f=c_2\cdot\id$ by the ``harmonicity'' in Lemma~\ref{vanlem}.
\end{proof}

\subsection{Obstruction classes of modulo-\texorpdfstring{$(t)$}{(t)}-holomorphicity of the isomonodromic deformation of a graded Higgs bundle} \label{sec_ob_mod_t_hol_isom_grad}

Let $(\E,\bar\p,\theta)$ be a graded stable Higgs bundle on $X_0$ with weight $w$. In this section we always let $(\E,\bar\p_t,\theta_t)$ be the isomonodromic deformation of $(\E,\bar\p,\theta)$ on $X_0$ to $X_n$, which is a real analytic deformation as in Definition~\ref{def_real_ana_deform}. We try to investigate the obstruction classes of modulo-$(t)$-holomorphicity defined in Proposition~\ref{prop_existence_ob} and give some necessary conditions of the vanishing of those obstruction classes.

Firstly, we have canonical decomposition
$\End \E=\bigoplus\limits_{l=-w}^w(\End\E)^{i,-i}$,
where $(\End \E)^{i,-i}:=\{f\in \End \E\mid f(\E^{p,k-p})\subset \E^{p+i,k-p-i}\}$.
We extend the grading by setting $(\End \E)^{i,-i}=0$ for all $|i|>w$, so that we have the direct sum decomposition
\begin{align*}
\End \E = \bigoplus_{i\in \mathbb{Z}} (\End\E)^{i,-i}.
\end{align*}
This grading naturally extends to differential forms: any $\End\E$-valued $l$-form $f$ decomposes uniquely as $f = \bigoplus_i f^{i,-i}$ with $f^{i,-i} \in \mathcal{A}^l((\End\E)^{i-l,-i+l})$. We refer to $f^{i,-i}$ simply as the \textbf{$(i,-i)$-grading piece} of $f$. In particular,
\[\theta = \theta^{0,0},\quad \theta^\st=(\theta^\st)^{2,-2},\quad \bar\p (f^{i,-i})=(\bar\p f)^{i+1,-i-1},\quad \Ch^{1,0} (f^{i,-i})=(\Ch^{1,0}  f)^{i+1,-i-1},\]
and $\eta(f^{i,-i})=(\eta(f))^{i,-i}$ for any $f\in \A^{1,0}(\End\E),\ \eta\in \A^{0,1}(T_{X_0})$.
\subsubsection{First order obstruction class}
To study the first order holomorphicity, let $X_1:=X_n\times_{\operatorname{Spec}A_n} \operatorname{Spec}A_1$ and we may pull-back $(\E,\bar\p_t,\theta_t)$ to $X_1$. In this case, modulo-$(t,\bar t^2)$-holomorphicity on $X_n$ is equivalent to the holomorphicity on $X_1$ (after base changing). 
\begin{proposition}\label{prop_firstob}
Let $(\E,\bar\p_t,\theta_t)$ be the real analytic deformation in Definition~\ref{def_real_ana_deform} of $(\E,\bar\p,\theta)$ on $X_0$ along $X_n$. If it is isomonodromic, then the obstruction class $\ob_1$ of modulo-$(t,\bar t^2)$-holomorphicity is 
\begin{align*}\ob_1=[(\bar\eta_1(\theta^\st),0)]\in\barHb.
\end{align*}
\end{proposition}
\begin{proof}\textbf{Step1:} modulo-$(t,\bar t^2)$-holomorphicity implies the vanishing of the class $[(\bar\eta_1(\theta^\st),0)]$. By equations \eqref{hol} for gauge transformation, the equations \eqref{varphi} and \eqref{psi} of $\varphi_1$ and $\psi_1$, we have
\begin{equation} \label{g1}
\begin{aligned}
\cc\bar\p \bg_1-\cc[\theta^{\star_{h_0}},\bg_1]+\bar\p u_1&=0;\\
\bar \eta_1(\theta^{\star_{h_0}})+\cc[\theta,\bg_1]-\cc D_{h_0}^{1,0} \bg_1+[\theta,u_1]&=0.
\end{aligned}
\end{equation}
Hence
\begin{align*}
\Ch^{1,0}(\bar\p (\cc\bg_1+u_1)-\cc[\theta^{\star_{h_0}},\bg_1])=0;\qquad[\theta^\st,[\theta,\cc\bg_1+u_1]+\bar \eta_1(\theta^{\star_{h_0}})-\cc D_{h_0}^{1,0} \bg_1]=0.
\end{align*}
Note that $[\theta^\st,\bar\eta_1(\theta^\st)]=0$ and $\Ch^{1,0}([\theta^\st,\bg_1])=-[\theta^\st,\Ch^{1,0}\bg_1]$, we have 
\begin{align*}
&\Ch^{1,0}(\bar\p (\cc\bg_1+u_1)-\cc[\theta^{\star_{h_0}},\bg_1])+
[\theta^\st,[\theta,\cc\bg +u_1]+\bar \eta_1(\theta^{\star_{h_0}})-\cc D_{h_0}^{1,0} \bg_1]\\
=&\Ch^{1,0}\bar\p (\cc\bg_1+u_1)+
[\theta^\st,[\theta,\cc\bg +u_1]]=0.
\end{align*}
By the ``harmonicity'' in Lemma~\ref{vanlem}, we have $\bar\p (\cc\bg_1+u_1)=0$ and $[\theta,\cc\bg_1 +u_1]=0$. This implies $\cc\bg_1+u_1\in \mathbb H^0(X_0,(\End E,\operatorname{ad}(\theta)))$. The stability gives $u_1=-\cc\bg_1$ (up to adding a term $c\cdot\id$ with $c\in\C$, which we may ignore since it does not affect \eqref{g1}). Thus \eqref{g1} reduces to
\begin{align}\label{eq_ob1exact}0=\cc[\theta^{\star_{h_0}},\bg_1];\qquad\ 
\bar \eta_1(\theta^{\star_{h_0}})=\cc D_{h_0}^{1,0} \bg_1,\end{align} 
i.e. $\ob_1\in\barHb$ vanishes. 

\textbf{Step2:} the vanishing of the class $[(\bar\eta_1(\theta^\st),0)]$ implies modulo-$(t,\bar t^2)$-holomorphicity. We aim to prove the solvability of \eqref{g1} on $u_1$. Since $[(\bar\eta_1(\theta^\st),0)]$ vanishes, there exists $f_1\in \A^0(\End\E)$ such that 
\begin{align}\label{eq_ob1vanish}0=[\theta^{\star_{h_0}},f_1];\qquad\ 
\bar \eta_1(\theta^{\star_{h_0}})= D_{h_0}^{1,0} f_1.
\end{align}
By \cite[Proposition 4.4 (4.4)]{HSZ}, the condition  $\bar\p_t\theta_t\equiv 0\pmod{(t,\bar t)^2}$ implies that 
\begin{align*}\bar\p\Ch^{1,0}\bg_1=2\bar\p(\bar\eta(\theta^\st))-[\theta,[\theta^\st,\bg_1]].
\end{align*}
This together with the assumption \eqref{eq_ob1vanish} gives 
\begin{align*}\bar\p\Ch^{1,0}(\bg_1-2f_1)+[\theta,[\theta^\st,\bg_1-2f_1]]=0,
\end{align*}
which implies $\bg_1=2f_1+c\cdot\id$ for some $c\in\mathbb C$ by the ``harmonicity'' in Lemma~\ref{vanlem}.
Substituting this and \eqref{eq_ob1vanish} into \eqref{g1}, we have
\begin{align*}
\bar\p (\cc\bg_1+u_1)=0;\qquad
[\theta,\cc\bg_1+u_1]=0.
\end{align*}
Thus $u_1=-\cc\bg_1$ is the solution of \eqref{g1}.
\end{proof}

\begin{corollary}\label{coro_g1}If $\ob_1$ in Proposition~\ref{prop_firstob} vanishes, we have
\begin{align}\label{eq_g1u1}
 \bg_1\in \A^0((\End \E)^{1,-1})\qquad\text{and}\qquad \ u_1=-\cc\bg_1.\end{align}
\end{corollary}
\begin{proof}Using \eqref{eq_ob1exact} and the fact $\bar\eta_1(\theta^\st)\in \A^0((\End \E)^{2,-2})$, we have 
\begin{align*}[\theta^\st,(\bg_1)^{l,-l}]=0;\ 
 \Ch^{1,0} (\bg_1)^{l,-l}=0\ \text{ for }l\ne1.
 \end{align*}
 Thus $\bg_1=(\bg_1)^{1,-1}+c\cdot\operatorname{id}$ for some $c\in\C$. Since $c\cdot\operatorname{id}$ in \eqref{metric} can be eliminated by a gauge transformation, we have the first claim in \eqref{eq_g1u1}. The second claim in \eqref{eq_g1u1} is derived in the proof of Proposition~\ref{prop_firstob}.
\end{proof}

\subsubsection{The Zariski tangent space of the non-abelian Noether-Lefschetz locus}

Recall the setup for the non-abelian Noether-Lefschetz locus given in equation \eqref{eq_NLlocus}. Let $$\theta_*:H^1(T_{X_0})\to \mathbb{H}^1\bigl(X_0,(\End E,\operatorname{ad}(\theta))\bigr)$$ be the non-abelian Higgs field defined in \eqref{eq_nabHiggs} and let $\tau_0:T_0^{1,0}S\to H^1(T_{X_0})$ is the Kodaira-Spencer map of the family $X/S$ at $0$. By \cite[Theorem A]{HSZ}, the condition $\theta_*\circ\tau_0(v)=0$ is equivalent to the first-order holomorphicity of the Dolbeault $\sigma_{\Dol}$ along any tangent vector $v\in T_0S$. By virtue of this result, the first-order truncated case of Theorem~\ref{thm_main_tru} reduces to the following theorem. For all higher-order truncated cases, we will employ the same strategy to establish the full statement of Theorem~\ref{thm_main_tru} in section~\ref{sec_holo_gauge_tran} and section~\ref{sec_proof_main}.

\begin{thm}\label{thm_Zar_NL}
The Zariski tangent space of $\mathcal{NL}$ at $0\in S$ is  
\begin{align*}T^{\mathrm{Zar}}_0\mathcal{NL}=\{v\in T^{1,0}_0S\mid  v\in \ker(\theta_*\circ\tau_0)\}.
\end{align*}
\end{thm}
\begin{proof}
By \cite[Theorem C]{HSZ}, for any $v\in T^{1,0}_0S$, if $v\notin \ker(\theta_*\circ\tau_0)$, then the isomonodromic deformed Higgs bundles $\sigma_{\Dol}$ is not graded (in fact not nilpotent) along $v$. Thus $v\notin T^{\mathrm{Zar}}_0\mathcal{NL}$.

If $v\in \ker(\theta_*\circ\tau_0)$, we prove the isomonodromic deformed Higgs bundles $\sigma_{\Dol}$ coincides with a holomorphic family of graded Higgs bundles up to first order along $v$. Let $[\eta_1]:=\tau_0(v)\in H^1(T_{X_0})$. The condition $\theta_*([\eta_1])=0$ is equivalent to the existence of $f_1\in \A^0(\End\E)$ such that \eqref{eq_ob1vanish} holds by taking $\st$. Thus by the proof of Proposition~\ref{prop_firstob} Step2, we know that $\bg_1=2f_1+c\cdot \id$, for some $c\in\C$. Substituting this into the deformation terms Proposition~\ref{prop_Taylor_Exp_DolHig}, we have
\begin{align*}\varphi_1=\cc[\theta,\bg_1],&\quad \psi_1=\cc\bar\p\bg_1,\\
\alpha_1=-\eta_1(\theta)-\cc\Ch^{1,0}g_1,&\quad \beta_1=-\cc[\theta^\st,g_1]+\bar\p g_1.
\end{align*}
By taking $\mathscr U:=\id-\frac{t}{2}g_1-\frac{\bar t}{2}\bg_1$, we have on $X_1$
\begin{align}\label{eq_firstdeforma}\mathscr U^{-1}\circ\bar\p_t\circ\mathscr U=\pi''_{\eta}\Ch^{1,0}+t\cdot\cc[\theta^\st,g_1],\qquad
\mathscr U^{-1}\circ\theta_t\circ\mathscr U=\theta-t\eta_1(\theta)-t\cdot\cc\Ch^{1,0}g_1.
\end{align}

Note that the condition $\theta_*([\eta_1])=0$ is equivalent to the vanishing of the first-order obstruction class $\ob_1$ stated in Proposition~\ref{prop_firstob}, which in turn implies $g_1\in \mathcal{A}^0((\End\E)^{-1,1})$ as established in Corollary~\ref{coro_g1}. It follows immediately that
\begin{align*}
[\theta^\st,g_1]\in\mathcal{A}^{1}((\End\E)^{0,0}) \qquad \text{and} \qquad \eta_1(\theta)+\cc\Ch^{1,0}g_1\in \mathcal{A}^{1}((\End\E)^{-1,1}).
\end{align*}
This and \eqref{eq_firstdeforma} imply that the isomonodromically deformed Higgs bundle preserves a graded structure up to first order.\end{proof}

In summary, we have the following slogan:
\[
\begin{gathered}
\text{First order holomorphicity of the isomonodromic deformation } \;\Leftrightarrow\; 
\ob_1 \text{ in Proposition~\ref{prop_firstob} vanishes} \\ 
\Downarrow \\ 
\text{a full restriction on } \bg_1: \eqref{eq_ob1exact} \;\Leftrightarrow\; \text{a full restriction on } g_1 \\ 
\Downarrow \\
\text{a full restriction on the first order holomorphic deformation: }  \eqref{eq_firstdeforma}\\ 
\Downarrow \\
\text{the first order liftablity of the initial graded structure.}
\end{gathered}
\]
We will extend this method to higher orders to answer Question~\ref{EsnKer}.

\subsubsection{A necessary condition on grading pieces}
\begin{proposition} \label{wtprop}
Suppose the isomonodromic deformation of the initial graded stable Higgs bundle is modulo-$(t,\bar t^{k+1})$-holomorphic. Then we have for any $m=1,2,\cdots,k$
\begin{align*}\bg_m\in\bigoplus_{l>0} \A^0((\End \E)^{l,-l}),
\end{align*}
where $\bg_k$ is defined in \eqref{metric}.
 Equivalently, $g_m\in \bigoplus_{l<0} \A^0((\End \E)^{l,-l})$ for such $m$.
\end{proposition}

\begin{proof}The case $k=1$ has been proved in \eqref{eq_g1u1}. We inductively assume that \begin{align}\label{u-wt1}\bg_m,u_m\in\bigoplus_{l>0} \A^0((\End \E)^{l,-l})\qquad\text{and }\qquad u_m^{1,-1}=-\cc\bg_m\end{align}
 for $m=1,2,\cdots,k-1$ and prove this for $\bg_k,u_k$.

By induction and the expressions \eqref{varphi}, \eqref{psi}, we have for $1\leq m\leq k-1$
\begin{align}\label{eq_wtphipsi}\varphi_m\in \bigoplus_{l\geq1} \A^0((\End \E)^{l,-l})\quad\text{ and }\quad\psi_m\in \bigoplus_{l\geq2} \A^0((\End \E)^{l,-l}).\end{align} 
By the equations \eqref{hol} for gauge transformation with $m=k$,
we have that for any integer $l\leq 1$,
 \begin{align*}
 & 0=(\varphi_k+\sum_{j=1}^{k-1} \varphi_{j}u_{k-j}+[\theta,u_k])^{l,-l}
 =(\cc[\theta,\bg_k]- \cc \Ch^{1,0}\bg_k+[\theta,u_k])^{l,-l};\\
 & 0=(\psi_{k}+\sum_{j=1}^{k-1}\psi_{j}u_{k-j}+\bar\p u_{k}-\sum_{j=1}^{k-1}\bar\eta_{j}(\bar\p u_{k-j}))^{l+1,-l-1}
 = (\cc\bar\p \bg_k-\cc[\theta^\st,\bg_k]+\bar\p u_k)^{l+1,-l-1}.
 \end{align*}
 By the ``harmonicity'' in Corollary~\ref{coro_harm}, this implies for any integer $l\leq0$
\begin{align*}(\bg_k)^{l,-l}=0\quad \text{and}&\quad u_k^{l,-l}=0 ;\\
 u_k^{1,-1}=-\cc (&\bg_k)^{1,-1},\end{align*}
 (note that any constant multiple of $\id$ in $\bg_k$ can be eliminated by a gauge transformation). Thus \eqref{eq_wtphipsi} also holds for $m=k$.
\end{proof}

\subsubsection{Higher order holomorphicity}
\begin{proposition}[Partial equations for the deformed harmonic metric]
\label{prop_ob}
Suppose the isomonodromic deformation of the initial graded stable Higgs bundle is modulo-$(t,\bar t^k)$-holomorphic. Denote 
\begin{equation}\begin{aligned}
\ob_{k,1}^{2,-2} := & \bar\eta_{k}(\theta^\st)-\cc\sum_{ |I_2|= k}\bar\eta_{i_1}(\bar\p\bg_{i_2})^{2,-2}+\frac{1}{4}\sum_{ |I_2| = k}\frac{i_2}{i_2+i_1}\cdot[[\theta,\bg_{i_1}],\bg_{i_2}]^{2,-2};\\
\ob_{k,2}^{3,-3} := & \frac{1}{4} \sum_{ |I_2| = k}\frac{i_2}{i_2+i_1}\cdot[\bar\p\bg_{i_1},\bg_{i_2}]^{3,-3}, 
\end{aligned}
\end{equation}
Then 
\begin{enumerate}
\item[$(i)$] $[\theta^\st,\ob_{k,1}^{2,-2}]  +\Ch^{1,0}\ob_{k,2}^{3,-3} =0$;

\item[$(ii)$] if the isomonodromic deformation is in addition modulo-$(t,\bar t^{k+1})$-holomorphic, then the identities
\begin{equation}\label{obnvanish}
\ob_{k,1}^{2,-2} = \cc(\Ch^{1,0}\bg_{k})^{2,-2} \quad \text{and} \quad
\ob_{k,2}^{3,-3} = \cc[\theta^\st,\bg_{k}]^{3,-3}
\end{equation}
hold.

\end{enumerate}
\end{proposition}

\begin{remark}
\begin{enumerate}
\item We explain why $[(\ob_{k,1}^{2,-2},\ob_{k,2}^{3,-3})]$ indeed forms a class in $\barHb$. If the fiber $X_0$ has dimension $1$, the conclusion of (i) is sufficient. If $\dim X_0>1$, one may further prove $\Ch^{1,0}\ob_{k,1}^{2,-2}=0$ and $[\theta^\st,\ob_{k,2}^{3,-3}]=0
$ by a detailed computation. However, this fact is not used elsewhere in this paper, and we therefore omit the computation.

\item Consequently, by (ii), the modulo-$(t,\bar t^{k+1})$-holomorphicity implies the vanishing of the class $[(\ob_{k,1}^{2,-2},\ob_{k,2}^{3,-3})]$.  We thus believe $[(\ob_{k,1}^{2,-2},\ob_{k,2}^{3,-3})]$ coincides exactly with the first graded piece of the complete obstruction class $\ob_k$ defined in Proposition~\ref{prop_existence_ob}. This assertion has been verified by direct computation for small $k$. For arbitrary $k$, the full calculation of $\ob_k$ is prohibitively complex, and a general proof will be deferred to future work. This is the rationale behind our notational choice.
\item By definition and Proposition~\ref{wtprop}, the reader may verify that this proposition involves only $\{(\bg_i)^{1,-1}\}_{i=1}^k$, i.e. no terms of the form $\{(\bg_i)^{l,-l}\}_{i=1}^k$ with $l\ne1$ appear. In conclusion, the modulo-$(t,\bar t^{k+1})$-holomorphicicty yields a precise equation \eqref{obnvanish} on $(\bg_k)^{1,-1}$ and indeed determines $(\bg_k)^{1,-1}$ by the ``harmonicity'' in Lemma~\ref{vanlem}. Also,  \eqref{obnvanish} is a generalization of the first order restriction \eqref{eq_ob1exact}.
\end{enumerate}
\end{remark}

\begin{proof}[Proof of Proposition~\ref{prop_ob}]To keep the flow of the proof clear, we defer the verification of some technical, computational identities to the end of this section, where the reader may consult them.

When $k=1$, our claim holds by \eqref{eq_ob1exact} in the proof of 
Proposition~\ref{prop_firstob}. We inductively assume (i)(ii) hold for $1,2,\cdots,k-1$ and verify this for $k\geq 2$. 
By induction, we have the following equations coming from modulo-$(t,\bar t^{m+1})$-holomorphic, where $m=1,2,\cdots,k-1$:

\begin{align}
\cc(\Ch^{1,0}\bg_m)^{2,-2}=&\bar\eta_{m}(\theta^\st)-\cc\sum_{i=1}^{m-1}\bar\eta_i(\bar\p\bg_{m-i})^{2,-2}+\frac{1}{4}\sum_{i=1}^{m-1} \frac{i}{m}\cdot[[\theta,\bg_{m-i}],\bg_i]^{3,-3}; \label{vanisheq_1}\\
\cc[\theta^\st,\bg_m]^{3,-3}=&\frac{1}{4}\sum_{i=1}^{m-1} \frac{i}{m}\cdot[\bar\p\bg_{m-i},\bg_i]^{3,-3}. \label{vanisheq_2}
\end{align}

Now we prove the following identity:

\begin{flalign}\label{ab(1)}
&0=[\theta^\st,\ob_{k,1}^{2,-2}]  +\Ch^{1,0}\ob_{k,2}^{3,-3} =   \sum_{|I_2| = k}[\theta^\st,\Big(-\cc\bar\eta_{i_1}(\bar\p\bg_{i_2}) 
+ \frac14\cdot\frac{i_2}{i_2+i_1} \cdot[[\theta,\bg_{i_1}],\bg_{i_2}]\Big)^{2,-2}] \notag &\\
& \makebox[5cm]{}+\sum_{ |I_2|= k}\Ch^{1,0}(\frac{1}{4} \cdot\frac{i_2}{i_2+i_1}\cdot[\bar\p\bg_{i_1},\bg_{i_2}])^{3,-3}.&
\end{flalign}
Taking the following two pieces in \eqref{ab(1)}
\begin{align*}
\textbf{Eq1} & :=\sum\limits_{|I_2|=k}[\theta^\st,-\cc\bar\eta_{i_1}(\bar\p\bg_{i_2})]^{4,-4} \\
\textbf{Eq2} & := \sum\limits_{|I_2| = k}\big([\theta^\st,\frac{1}{4}\cdot\frac{i_2}{i_2+i_1}\cdot[[\theta,\bg_{i_1}],\bg_{i_2}]]+\Ch^{1,0}(\frac{1}{4} \cdot\frac{i_2}{i_2+i_1}\cdot[\bar\p\bg_{i_1},\bg_{i_2}])\big)^{4,-4}.\\
\end{align*}
We reduce \eqref{ab(1)} to prove $\textbf{Eq1}+\textbf{Eq2}=0$. Firstly, 
{\fontsize{10}{12}\selectfont
\begin{equation} \label{ab(1)1}
\begin{aligned}
\textbf{Eq1} = & \cc\sum_{|I_2|= k}[\bar\eta_{i_1}(\theta^\st),\bar\p\bg_{i_2}]^{4,-4}\\
\overset{\eqref{vanisheq_1}}{=}
& \frac{1}{4}\sum_{|I_3|= k}[\bar\eta_{i_1}(\bar\p\bg_{i_3})-\cc \frac{i_3}{i_1+i_3}[[\theta,\bg_{i_1}],\bg_{i_3}],\bar\p\bg_{i_2}]^{4,-4}+\frac{1}{4}\sum_{|I_2| = k}[\Ch^{1,0}\bg_{i_1},\bar\p\bg_{i_2}]^{4,-4}\\
= & -\frac{1}{8}\sum_{|I_3|=k} \frac{i_3}{i_1+i_3}[[[\theta,\bg_{i_1}],\bg_{i_3}],\bar\p\bg_{i_2}]^{4,-4}+\frac{1}{4}\sum_{|I_2|=k}[\Ch^{1,0}\bg_{i_1},\bar\p\bg_{i_2}]^{4,-4}.
\end{aligned}
\end{equation}}
The third equality follows from the fact  

\begin{align*}
\sum_{|I_3|=k}[\bar\eta_{i_1}(\bar\p\bg_{i_3}),\bar\p\bg_{i_2}]^{4,-4}
& =\cc\sum_{|I_3|=k}\Big([\bar\eta_{i_1}(\bar\p\bg_{i_3}),\bar\p\bg_{i_2}]+[\bar\eta_{i_1}(\bar\p\bg_{i_2}),\bar\p\bg_{i_3}]\Big)^{4,-4}\\
& \overset{\eqref{eq_C2}}=  \cc\sum_{|I_3|=k}\Big([\bar\eta_{i_1}(\bar\p\bg_{i_3}),\bar\p\bg_{i_2}]-[\bar\p\bg_{i_2},\bar\eta_{i_1}(\bar\p\bg_{i_3})]\Big)^{4,-4}\\
& =0.\end{align*}

By taking $g = \bg_{i_1}$ and $h = \bg_{i_2}$ in the identity \eqref{eq_C1}, we have 
{\small \begin{align*}
&[\theta^\st,  [[\theta,\bg_{i_1}],\bg_{i_2}]]  +\Ch^{1,0}([\bar\p \bg_{i_1},\bg_{i_2}])\\
&= 
-[[\theta,[\theta^\st,\bg_{i_1}]],\bg_{i_2}]
-[[\theta,\bg_{i_1}],[\theta^\st,\bg_{i_2}]]
-\bar\p([\Ch^{1,0}\bg_{i_1},\bg_{i_2}])
-[\bar\p \bg_{i_1},\Ch^{1,0}\bg_{i_2}]
-[\bar\p \bg_{i_2},\Ch^{1,0}\bg_{i_1}]
.\end{align*} }
substituting this into \textbf{Eq2}, we obtain
{\fontsize{10}{12}\selectfont
\begin{equation} \begin{aligned} \label{ab(1)2}
\textbf{Eq2} = & 
-\frac{1}{4} \sum_{|I_2| = k} \frac{i_2}{i_2+i_1} \cdot \Big(
[[\theta,[\theta^\st,\bg_{i_1}]],\bg_{i_2}]
+[[\theta,\bg_{i_1}],[\theta^\st,\bg_{i_2}]]
+\bar\p([\Ch^{1,0}\bg_{i_1},\bg_{i_2}])  \\
& \makebox[6.3cm]{} +[\bar\p\bg_{i_1},\Ch^{1,0}\bg_{i_2}] +[\bar\p\bg_{i_2},\Ch^{1,0}\bg_{i_1}]
\Big)^{4,-4}\\
= & -\frac{1}{4}\sum_{|I_2| = k} \frac{i_2}{i_2+i_1} \cdot \Big(
[[\theta,[\theta^\st,\bg_{i_1}]],\bg_{i_2}]
+ [[\theta,\bg_{i_1}],[\theta^\st,\bg_{i_2}]] 
+\bar\p([\Ch^{1,0}\bg_{i_1},\bg_{i_2}] 
\Big)^{4,-4} \\
& \makebox[6cm]{}  -\frac{1}{4} \sum_{|I_2|= k} [\bar\p\bg_{i_1},\Ch^{1,0}\bg_{i_2}]^{4,-4}.
\end{aligned}\end{equation}}
The second equality follows from the equality:
\begin{align*}
\sum\limits_{|I_2|= k} \frac{i_2}{i_2+i_1}\cdot\big([\bar\p\bg_{i_1},\Ch^{1,0}\bg_{i_2}]+[\bar\p\bg_{i_2},\Ch^{1,0}\bg_{i_1}]\big)^{4,-4} 
= &\sum_{|I_2| = k} (\frac{i_2}{i_2+i_1}+\frac{i_1}{i_1+i_2})[\bar\p\bg_{i_1},\Ch^{1,0}\bg_{i_2}]^{4,-4} \\
= & \sum_{|I_2| = k} [\bar\p\bg_{i_1},\Ch^{1,0}\bg_{i_2}]^{4,-4}.
\end{align*}

Now, we consider third piece in $\textbf{Eq2}$:
{\fontsize{10}{12}\selectfont
\begin{align*}&\sum_{|I_2| = k}\frac{i_2}{i_2+i_1}\cdot\bar\p([\Ch^{1,0}\bg_{i_1},\bg_{i_2}])^{4,-4} \\
& \overset{\eqref{vanisheq_1}}=  \sum_{|I_2| = k} \frac{i_2}{i_2+i_1} \bar\p([2\bar\eta_{i_1}(\theta^\st),\bg_{i_2}])^{2,-2}+\sum_{|I_3| = k}\big\{\frac{i_2}{i_1+i_2+i_3}\bar\p([-\bar\eta_{i_1}(\bar\p\bg_{i_3})+\frac{1}{2}\frac{i_3}{i_1+i_3}[[\theta,\bg_{i_1}],\bg_{i_3}],\bg_{i_2}])\big\}^{4,-4}\\
& \overset{\eqref{vanisheq_2}}= \sum_{|I_3| = k} \Big\{\frac{i_2+i_3}{i_1+i_2+i_3}\frac{i_2}{i_2+i_3} \bar\p([\bar\eta_{i_1}(\bar\p g_{i_3}),\bg_{i_2}])-\frac{i_2}{i_1+i_2+i_3}\bar\p([\bar\eta_{i_1}(\bar\p\bg_{i_3})-\frac{1}{2}\frac{i_3}{i_1+i_3}[[\theta,\bg_{i_1}],\bg_{i_3}],\bg_{i_2}])\Big\}^{4,-4}\\
& \overset{\ \ \ \ }= \cc \sum_{|I_3| = k}\frac{i_2}{i_1+i_2+i_3}\frac{i_3}{i_1+i_3}\bar\p([[[\theta,\bg_{i_1}],\bg_{i_3}],\bg_{i_2}])^{4,-4}.
\end{align*}}

Substituting this into \eqref{ab(1)2} and then substituting \eqref{ab(1)1} and \eqref{ab(1)2} into \eqref{ab(1)}, we have
\begin{align*}&[\theta^\st,\ob_{k,1}^{2,-2}]+\Ch^{1,0}\ob_{k,2}^{3,-3} \overset{\eqref{ab(1)}}= \textbf{Eq1} + \textbf{Eq2}\\
&\overset{\eqref{ab(1)1}\eqref{ab(1)2}}= -\frac{1}{8}\sum_{|I_2|=k}2\cdot\frac{i_2}{i_2+i_1}\cdot\big\{[[\theta,[\theta^\st,\bg_{i_1}]],\bg_{i_2}]+[[\theta,\bg_{i_1}],[\theta^\st,\bg_{i_2}]]\big\}^{4,-4}\\
&\makebox[1cm]{}-\frac18\sum_{|I_3|=k}\big\{\frac{i_2}{i_1+i_2+i_3}\frac{i_3}{i_1+i_3}\bar\p([[[\theta,\bg_{i_1}],\bg_{i_3}],\bg_{i_2}])+\frac{i_3}{i_1+i_3}[[[\theta,\bg_{i_1}],\bg_{i_3}],\bar\p \bg_{i_2}]\big\}^{4,-4}.
\end{align*}
By \eqref{vanisheq_2}, all the above terms can be reduced to terms expressed by $\bg,\bar\p \bg$ and $\theta$, so we should be able to directly compute that the above expression equals 0. We place this complicated computation in Lemma~\ref{comp1} at the end of this section. Thus we obtain \eqref{ab(1)} for $[(\ob_{k,1}^{2,-2},\ob_{k,2}^{3,-3})]$. This completes the proof of (i).

\medskip

Now we come to (ii). Since the isomonodromic deformation is further modulo-$(t,\bar t^{k+1})$-holomorphic, then by equations \eqref{hol} for gauge transformation we have 
\begin{align}\label{holwt2}[\theta^\st,(\varphi_{k}+\sum_{j=1}^{k-1}\varphi_{j}u_{k-j}+[\theta,u_{k}])^{2,-2}]+\Ch^{1,0}(\psi_{k}+\sum_{j=1}^{k-1}\psi_{j}u_{k-j}+\bar\p u_{k})^{3,-3}=0.
\end{align}
Now we prove \eqref{obnvanish}. By Corollary~\ref{coro_g1}, we may inductively assume for $m=1,\cdots,k-1$: \begin{align}\label{u-wt2}u_{m}^{2,-2}=-\cc(\bg_{m})^{2,-2}+\sum_{j=1}^{m-1}\frac14(1+\frac{m-j}{m})\cdot(\bg_j\bg_{m-j})^{2,-2}.\end{align}
And prove this for $m=k$ and then we use it to prove \eqref{obnvanish}. 

By explicit formulas for $\varphi_i$, $\psi_i$ in \eqref{varphi}, \eqref{psi} and positivity of weights of $\bg_{i}$ in Proposition~\ref{wtprop}, we have for any $i=1,2,\cdots,k-1$
{\small \begin{flalign}
&(\varphi_i)^{1,-1}=\cc[\theta,\bg_i]^{1,-1},\quad (\psi_i)^{2,-2}=\cc(\bar\p\bg_i)^{2,-2}+\text{a section of }\A^{1,0}(\End\E),&\notag\\&(\varphi_k)^{2,-2}=\bar\eta_k(\theta^\st)+\cc[\theta,\bg_k]^{2,-2}-\cc(\Ch^{1,0}\bg_k)^{2,-2}-\cc\sum_{|I_2|=k}(\bg_{i_1}[\theta,\bg_{i_2}])^{2,-2}-\cc\sum_{|I_2|=k}\bar\eta_{i_1}(\bar\p\bg_{i_2}])^{2,-2},&\notag\\
&(\psi_k)^{3,-3}=-\cc \big\{[\theta^\st,\bg_k]^{3,-3}-(\bar\p\bg_k)^{3,-3}+\makebox[-.2cm]{}\sum_{|I_2|=k}(\bg_{i_1}\bar\p\bg_{i_2})^{3,-3}\big\}+\text{a section of }\A^{1,0}(\End\E).&\label{eq_varphipsi_low_wt}
\end{flalign}}
Let $\square(-):=[\theta^\st,[\theta,-]]+\Ch^{1,0}\bar\p(-)$. By \eqref{eq_varphipsi_low_wt} and \eqref{u-wt1}, \eqref{eq_wtphipsi}, the $\mathcal A^{1,1}(\End\E)$-part of the equation \eqref{holwt2} can be rewritten as
\begin{align*}0=&\square\Big(u_{k}+\cc\bg_{k}\Big)^{2,-2}-\sum_{|I_2|=k}[\theta^\st,\cc\bar\eta_{i_1}(\bar\p\bg_{i_2})^{2,-2}]-\cc\sum_{|I_2|=k}[\theta^\st,\big(\bg_{i_1}[\theta,\bg_{i_2}]\big)^{2,-2}]\\
&-\frac{1}{4}\sum_{|I_2|=k}[\theta^\st,\big([\theta,\bg_{i_1}]\bg_{i_2}\big)^{2,-2}]-\cc\sum_{|I_2|=k}\Ch^{1,0}(\bg_{i_1}\bar\p\bg_{i_2})^{3,-3}-\frac{1}{4}\sum_{|I_2|=k}\Ch^{1,0}((\bar\p\bg_{i_1})\bg_{i_2})^{3,-3}.
\end{align*}
Combining the above identity with \eqref{ab(1)}, we have
\begin{align*}\square\Big(u_{k}+\cc\bg_{k}\Big)^{2,-2}=&\frac{1}{4}\sum_{|I_2|=k}(1+\frac{i_2}{i_2+i_1})\Big\{[\theta^\st,[\theta,(\bg_{i_1}\bg_{i_2})^{2,-2}]]+\Ch^{1,0}\bar\p(\bg_{i_1}\bg_{i_2})^{2,-2}\Big\}.\end{align*}
By ``harmonicity'' in Lemma~\ref{vanlem}, we have \eqref{u-wt2} for $m=k$. Substituting \eqref{u-wt1}, \eqref{u-wt2} and \eqref{eq_varphipsi_low_wt} into 
\begin{align*}
\begin{cases}
\big\{\varphi_{k}+\sum\limits_{j=1}^{k-1}\varphi_{j}u_{k-j}+[\theta,u_{k}]\big\}^{2,-2}=0;\\
\big\{\psi_{k}+\sum\limits_{j=1}^{k-1}\psi_{j}u_{k-j}+\bar\p u_{k}-\sum\limits_{j=1}^{k-1}\bar\eta_{j}(\bar\p u_{k-j})\big\}^{3,-3}=0,
\end{cases}
\end{align*}
given by the equations \eqref{hol} for gauge transformation, we have \eqref{obnvanish}. 
\end{proof}

\begin{lemma}For any $\omega_1,\omega_2\in\mathcal A^{0,1}(\End\E)$ and $\bar\eta\in\mathcal A^{1,0}(T^{0,1}({X_0}))$, we have
\begin{align}\label{eq_C2}[\bar\eta(\omega_1),\omega_2]=-[\omega_1,\bar\eta(\omega_2)].
\end{align}
\end{lemma}
\begin{proof}We first remark that $\bar\eta(\omega_1)$ is the contraction and lies in $\A^{1,0}(\End\E)$. Thus
\begin{align*}[\bar\eta(\omega_1),\omega_2]=\bar\eta(\omega_1)\circ\omega_2+\omega_2\circ\bar\eta(\omega_1)=-\omega_1\circ\bar\eta(\omega_2)-\bar\eta(\omega_2)\circ\omega_1=-[\omega_1,\bar\eta(\omega_2)],
\end{align*}
where the second equality follows from $dz_i\wedge d\bar z_j=-d\bar z_j\wedge d z_i$ for any local chart $(z_i)$ of $X_0$.
\end{proof}

\begin{lemma}
For any $g,h\in\A^0(\End\E)$, 
we have 
\begin{equation}\label{eq_C1}\begin{aligned}
&[\theta^\st,  [[\theta,g],h]]  +\Ch^{1,0}([\bar\p g,h])\\
&=  
-[[\theta,[\theta^\st,g]],h]
-[[\theta,g],[\theta^\st,h]]
-\bar\p([\Ch^{1,0}g,h])
-[\bar\p g,\Ch^{1,0}h]
-[\bar\p h,\Ch^{1,0}g]
.\end{aligned} \end{equation}
\end{lemma}
\begin{proof}
By the definition of Lie bracket \eqref{eq_Liebracket_2}, for $\omega_1\in \mathcal A^k(\End\E)$, $\omega_2\in \mathcal A^l(\End\E)$ and $\omega_3\in \mathcal A^j(\End\E)$, we have the following Jacobi identity
\begin{align}\label{eq_Liebracket_3}
(-1)^{kl} \bigl[ \omega_1, [\omega_2, \omega_3] \bigr]
+ (-1)^{lj} \bigl[ \omega_2, [\omega_3, \omega_1] \bigr]
+ (-1)^{jk} \bigl[ \omega_3, [\omega_1, \omega_2] \bigr] = 0.
\end{align}
By taking $\omega_1=\theta^\st,\ \omega_2=[\theta,g]$ and $\omega_3=h$ in the Jacobi identity \eqref{eq_Liebracket_3}, we have
 \begin{align}\label{eq_firststep_C1}
 [\theta^\st,  [[\theta,g],h]]=[[\theta,g],[h,\theta^\st]]+[h,[\theta^\st,[\theta,g]]].
 \end{align}
By applying the Jacobi identity \eqref{eq_Liebracket_3} again to $[\theta^\st,[\theta,g]]$, we have
 \begin{align*}[h,[\theta^\st,[\theta,g]]]=[h,[\theta,[\theta^\st,g]]]+[h,[g,[\theta^\st,\theta]]].
 \end{align*}
 Substituting this into \eqref{eq_firststep_C1}, we have
 \begin{align*} [\theta^\st,  [[\theta,g],h]]=&-[[\theta,g],[\theta^\st,h]]-[[\theta,[\theta^\st,g]],h]+[h,[g,[\theta^\st,\theta]]]\\
 =&-[[\theta,g],[\theta^\st,h]]-[[\theta,[\theta^\st,g]],h]+[h,(\Ch^{1,0}\bar\p+\bar\p\Ch^{1,0})g]
 \end{align*}
by the identity  
$[[\theta,\theta^\st],g]=-F(\Ch)g=-(\Ch^{1,0}\bar\p+\bar\p\Ch^{1,0})g$ which follows from the harmonic metric equation. Therefore
 \begin{align*} &[\theta^\st,  [[\theta,g],h]]+\Ch^{1,0}([\bar\p g,h])=-[[\theta,g],[\theta^\st,h]]-[[\theta,[\theta^\st,g]],h]-[(\Ch^{1,0}\bar\p+\bar\p\Ch^{1,0})g,h]+\Ch^{1,0}([\bar\p g,h])\\
 &=-[[\theta,g],[\theta^\st,h]]-[[\theta,[\theta^\st,g]],h]-[\bar\p\Ch^{1,0}g,h]-[\bar\p g,\Ch^{1,0}h]\\
 &= 
-[[\theta,[\theta^\st,g]],h]
-[[\theta,g],[\theta^\st,h]]
-\bar\p([\Ch^{1,0}g,h])
-[\bar\p g,\Ch^{1,0}h]
-[\bar\p h,\Ch^{1,0}g]. \qedhere
 \end{align*}
 \end{proof}

\begin{lemma}\label{comp1}
Assuming \eqref{vanisheq_2} for $m=1,\cdots,k-1$, we have
\begin{align*} &\sum_{|I_2|=k}2\cdot\frac{i_2}{i_2+i_1}\cdot\big\{[[\theta,[\theta^\st,\bg_{i_1}]],\bg_{i_2}]+[[\theta,\bg_{i_1}],[\theta^\st,\bg_{i_2}]]\big\}^{4,-4}\\
&\sum_{|I_3|=k}\big\{\frac{i_2}{i_1+i_2+i_3}\frac{i_3}{i_1+i_3}\bar\p([[[\theta,\bg_{i_1}],\bg_{i_3}],\bg_{i_2}])+\frac{i_3}{i_1+i_3}[[[\theta,\bg_{i_1}],\bg_{i_3}],\bar\p \bg_{i_2}]\big\}^{4,-4}=0.
\end{align*}
\end{lemma}

\begin{proof}
Using the Leibniz rule of $\bar\p$, we have
\begin{align*}&\sum_{|I_3|=k}\big\{\frac{i_2}{i_1+i_2+i_3}\frac{i_3}{i_1+i_3}\bar\p([[[\theta,\bg_{i_1}],\bg_{i_3}],\bg_{i_2}])+\frac{i_3}{i_1+i_3}[[[\theta,\bg_{i_1}],\bg_{i_3}],\bar\p \bg_{i_2}]\big\}^{4,-4}\\
&=-\sum_{|I_3|=k}\Big(\frac{i_2}{i_1+i_2+i_3}\frac{i_3}{i_1+i_3}[[[\theta,\bar\p\bg_{i_1}],\bg_{i_3}],\bg_{i_2}]
+\frac{i_2}{i_1+i_2+i_3}\frac{i_3}{i_1+i_3}[[[\theta,\bg_{i_1}],\bar\p\bg_{i_3}],\bg_{i_2}]
\\&\makebox[1.7cm]{}+(\frac{i_2}{i_1+i_2+i_3}\frac{i_3}{i_1+i_3}-\frac{i_3}{i_1+i_3})[[[\theta,\bg_{i_1}],\bg_{i_3}],\bar\p\bg_{i_2}]\Big)^{4,-4}
\\&=-\sum_{|I_3|=k}\Big(\frac{i_2}{i_1+i_2+i_3}\frac{i_3}{i_1+i_3}[[[\bar\p\bg_{i_1},\bg_{i_3}],\theta],\bg_{i_2}]+\frac{i_2}{i_1+i_2+i_3}\frac{i_3}{i_1+i_3}[[[\theta,\bg_{i_3}],\bar\p\bg_{i_1}],\bg_{i_2}]
\\&\makebox[1.7cm]{}+\frac{i_2}{i_1+i_2+i_3}\frac{i_3}{i_1+i_3}[[[\theta,\bg_{i_1}],\bar\p\bg_{i_3}],\bg_{i_2}]-\frac{i_3}{i_1+i_2+i_3}[[\bar\p\bg_{i_2},[\theta,\bg_{i_1}]],\bg_{i_3}]
\\&\makebox[1.7cm]{}-\frac{i_3}{i_1+i_2+i_3}[[\bg_{i_3},\bar\p\bg_{i_2}],[\theta,\bg_{i_1}]]\Big)^{4,-4}\quad\text{(where we use the Jacobi identity \eqref{eq_Liebracket_3})}
\\
&\stackrel{\eqref{vanisheq_2}}{=}-\sum_{|I_2|=k}2\cdot\frac{i_2}{i_2+i_1}\cdot\big\{[[\theta,[\theta^\st,\bg_{i_1}]],\bg_{i_2}]+[[\theta,\bg_{i_1}],[\theta^\st,\bg_{i_2}]]\big\}^{4,-4}.\qedhere
\end{align*}
\end{proof}

\newpage

\section{Holomorphicity and the gauge transformation}
\label{sec_holo_gauge_tran}

In Lemma~\ref{lem_gauge}, modulo-$(t)$-holomorphicity is equivalent to solvability of those equations in \eqref{hol} for gauge transformation. As a consequence of this solvability, we get the partial equations \eqref{vanisheq_1} and \eqref{vanisheq_2} for the deformed harmonic metric in Proposition~\ref{prop_ob}. In this section, We shall use them to obtain a recursive formula for the deformed harmonic metric along pure anti-holomorphic direction. Precisely, see the main result Proposition~\ref{1,-1}.

\noindent \textbf{Assumption.}
Let $\xi_{1},\cdots,\xi_{N}$ be independent continuous random variables with distributions  $$\operatorname{Beta}(i_1,1),\cdots,\operatorname{Beta}(i_N,1)$$
respectively, where $i_1,\cdots,i_N$ are positive integers. (For background on probability theory, see \cite{Ach}; for the definition of the Beta distribution, see \cite[Example 1.107]{Ach}.) 
By continuity, $\mathbb P(\xi_i = \xi_j) = 0$ for all $i \neq j$, so we may assume without loss of generality that the $\xi_i$ take pairwise distinct values.

Before presenting the recursive formula, we introduce a sequence of positive rational numbers
\begin{equation}\label{bN}
\begin{aligned}
b_{i_1,\cdots,i_N}:=\frac{1}{2^{N-1}}\bigl(&\mathbb P(\xi_{1}<\xi_{2}<\cdots<\xi_{N})+\mathbb P(\xi_{1}<\xi_{2}<\cdots<\xi_{{N-1}}>\xi_N)\\
&+\cdots+\mathbb P(\xi_{1}>\xi_{2}>\cdots>\xi_{N})\bigr).
\end{aligned}\end{equation}
For example
\begin{align*}
b_{i_1}=&1,\\
b_{i_1,i_2}=&\frac12,\\
b_{i_1,i_2,i_3}=&\frac14\bigl(\mathbb P(\xi_{1}<\xi_{2}<\xi_{3})+\mathbb P(\xi_{1}<\xi_{2}>\xi_{3})+\mathbb P(\xi_{1}>\xi_{2}>\xi_{3})\bigr),\\
b_{i_1,i_2,i_3,i_4}=&\frac18\bigl(\mathbb P(\xi_{1}<\xi_{2}<\xi_{3}<\xi_{4})+\mathbb P(\xi_{1}<\xi_{2}<\xi_{3}>\xi_{4})\\
&+\mathbb P(\xi_{1}<\xi_{2}>\xi_{3}>\xi_{4})+\mathbb P(\xi_{1}>\xi_{2}>\xi_{3}>\xi_{4})\bigr).
\end{align*}

An explicit expression of the probability $\mathbb P(\xi_1<\cdots<\xi_{l-1}<\xi_l>\xi_{l+1}>\cdots>\xi_N)$ can be found in \eqref{probmountain}. Hence, this gives an explicit expression of the sequence $b_{i_1,\cdots,i_N}$.

This sequence emerges naturally from the following question: what conditions must $\bg_k$ satisfy to guarantee modulo-$(t)$-holomorphicity? In the following Proposition~\ref{1,-1}, we will determine these conditions in a recursive way by solving \eqref{hol}, which fully determine $g_k^\st$:  in summary 
\begin{enumerate}
\item $\bg_k\in \A^0(\bigoplus_{l>0}(\End\E)^{l,-l})$ by Proposition~\ref{wtprop}; 
\item Equations \eqref{vanisheq_1} and \eqref{vanisheq_2} determine $(\bg_k)^{1,-1}$; 
\item Equation \eqref{thmbg} below determines  $(\bg_k)^{N,-N},\ N\geq 2.$
\end{enumerate}
\begin{proposition}[Recursive formula] \label{1,-1}
Let $(\E,\bar\p,\theta)$ be a graded stable Higgs bundle of weight $w$ on $X_0$. Let $(\E,\bar\p_t,\theta_t)$ be the isomonodromic deformation in Definition~\ref{def_real_ana_deform} of $(\E,\bar\p,\theta)$ on $X_0$ along $X_n$.
If the deformation is modulo-$(t)$-holomorphic,
then for any $k=1,2,\cdots,n$ and any positive integer $N$, we have
\begin{align}
& (\bg_k)^{N,-N}  = \sum_{|I_N|= k} b_{i_1,i_2,\cdots,i_N}(\bg_{i_1}\bg_{i_2}\cdots\bg_{i_N})^{N,-N}. \label{thmbg}\\
& u_k^{N,-N}  = \frac{(-1)^N}{2^N}\sum_{|I_N| = k} \mathbb P(\xi_1<\xi_2<\cdots<\xi_N)(\bg_{i_1}\bg_{i_2}\cdots\bg_{i_N})^{N,-N}, \label{thmu}
\end{align}
where $u_k$ is defined in the equations for gauge transformation \eqref{hol}.
\end{proposition}

\begin{remark}
(1). Since $\bg_k\in \A^0(\bigoplus_{l>0}(\End\E)^{l,-l})$ by Proposition~\ref{wtprop}, we have 
\[(\bg_{i_1}\bg_{i_2}\cdots\bg_{i_N})^{N,-N}=(\bg_{i_1})^{1,-1}\cdots(\bg_{i_N})^{1,-1}\]
Thus each $\bg_k$ can be expressed in terms of $(\bg_1)^{1,-1},\cdots,(\bg_{k-1})^{1,-1}$.

(2). When $N>\min(w,k)$, the right hand sides of \eqref{thmbg} and \eqref{thmu} are both zero.
\end{remark}

\bigskip

We introduce the following notation, which will greatly simplify the writing of the expressions:
\begin{align}\label{eq_Lieopera}
\mathcal L_J(-):=[\cdots[-,x_1],x_2],\cdots],x_N],
\end{align}
where $J=(x_1,x_2 ,\cdots, x_N )$ is a ordered set with $x_1,\cdots,x_N\in \A^0(\End\E)$. Let 

\begin{align*}\Lambda_N^\st:=(\bg_{i_2}, \bg_{i_3},\cdots,\bg_{i_N})\end{align*}
be a ordered set of $\bg$. Besides, recall the notation \eqref{eq_altersum_S} and we have
\begin{align*}S_k^\st=\sum_{m=1}^k\sum_{i_1+i_2+\cdots+i_m=k}(-1)^{m-1}\bg_{i_1}\bg_{i_2}\cdots \bg_{i_m}.
\end{align*}

We firstly simplify a term appeared in the explicit formula \eqref{varphi} for $\varphi_i$ in the following lemma.
\begin{lemma} \label{recursivelem} 
Assuming \eqref{thmbg} holds for $k=1,2,\cdots,n-1$ and any positive integer, we have 
{\fontsize{10}{12}\selectfont
\begin{equation}\label{eqb1} 
\begin{aligned}
   \sum_{ |I_2| =i} (S^\st_{i_1}\Ch^{1,0}\bg_{i_2})^{N+1,-N-1} 
=   \sum_{|I_N|=n}{b}_{i_1,\cdots,i_N}
\left\{
\Ch^{1,0}(\bg_{i_1}\cdots\bg_{i_N})
-
\mathcal L_{\Lambda_N^\st} \Ch^{1,0}\bg_{i_1} 
\right\}^{N+1,-N-1}.
\end{aligned}\end{equation}}
\end{lemma}

We expand the iterated Lie brackets appeared in \eqref{eqb1}.
\begin{lemma} \label{brackets}
For any $j\in\{1,2,\cdots,N\}$, let $A_{j;N}'$ and $V_{j;N}$ be the sets defined in Definition~\ref{AkN}. Then we have

\begin{align*}
&-\sum_{|I_N|=n}({b}_{i_1,\cdots,i_N}\mathcal L_{\Lambda_N^\st} \Ch^{1,0}\bg_{i_1} )^{N+1,-N-1}\\
=&\sum_{|I_N|=n}\sum_{j=1}^N(-1)^j b_{i_1,\cdots,i_N}\cdot\sum_{\sigma\in V_{j;N}}(\bg_{i_{\sigma(1)}}\cdots\bg_{i_{\sigma(j-1)}}(\Ch^{1,0}\bg_{i_{\sigma(j)}})\bg_{i_{\sigma(j+1)}}\cdots\bg_{i_{\sigma(N)}})^{N+1,-N-1}\\
=&\sum_{|I_N|=n}\sum_{j=1}^N(-1)^j (\sum_{\sigma\in A'_{j;N}}b_{i_{\sigma(1)},i_{\sigma(2)},\cdots,i_{\sigma(N)}})(\bg_{i_1}\cdots\bg_{i_{j-1}}(\Ch^{1,0}\bg_{i_j})\bg_{i_{j+1}}\cdots\bg_{i_N})^{N+1,-N-1}.
\end{align*}
\end{lemma}
\begin{proof}We only prove the first equality as the second follows directly by reordering the summation indices. For the first one, we prove inductively on $N\geq 1$ that \begin{align*}&-\mathcal L_{\Lambda_N^\st} \Ch^{1,0}\bg_{i_1} =\sum_{j=1}^N(-1)^j \sum_{\sigma\in V_{j;N}}(\bg_{i_{\sigma(1)}}\cdots\bg_{i_{\sigma(j-1)}}(\Ch^{1,0}\bg_{i_{\sigma(j)}})\bg_{i_{\sigma(j+1)}}\cdots\bg_{i_{\sigma(N)}}).
\end{align*}
For $N=1,2$, the above identity holds trivially. Assume inductively it holds for $N\geq 2$, and we prove it for $N+1$. By taking $[-,\bg_{i_{N+1}}]$ of the above identity, we have\begin{align*}&-\mathcal L_{\Lambda_{N+1}^\st}\Ch^{1,0}\bg_{i_1} 
=\sum_{j=1}^N(-1)^j \sum_{\sigma\in V_{j;N}}(\bg_{i_{\sigma(1)}}\cdots\bg_{i_{\sigma(j-1)}}(\Ch^{1,0}\bg_{i_{\sigma(j)}})\bg_{i_{\sigma(j+1)}}\cdots\bg_{i_{\sigma(N)}})\bg_{i_{N+1}}\\
&\makebox[3cm]{}-\sum_{j=1}^N(-1)^j\sum_{\sigma\in V_{j;N+1}}\bg_{i_{N+1}}(\bg_{i_{\sigma(1)}}\cdots\bg_{i_{\sigma(j-1)}}(\Ch^{1,0}\bg_{i_{\sigma(j)}})\bg_{i_{\sigma(j+1)}}\cdots\bg_{i_{\sigma(N)}}),
\end{align*}
and one can verify the right hand side of the above is exactly 
\[\sum_{j=1}^{N+1}(-1)^j \sum_{\sigma\in V_{j;N+1}}(\bg_{i_{\sigma(1)}}\cdots\bg_{i_{\sigma(j-1)}}(\Ch^{1,0}\bg_{i_{\sigma(j)}})\bg_{i_{\sigma(j+1)}}\cdots\bg_{i_{\sigma(N)}}\bg_{i_{\sigma(N+1)}}). \qedhere\]
\end{proof}
\begin{remark}By a similar argument as in the proof of Lemma~\ref{brackets}, we have for any sequence $a_{i_1,i_2,\cdots,i_N}$
\begin{equation}\label{eq_iterated_brac}\begin{aligned}
&\sum_{|I_{N+1}|=n} a_{i_1,i_2,\cdots,i_N}(\mathcal L_{\Lambda_{N+1}^\st} [\theta,\bg_{i_1}])^{N+1,-N-1}\\
&=\sum_{|I_{N+1}|=n}\sum_{l=1}^{N+1}(-1)^{l-1}(\sum_{\sigma\in A'_{l;N+1}}a_{i_{\sigma(1)},i_{\sigma(2)},\cdots,i_{\sigma(N+1)}))}(\bg_{i_1}\cdots[\theta,\bg_{i_l}]\cdots\bg_{i_{N+1}})^{N+1,-N-1}.
\end{aligned}\end{equation}
\end{remark}

\begin{proof}[Proof of Lemma~\ref{recursivelem}] 

All three terms in \eqref{eqb1} can be written as a linear combination of 
\begin{align*}&(\bg_{i_1}\cdots\bg_{i_{j-1}}(\Ch^{1,0}\bg_{i_j})\bg_{i_{j+1}}\cdots\bg_{i_N})^{N+1,-N-1}\\&=(\bg_{i_1})^{1,-1}\cdots(\bg_{i_{j-1}})^{1,-1}\big(\Ch^{1,0}(\bg_{i_j})^{1,-1}\big)(\bg_{i_{j+1}})^{1,-1}\cdots(\bg_{i_N})^{1,-1}\end{align*}
by the Leibniz rule of $\Ch^{1,0}$. In the following we compare their coefficients on both sides to prove \eqref{eqb1}. The main ingredient in this proof is the combinatorial property of $b_{i_1,\cdots,i_N}$ given in Proposition~\ref{bkbN-k}.

We firstly simplify the term in the left hand side of \eqref{eqb1}:
\begin{align} \label{eq_pf43_1}
& \sum_{ |I_2| =i} (S^\st_{i_1}\Ch^{1,0}\bg_{i_2})^{N+1,-N-1}
=\sum_{|I_2|=n}\sum_{j=1}^{N-1}(S^\st_{i_1})^{j,-j}(\Ch^{1,0}\bg_{i_2})^{N+1-j,-N-1+j}.
\end{align}

for any positive integer $j$ and $k=1,2,\cdots,n-1$. We prove by induction on $j$ that \begin{align}\label{SjN}(S^\st_k)^{j,-j}=(-1)^{j-1}(\bg_k)^{j,-j}.\end{align}
For $j=1,2$, \eqref{SjN} is clear by \eqref{thmbg}. Now we assume \eqref{SjN} holds for $2\leq j<N$ and prove it for $(S^\st_k)^{j+1,-j-1}$. If $j+1$ is odd, then 
\begin{align*}&(S^\st_k)^{j+1,-j-1}=(\bg_k)^{j+1,-j-1}+\sum_{l=1}^{j}\sum_{|I_2|= k} (\bg_{i_1})^{l,-l}(-S^\st_{i_2})^{j+1-l,-j-1+l}
\\&=(\bg_k)^{j+1,-j-1}+\sum_{l=1}^{j}\sum_{|I_2| = k} (-1)^{j+1-l}(\bg_{i_1})^{l,-l}(\bg_{i_2})^{j+1-l,-j-l+l}\\
&\overset{\eqref{thmbg}}=(\bg_k)^{j+1,-j-1}+\sum_{l=1}^{j}\sum_{|I_j|= k} (-1)^{j+1-l}(b_{i_{1},\cdots,i_l}\bg_{i_1}\cdots\bg_{i_l})^{l,-l}(b_{i_{l+1},\cdots,i_j}\bg_{i_{l+1}}\cdots\bg_{i_{j}})^{j+1-l,-j-l+l}
\\&\overset{\eqref{eq_alt_sum_2}}=(\bg_k)^{j+1,-j-1}.
\end{align*}
 Similar, if $j+1$ is even, we can also prove \eqref{SjN} by using \eqref{eq_alt_sum_1} in Corollary~\ref{altersum}. 
 
 Substituting \eqref{SjN} into \eqref{eq_pf43_1}, we obtain \begin{flalign}\label{coeffaltersum}
&\sum_{ |I_2|=i} (S^\st_{i_1}\Ch^{1,0}\bg_{i_2})^{N+1,-N-1}&\notag \\&=\sum_{|I_N|=n}\sum_{j=1}^{N-1}(-1)^{j-1}b_{i_1,\cdots,i_j}b_{i_{j+1},\cdots,i_N}\big(\bg_{i_1}\cdots \bg_{i_j}\Ch^{1,0}(\bg_{i_{j+1}}\cdots \bg_{i_N})\big)^{N+1,-N-1}&\notag\\
&=
\sum_{|I_N|=n}
\sum_{j=1}^{N}
\sum_{l=1}^{j-1}(-1)^{l-1}
b_{i_1,\cdots,i_l}b_{i_{l+1},\cdots,i_N}
(\bg_{i_1}\cdots\bg_{i_{j-1}}(\Ch^{1,0}\bg_{i_j})\bg_{i_{j+1}}\cdots\bg_{i_N})^{N+1,-N-1}.&
\end{flalign}
by expanding $\Ch^{1,0}(\bg_{i_{j+1}}\cdots \bg_{i_N})$ by Leibniz rule. 
To prove \eqref{eqb1}, it suffices to prove the coefficient of \begin{align*}(\bg_{i_1}\cdots\bg_{i_{j-1}}(\Ch^{1,0}\bg_{i_j})\bg_{i_{j+1}}\cdots\bg_{i_N})^{N+1,-N-1}\end{align*}
on both sides of \eqref{eqb1} are equal. Using \eqref{coeffaltersum} and Lemma~\ref{brackets}, we only need to prove for any $j=1,2,\cdots,N$
\begin{align*}\sum_{l=1}^{j-1}(-1)^{l-1}b_{i_1,\cdots,i_l}b_{i_{l+1},\cdots,i_N}=
b_{i_1,\cdots,i_N}-(-1)^j\sum_{\sigma\in A'_{j;N}}b_{i_{\sigma(1)},\cdots,i_{\sigma(N)}}.
\end{align*}
Note that the above holds by \eqref{redeqb1}.
\end{proof}

\begin{lemma}\label{lem_recur2}
Assuming \eqref{thmbg} holds for $k=1,2,\cdots,n-1$ and any positive integer $N$, we have 
{\fontsize{10}{12}\selectfont
\begin{flalign}
&\text{(i)}\quad \sum_{|I_2| = n}\big(S_{i_1}^\st[\theta,\bg_{i_2}]\big)^{N+1,-N-1}&\notag\\
&\qquad=\sum_{|I_{N+1}|=n} b_{i_1,\cdots,i_{N+1}}([\theta,\bg_{i_1}\cdots\bg_{i_{N+1}}]-\mathcal L_{\Lambda_{N+1}^\st} [\theta,\bg_{i_1}])^{N+1,-N-1}.& \label{eqb2}
\\[6pt]
&\text{(ii)}\quad \Big\{ -[\theta^\st,\sum_{|I_{N-1}|=n} b_{i_1,\cdots,i_{N-1}}\bg_{i_1}\cdots\bg_{i_{N-1}}]+\bar\p(\sum_{|I_{N}|=n} b_{i_1,\cdots,i_{N}}\bg_{i_1}\cdots\bg_{i_{N}}]) &\notag\\
&\qquad\makebox[5cm]{} +\sum_{i_1 + i_2 = n}\Big(S^\st_{i_1}([\theta^\st,\bg_{i_2}]-\bar\p\bg_{i_2})\Big)\Big\}^{N+1,-N-1}& \notag\\
&\qquad=\Big\{\sum_{|I_{N}|=n} b_{i_1,\cdots,i_{N}}\mathcal L_{\Lambda_N^\st}\bar\p\bg_{i_1}-\sum_{|I_{N-1}|=n} b_{i_1,i_2,\cdots,i_{N-1}}\mathcal L_{\Lambda_{N-1}^\st}[\theta^\st,\bg_{i_1}]\Big\}^{N+1,-N-1}.& \label{eqb3}
\\[6pt]
&\text{(iii)}\quad \big\{\sum_{|I_2|= n}\bar\eta_{i_1}\big([\theta^\st,\bg_{i_2}]-\bar\p\bg_{i_2}\big)-\sum_{|I_3|= n}\bar\eta_{i_1}\Big(S_{i_2}^\st([\theta^\st,\bg_{i_3}]-\bar\p\bg_{i_3})\Big)\big\}^{N+1,-N-1} \notag\\
&\qquad=\Big\{\sum_{|I_{N}|=n} b_{i_1,\cdots,i_{N-1}}\bar\eta_{i_N}\mathcal L_{\Lambda_{N-1}^\st} [\theta^\st,\bg_{i_1}]-\sum_{|I_{N+1}|=n} {b}_{i_1,\cdots,i_{N}}\bar\eta_{i_{N+1}}\mathcal L_{\Lambda_N^\st} \bar\p\bg_{i_1}\Big\}^{N+1,-N-1}.& \label{eqb4}
\end{flalign}
}
\end{lemma}
\begin{proof}
(i) Note that in the proof of \eqref{eqb1}, we crucially the $\mathbb C$-linearity and the Leibniz rule of $\Ch^{1,0}$. And here $[\theta,-]$ also satisfies the $\mathbb C$-linearity and the Leibniz rule. Hence we just run a similar argument to prove \eqref{eqb4}.

\noindent (ii)(iii) Note that $[\theta^\st,-]$, $\bar\p(-)$ (and $\bar\eta\circ [\theta^\st,-]$, $\bar\eta\circ\bar\p(-)$) also satisfy the $\mathbb C$-linearity and the Leibniz rule. By a similar argument as in the proof of \eqref{eqb1}, we have \eqref{eqb3} (and \eqref{eqb4}).
\end{proof}

Now we use Lemma~\ref{recursivelem} and Lemma~\ref{lem_recur2} to simplify the expressions of $\varphi_n,\psi_n$ in \eqref{varphi} and \eqref{psi} under the modulo-$(t)$-holomorphicity condition. 
\begin{lemma}[Simplified expressions of $\varphi_n,\psi_n$]\label{lem_simp_expre} 
Under the assumption of Proposition~\ref{1,-1} and we assume \eqref{thmbg} holds for $k=1,2,\cdots,n-1$ and any positive integer $N$. Then for any non-negative integer $N$
{\small \begin{flalign}\label{varphired}\varphi_n^{N+1,-N-1}=&\cc[\theta,\bg_n-\sum_{|I_{N+1}|=n} b_{i_1,\cdots,i_{N+1}}\bg_{i_1}\cdots\bg_{i_{N+1}}]^{N+1,-N-1}&\notag\\
-&\cc\Big(\Ch^{1,0}(\bg_n-\sum_{|I_{N}|=n} b_{i_1,\cdots,i_{N}}\bg_{i_1}\cdots\bg_{i_{N}})\Big)^{N+1,-N-1}&\\
+&\frac{1}{2^{N+1}}\sum_{|I_{N+1}|=n}\sum_{l=1}^{N+1}(-1)^{l-1}\mathbb P(\xi_1<\cdots<\xi_l>\cdots>\xi_{N+1})(\bg_{i_1}\cdots[\theta,\bg_{i_l}]\cdots\bg_{i_{N+1}})^{N+1,-N-1}\notag&
\end{flalign}}
and we have for any positive integer $N$

{\small \begin{flalign}\label{psired}
\psi_n^{N+1,-N-1}=&\cc\bar\p(\bg_n-\sum_{|I_{N}|=n} b_{i_1,\cdots,i_{N}}\bg_{i_1}\cdots\bg_{i_{N}})^{N+1,-N-1}&\notag\\
&-\cc[\theta^\st,\bg_n-\sum_{|I_{N+1}|=n} b_{i_1,\cdots,i_{N-1}}\bg_{i_1}\cdots\bg_{i_{N-1}}]^{N+1,-N-1}&\\
&+\frac{1}{2^{N}}\sum_{|I_{N}|=n}\sum_{l=1}^{N}(-1)^{l-1}\mathbb P(\xi_1<\cdots<\xi_l>\cdots>\xi_{N}) (\bg_{i_1}\cdots(\bar\p\bg_{i_l})\cdots\bg_{i_{N}})^{N+1,-N-1}& \notag\\
&+\frac{1}{2^{N}}\sum_{|I_{N+1}|=n}\sum_{l=1}^{N}(-1)^{l}\mathbb P(\xi_1<\cdots<\xi_l>\cdots>\xi_{N})(\bg_{i_1}\cdots(\bar \eta_{i_{N+1}}(\bar\p\bg_{i_l}))\cdots\bg_{i_{N}})^{N+1,-N-1}.&\notag
\end{flalign}}
\end{lemma}

\begin{proof}
Substituting \eqref{eqb1}, \eqref{eqb2} and \eqref{eqb4} into explicit formula \eqref{varphi} of $\varphi_n$, we have
{\small \begin{align*}(\varphi_n)^{N+1,-N-1}=&\cc\Big\{[\theta,\bg_n-\sum_{|I_{N+1}|=n} b_{i_1,\cdots,i_{N+1}}\bg_{i_1}\cdots\bg_{i_{N+1}}]
\\&-\Ch^{1,0}(\bg_n-\sum_{|I_{N}|=n} b_{i_1,\cdots,i_N}\bg_{i_1}\cdots\bg_{i_N})\Big\}^{N+1,-N-1}\\
&+\cc\Big\{\sum_{|I_{N+1}|=n} b_{i_1,\cdots,i_{N+1}}(\mathcal L_{\Lambda_{N+1}^\st} [\theta,\bg_{i_1}])-\sum_{|I_{N}|=n}{b}_{i_1,\cdots,i_N}
(
\mathcal L_{\Lambda_N^\st} \Ch^{1,0}\bg_{i_1} 
)\Big\}^{N+1,-N-1}\\
&+\cc\Big\{\sum_{|I_{N}|=n} b_{i_1,\cdots,i_{N-1}}\bar\eta_{i_N}(\mathcal L_{\Lambda_{N-1}^\st} [\theta^\st,\bg_{i_1}])-\sum_{|I_{N+1}|=n} {b}_{i_1,\cdots,i_{N}}\bar\eta_{i_{N+1}}(\mathcal L_{\Lambda_N^\st} \bar\p\bg_{i_1})\Big\}^{N+1,-N-1}.
\end{align*}}
The last two lines of the above expression are denoted as \textbf{Eq3}. To prove \eqref{varphired}, it suffices to prove 
\[\textbf{Eq3}=\frac{1}{2^{N+1}}\sum_{|I_{N+1}|=n}\sum_{l=1}^{N+1}(-1)^{l-1}\mathbb P(\xi_1<\cdots<\xi_l>\cdots>\xi_{N+1})(\bg_{i_1}\cdots[\theta,\bg_{i_l}]\cdots\bg_{i_{N+1}})^{N+1,-N-1}.
\]

Substituting partial equation \eqref{vanisheq_2} into $[\theta^\st,\bg_{i_1}]$ below, we have
\begin{align*}&\sum_{|I_{N}|=n} b_{i_1,\cdots,i_{N-1}}\bar\eta_{i_N}(\mathcal L_{\Lambda_{N-1}^\st}[\theta^\st,\bg_{i_1}])^{N+1,-N-1}
\\&=\sum_{|I_{N+1}|=n}b_{i_1+i_2,i_3,\cdots,i_{N}}\cdot\frac{i_2}{i_2+i_1}\cdot\bar\eta_{i_{N+1}}(\mathcal L_{\Lambda_N^\st}\bar\p\bg_{i_1})^{N+1,-N-1},
\end{align*}
and using partial equations \eqref{vanisheq_1} and \eqref{vanisheq_2}, we have
\begin{align*}&-\sum_{|I_{N}|=n}{b}_{i_1,\cdots,i_N}
(
\mathcal L_{\Lambda_N^\st}\Ch^{1,0}\bg_{i_1} 
)^{N+1,-N-1}\\
&=\sum_{|I_{N+1}|=n} \Big\{-\frac{i_2}{i_2+i_1}\cdot b_{i_{N+1},i_1+i_2,i_3\cdots,i_{N}}\bar\eta_{i_{N+1}}(\mathcal L_{\Lambda_N^\st}\bar\p\bg_{i_1})+b_{i_{N+1}+i_1,i_2,\cdots,i_{N}}\bar\eta_{i_{N+1}}(\mathcal L_{\Lambda_N^\st}\bar\p\bg_{i_1})\\&\makebox[2cm]{}-\cc\cdot \frac{i_2}{i_2+i_1}\cdot b_{i_1+i_2,i_3,\cdots,i_{N+1}}\mathcal L_{\Lambda_{N+1}^\st} [\theta,\bg_{i_1}]\Big\}^{N+1,-N-1}.
\end{align*}
Substituting the above two identities into \textbf{Eq3}, we have
\begin{align*}\textbf{Eq3}=\cc\cdot&\Big\{\sum_{|I_{N+1}|=n}(b_{i_{N+1}+i_1,i_2,\cdots,i_{N}}-{b}_{i_1,i_2,\cdots,i_{N}}+ \frac{i_2}{i_2+i_1}\cdot(\cc\cdot b_{{i_1+i_2},i_3,\cdots,i_{N}}-b_{i_{N+1},i_1+i_2,i_3,\cdots,i_{N}}))\times
\\&\bar\eta_{i_{N+1}} (\mathcal L_{\Lambda_N^\st}\bar\p\bg_{i_1})
+\sum_{|I_{N+1}|=n} (b_{i_1,\cdots,i_{N+1}}-\frac12\cdot \frac{i_2}{i_2+i_1}\cdot b_{i_1+i_2,i_3,\cdots,i_{N+1}})(\mathcal L_{\Lambda_{N+1}^\st} [\theta,\bg_{i_1}])\Big\}^{N+1,-N-1}\end{align*}
Note that \eqref{recursiveb} gives: $b_{i_{N+1}+i_1,i_2,\cdots,i_{N}}-{b}_{i_1,i_2,\cdots,i_{N}}+ \frac{i_2}{i_2+i_1}\cdot(\cc\cdot b_{{i_1+i_2},i_3,\cdots,i_{N}}-b_{i_{N+1},i_1+i_2,i_3,\cdots,i_{N}})=0$, and thus we have
\begin{align*}
 \textbf{Eq3} \overset{\eqref{recursiveb}}= \cc\sum_{|I_{N+1}|=n} (b_{i_1,\cdots,i_{N+1}}-\frac12\cdot  \frac{i_2}{i_2+i_1}\cdot b_{i_1+i_2,i_3,\cdots,i_{N+1}})(\mathcal L_{\Lambda_{N+1}^\st} [\theta,\bg_{i_1}])^{N+1,-N-1}.
\end{align*}
Note that \eqref{abN} gives: $b_{i_1,\cdots,i_{N+1}}-\frac12 \cdot\frac{i_2}{i_2+i_1}\cdot b_{i_1+i_2,i_3,\cdots,i_{N+1}}=\frac{1}{2^{N}} \mathbb P(\xi_1>\xi_2>\cdots>\xi_{N+1})$, and thus we have
\begin{align*}& \textbf{Eq3} \overset{\eqref{abN}}= \frac{1}{2^{N+1}}\sum_{|I_{N+1}|=n} \mathbb P(\xi_1>\xi_2>\cdots>\xi_{N+1})(\mathcal L_{\Lambda_{N+1}^\st} [\theta,\bg_{i_1}])^{N+1,-N-1}\\
&\overset{\eqref{eq_iterated_brac}}=\sum_{|I_{N+1}|=n}\sum_{l=1}^{N+1}\frac{(-1)^{l-1}}{2^{N+1}}(\sum_{\sigma\in A'_{l;N+1}}\mathbb P(\xi_{\sigma(1)}>\xi_{\sigma(2)}>\cdots>\xi_{\sigma(N+1)}))(\bg_{i_1}\cdots[\theta,\bg_{i_l}]\cdots\bg_{i_{N+1}})^{N+1,-N-1}.
\end{align*}
Note that \eqref{2ineq} gives: $\sum\limits_{\sigma\in A'_{l;N+1}}\mathbb P(\xi_{\sigma(1)}>\xi_{\sigma(2)}>\cdots>\xi_{\sigma(N+1)}))=\mathbb P(\xi_1<\cdots<\xi_l>\cdots>\xi_{N+1})$, and thus we obtain \eqref{varphired}:
 \begin{align*} \textbf{Eq3} \overset{\eqref{2ineq}}= \sum_{|I_{N+1}|=n}\sum_{l=1}^{N+1}\frac{(-1)^{l-1}}{2^{N+1}}\mathbb P(\xi_1<\cdots<\xi_l>\cdots>\xi_{N+1})(\bg_{i_1}\cdots[\theta,\bg_{i_l}]\cdots\bg_{i_{N+1}})^{N+1,-N-1}.
 \end{align*}

Now we give a sketch proof of \eqref{psired} using a similar argument as above. By \eqref{eqb4}, we have 
\begin{align*}&\big\{\sum_{|I_2|= n}\bar\eta_{i_1}\big([\theta^\st,\bg_{i_2}]-\bar\p\bg_{i_2}\big)-\sum_{|I_3|= n}\bar\eta_{i_1}\Big(S_{i_2}^\st([\theta^\st,\bg_{i_3}]-\bar\p\bg_{i_3})\Big)\big\}^{N+1,-N-1} \\
&=\Big\{\sum_{|I_{N}|=n} b_{i_1,\cdots,i_{N-1}}\bar\eta_{i_N}\mathcal L_{\Lambda_{N-1}^\st} [\theta^\st,\bg_{i_1}]-\sum_{|I_{N+1}|=n} {b}_{i_1,\cdots,i_{N}}\bar\eta_{i_{N+1}}\mathcal L_{\Lambda_N^\st} \bar\p\bg_{i_1}\Big\}^{N+1,-N-1}
\\
&\overset{\eqref{vanisheq_2}}=\sum_{|I_{N+1}|=n} (\cc \cdot\frac{i_2}{i_2+i_1}\cdot b_{i_1+i_2,i_3,\cdots,i_{N}}-b_{i_1,\cdots,i_{N}})\bar\eta_{i_{N+1}}(\mathcal L_{\Lambda_N^\st}\bar\p\bg_{i_1})^{N+1,-N-1}\\
&\overset{\eqref{abN}}=-\frac{1}{2^{N-1}}\sum_{|I_{N+1}|=n}\mathbb P(\xi_1>\xi_2>\cdots>\xi_{N})\bar\eta_{i_{N+1}}(\mathcal L_{\Lambda_N^\st}\bar\p\bg_{i_1})^{N+1,-N-1}\\
&\overset{\eqref{eq_iterated_brac}}=-\frac{1}{2^{N-1}}\sum_{|I_{N+1}|=n}\sum_{l=1}^{N}(-1)^{l-1}(\sum_{\sigma\in A'_{l;N}}\mathbb P(\xi_{\sigma(1)}>\xi_{\sigma(2)}>\cdots>\xi_{\sigma(N)}))(\bg_{i_1}\cdots\bar\p\bg_{i_l}\cdots\bg_{i_{N}})^{N+1,-N-1}\\
&\overset{\eqref{2ineq}}=\frac{1}{2^{N-1}}\sum_{|I_{N+1}|=n}\sum_{l=1}^{N}(-1)^{l}\mathbb P(\xi_1<\cdots<\xi_l>\cdots>\xi_{N})\bar\eta_{i_{N+1}}(\bg_{i_1}\cdots(\bar\p\bg_{i_l})\cdots\bg_{i_{N}})^{N+1,-N-1}.
\end{align*}
Using \eqref{eqb3} and by a similar argument as above, we have 
\begin{align*}&\Big\{ -[\theta^\st,\sum_{|I_{N-1}|=n} b_{i_1,\cdots,i_{N-1}}\bg_{i_1}\cdots\bg_{i_{N-1}}]+\bar\p(\sum_{|I_{N}|=n} b_{i_1,\cdots,i_{N}}\bg_{i_1}\cdots\bg_{i_{N}}]) \\&\makebox[5cm]{} +\sum_{i_1 + i_2 = n}\Big(S^\st_{i_1}([\theta^\st,\bg_{i_2}]-\bar\p\bg_{i_2})\Big)\Big\}^{N+1,-N-1}\\
&=-\sum_{|I_{N}|=n} (\cc\cdot \frac{i_2}{i_2+i_1}\cdot b_{i_1+i_2,i_3,\cdots,i_{N}}-b_{i_1,\cdots,i_{N}})(\mathcal L_{\Lambda_N^\st}\bar\p\bg_{i_1})^{N+1,-N-1}\\
&=\frac{1}{2^{N-1}}\sum_{|I_{N}|=n}\sum_{l=1}^{N}(-1)^{l-1}\mathbb P(\xi_1<\cdots<\xi_l>\cdots>\xi_{N})(\bg_{i_1}\cdots(\bar\p\bg_{i_l})\cdots\bg_{i_{N}})^{N+1,-N-1}.
\end{align*}
Substituting the above two identities into explicit formula \eqref{psi} of $\psi_n$ gives the proof of \eqref{psired}.
\end{proof}
\begin{remark}We remark that in \eqref{varphired}, the integer $n$ may be replaced by any $ 1, 2, \dots, n-1$. The resulting equality still holds. In this case, the first two terms on the right hand side of \eqref{varphired} vanish by \eqref{thmbg} applied to $k = 1, \dots, n-1$. The same observation applies to \eqref{psired}.
\end{remark}

\bigskip

Now we prove the main result in this section.
\begin{proof}[Proof of Proposition~\ref{1,-1}] For any positive integer $N$, we prove \eqref{thmbg} and \eqref{thmu} by an induction argument on $k$ . Firstly, the statement holds for $k=1$ by \eqref{u-wt2}. Assume the statement holds for $k=1,2,\cdots,n-1$, we prove the case for $k=n$.

When $k=n,\ N=1$, \eqref{thmbg} holds trivially. When $k=n,\ N=1,2$, \eqref{thmu} holds by \eqref{u-wt1} and \eqref{u-wt2}. 
 For any positive integers $N\geq 2$, under the inductive assumption, the conditions of Lemma~\ref{lem_simp_expre} are fulfilled, hence the simplified expression \eqref{varphired} of $\varphi_n$ holds for $k=n$.
 
 The modulo-$(t)$-holomorphicity gives \eqref{hol}. Substituting \eqref{varphired} with $k\leq n$  and \eqref{thmu} with $k\leq n-1$ into the grading piece \begin{align*}\big(\varphi_n+\sum\limits_{k=1}^{n-1}\varphi_{k}u_{n-k}+[\theta,u_n]\big)^{N+1,-N-1}=0\end{align*} of the first equation in \eqref{hol}, we have
\begin{align*}&[\theta,u_n]^{N+1,-N-1}
+\cc\big\{[\theta,\bg_n-\sum_{|I_{N+1}|=n} b_{i_1,\cdots,i_{N+1}}\bg_{i_1}\cdots\bg_{i_{N+1}}]\\&\makebox[3.1cm]{}-\Ch^{1,0}(\bg_n-\sum_{|I_{N}|=n} b_{i_1,\cdots,i_{N}}\bg_{i_1}\cdots\bg_{i_{N}})\big\}^{N+1,-N-1}\\
&+\frac{1}{2^{N+1}}\sum_{|I_{N+1}|=n}\sum_{l=1}^{N+1}(-1)^{l-1}\mathbb P(\xi_1<\cdots<\xi_l>\cdots>\xi_{N+1})(\bg_{i_1}\cdots[\theta,\bg_{i_l}]\cdots\bg_{i_{N+1}})^{N+1,-N-1}\\
&+\frac{1}{2^{N+1}}\sum_{|I_{N+1}|=n} \sum_{j=0}^{N-1}\sum_{l=1}^{j+1}(-1)^{l-1}\mathbb P(\xi_1<\cdots<\xi_l>\cdots>\xi_{j+1})(\bg_{i_1}\cdots[\theta,\bg_{i_l}]\cdots\bg_{i_{j+1}})^{j+1,-j-1}\times
\\&\makebox[3.7cm]{}(-1)^{N-j}\mathbb P(\xi_{j+2}<\cdots<\xi_{N+1})(\bg_{i_{j+2}}\cdots\bg_{i_{N+1}})^{N-j,-N+j}=0.
\end{align*}
Hence for $l=1,\cdots,N+1$, the coefficient of $(\bg_{i_1}\cdots[\theta,\bg_{i_l}]\cdots\bg_{i_{N+1}})^{N+1,-N-1}$ in last three lines of the left hand side of the above is 
\begin{align*}\frac{1}{2^{N+1}}\sum_{j=l-1}^{N}(-1)^{N-j+l-1}\mathbb P(\xi_1<\cdots<\xi_l>\cdots>\xi_{j+1})\mathbb P(\xi_{j+2}<\cdots<\xi_{N+1})
\stackrel{\eqref{altersumprob}}{=}\frac{(-1)^N }{2^{N+1}}\mathbb P(\xi_1<\cdots<\xi_{N+1}),
\end{align*}
where we denote $\mathbb P(\xi_{j+2}<\cdots<\xi_{N+1}):=1$ when $j+2\geq N+1$. In summary, $\big(\varphi_n+\sum\limits_{k=1}^{n-1}\varphi_{k}u_{n-k}+[\theta,u_n]\big)^{N+1,-N-1}=0$ gives
{\fontsize{10}{12}\selectfont \begin{equation}\label{exactphi}\begin{aligned}&[\theta,\cc\bg_n+\sum_{|I_{N+1}|=n} (-\cc b_{i_1,\cdots,i_{N+1}}+\frac{(-1)^N}{2^{N+1}}\mathbb P(\xi_1<\cdots<\xi_{N+1}))\bg_{i_1}\cdots\bg_{i_{N+1}}+u_n]^{N+1,-N-1}\\
&-\cc\Ch^{1,0}(\bg_n-\sum_{|I_{N}|=n} b_{i_1,\cdots,i_{N}}\bg_{i_1}\cdots\bg_{i_{N}})^{N+1,-N-1}=0.
\end{aligned}\end{equation} }

Substituting \eqref{psired} with $k\leq n$  and \eqref{thmu} with $k\leq n-1$ into the grading piece \begin{align*}\big(\psi_{n}+\sum\limits_{k=1}^{n-1}\psi_{k}u_{n-k}+\bar\p u_{n}-\sum\limits_{k=1}^{n-1}\bar\eta_{k}(\bar\p u_{n-k})\big)^{N+2,-N-2}=0\end{align*} of the second equation in \eqref{hol}, we have

\begin{equation}\label{exactpsi}\begin{aligned}&\big(\psi_{n}+\sum\limits_{k=1}^{n-1}\psi_{k}u_{n-k}+\bar\p u_{n}-\sum\limits_{k=1}^{n-1}\bar\eta_{k}(\bar\p u_{n-k})\big)^{N+2,-N-2}\\
&=\big\{\bar\p\big(\cc\bg_n+\sum_{|I_{N+1}|=n} (-\cc b_{i_1,\cdots,i_{N+1}}+\frac{(-1)^N}{2^{N+1}}\mathbb P(\xi_1<\cdots<\xi_{N+1}))\bg_{i_1}\cdots\bg_{i_{N+1}}+u_n\big)\\
&-\cc[\theta^\st,\bg_n-\sum_{|I_{N}|=n} b_{i_1,\cdots,i_{N}}\bg_{i_1}\cdots\bg_{i_{N}}]\big\}^{N+2,-N-2}=0.
\end{aligned}\end{equation}
Applying the ``harmonicity'' in Corollary~\ref{coro_harm} to \eqref{exactphi} and \eqref{exactpsi}, we get
\begin{align*}u_n^{N+1,-N-1}&=\big(-\cc\bg_n+\sum_{|I_{N+1}|=n} (\cc b_{i_1,\cdots,i_{N+1}}+\frac{(-1)^{N+1}}{2^{N+1}}\mathbb P(\xi_1<\cdots<\xi_{N+1}))\bg_{i_1}\cdots\bg_{i_{N+1}}\big)^{N+1,-N-1},\\
(\bg_n)^{N,-N}&=\sum\limits_{|I_{N}|=n} b_{i_1,\cdots,i_{N}}(\bg_{i_1}\cdots\bg_{i_{N}})^{N,-N}.\qedhere
\end{align*}
\end{proof}

\newpage

\section{Proof of Theorem~\ref{thm_main}}\label{sec_proof_main}

In this section, we first reduce the main Theorem~\ref{thm_main} to its truncated version Theorem~\ref{thm_main_tru} in Lemma~\ref{lem_reduce_to_tru}. Then, under the assumption of Theorem~\ref{thm_main_tru}, we find a gauge transformation in \eqref{eq_gauge_gr}, \eqref{thmv} via the partial equations in Proposition~\ref{prop_ob} and the recursive formula in Proposition~\ref{1,-1} such that under this gauge transformation, the deformation Higgs bundle can be checked to persevere a graded structure.

\begin{lemma}[Reduce the main theorem to truncated case] \label{lem_reduce_to_tru}
Under the assumption of Theorem~\ref{thm_main}. Let $(\E,\bar\p,\theta)$ on $X_0$ be the associated graded stable Higgs bundle of $(\mathbb V,\nabla,\mathcal F^\bullet,Q)$ via the non-abelian Hodge correspondence. If for any positive integer $n$ and any order-n germ $$\gamma:\operatorname{Spec}A_n\to U$$
passing through $0\in U$, the pull-back deformation of Higgs bundles $\sigma_{\Dol}\circ\gamma$ is a family of graded Higgs bundles, then Theorem~\ref{thm_main} holds true.
\end{lemma}

\begin{proof}
Without loss of generality, we may assume $U$ is smooth near $0$. If $U$ is singular, we resolve the singularity, denoted as $\pi:\hat U\to U$. By pulling back \( X \) and \( \sigma_{\Dol} \) to \( \hat{U} \), we also have a real analytic family of Higgs bundles over \( X_{\hat{U}} := X \times_U \hat{U} \), denoted by \( \hat{\sigma}_{\Dol} : \hat{U} \to M_{\Dol}(X_{\hat{U}} / \hat{U}) \). Choose any point \( \hat{u} \) of \( \hat{U} \) such that \( \pi(\hat{u}) = 0 \in U  \). By the definition of pull-back, \( \hat{\sigma}_{\Dol}(\hat{u}) \) is a stable graded Higgs bundle. We only need to prove Theorem~\ref{thm_main} for $\hat U$. Thus we reduce to the case that the base $U$ is smooth.

By the holomorphic assumption, we only need to prove the claim of Theorem~\ref{thm_main} for any truncated case and then use the convergence. This is exactly Theorem~\ref{thm_main_tru}.
\end{proof}

Under the assumption of Theorem~\ref{thm_main_tru}, the isomonodromic section of Higgs bundles $\sigma_{\Dol}$ is holomorphic and in particular, modulo-$(t)$-holomorphic. Hence, by taking $\st$ of the weight equation in Proposition~\ref{wtprop} we have for $k=1,2,\cdots,n$
\begin{align*}g_k\in \bigoplus_{l<0}\mathcal A^0((\End\E)^{l,-l})
\end{align*}
Besides, we have partial equations \eqref{vanisheq_1} and \eqref{vanisheq_2} for $k=1,2,\cdots,n$. 
After taking $\st$ of those two equations, we get for $k=1,2,\cdots,n$
\begin{align}
\cc(\bar\p g_k)^{0,0}=&\eta_{k}(\theta)-\cc\sum_{i=1}^{k-1}\eta_i(\Ch^{1,0}g_{k-i})^{0,0}+\frac{1}{4}\sum_{i=1}^{k-1} \frac{i}{k}\cdot[[\theta^\st, g_{k-i}], g_i]^{0,0};\label{vanisheqst_1} \\
\cc[\theta,g_k]^{-1,1}=&\frac{1}{4}\sum_{i=1}^{k-1} \frac{i}{k}\cdot[\Ch^{1,0} g_{k-i}, g_i]^{-1,1}.\label{vanisheqst_2}
\end{align}
Lastly by taking $\st$ in the recursive formula \eqref{thmbg} and using the fact $b_{i_1,i_2,\cdots,i_N}=b_{i_N,\cdots,i_2,i_1}$, we have for $k=1,2,\cdots,n$ and any positive integer $N$
\begin{align}\label{thmg}
g_k^{-N,N}=\sum_{|I_N|= k} b_{i_1,\cdots,i_N}(g_{i_1}\cdots g_{i_N})^{-N,N}.
\end{align}
Consequently, we have by \eqref{alpha} and \eqref{beta}: for any $i=1,\cdots,n$
\begin{align}\label{eq_wt_alphabeta}
\alpha_i^{l,-l}=0\quad\text{for}\quad l\geq 1;\qquad \beta_i^{j,-j}=0\quad\text{for}\quad j\geq 2.
\end{align}

Let $\Lambda_N:=(g_{i_2},g_{i_3},\cdots,g_{i_N})$. Using the notation \eqref{eq_Lieopera}, we have
\begin{align*}
\mathcal L_{\Lambda_N}(-):=[\cdots[-,g_{i_2}],g_{i_3}],\cdots],g_{i_N}].
\end{align*}
Recall we have introduced the notion of alternative sum in \eqref{eq_altersum_S}.

The following lemma is a key ingredient in simplifying the holomorphic deformation terms $\alpha_n$, $\beta_n$ given in \eqref{alpha} and \eqref{beta}, which is an analogue of Lemma~\ref{brackets} and Lemma~\ref{lem_recur2} without $\st$.
\begin{lemma}
For any $k=1,2,\cdots,n$ and any $N \ge 2$, the following identities hold:
{\fontsize{10}{12}\selectfont
\begin{flalign}
&\text{(i)}\quad 
\Big\{-\Ch^{1,0}g_k+ \sum_{ |I_2| = k} S_{i_1}\Ch^{1,0}g_{i_2}\Big\}^{-N+1,N-1}
= - \sum_{|I_{N}|=k} b_{i_1,\dots,i_N} (\mathcal L_{\Lambda_N}\Ch^{1,0}g_{i_1})^{-N+1,N-1}. & \label{eqb5}
\\
&\text{(ii)}\quad 
\Big\{[\theta,g_k] -\sum_{ |I_2| = k} S_{i_1} [\theta, g_{i_2}]\Big\} ^{-N+1,N-1}
=\sum_{|I_{N-1}|=k} b_{i_1,\dots,i_{N-1}} (\mathcal L_{\Lambda_{N-1}}[\theta,g_{i_1}])^{-N+1,N-1}. & \label{eqb6}
\\
&\text{(iii)}\quad 
\Big\{\bar\partial g_k - [\theta^{\star},g_k]+\sum_{ |I_2|= k} S_{i_1}([\theta^\st, g_{i_2}]-\bar\p g_{i_2})\Big\}^{-N+1,N-1} \notag \\
&\qquad =\Big\{\sum_{|I_{N}|=k} b_{i_1,\cdots,i_{N}}\mathcal L_{\Lambda_N}\bar\p g_{i_1}-\sum_{|I_{N-1}|=k} b_{i_1,i_2,\cdots,i_{N-1}}\mathcal L_{\Lambda_{N-1}}[\theta^\st,g_{i_1}]\Big\}^{-N+1,N-1}. & \label{eqb7}
\\
&\text{(iv)}\quad 
\Big\{\sum_{|I_2|=k}\eta_{i_1}(\Ch^{1,0}g_{i_2}-[\theta, g_{i_2}]) -\sum_{|I_3|=k} \eta_{i_1}\big\{S_{i_2}\big(\Ch^{1,0}g_{i_3}-[\theta, g_{i_3}]\big)\big\}\Big\}^{-N+1,N-1} \notag \\
&\qquad =\Big\{\sum_{|I_{N+1}|=k} b_{i_1,\dots,i_{N}}
\eta_{i_{N+1}} \bigl( \mathcal L_{\Lambda_N}\Ch^{1,0}g_{i_1}\bigr)-\sum_{|I_N|=k}b_{i_1,\cdots,i_{N-1}}\eta_{i_N}(\mathcal L_{\Lambda_{N-1}}[\theta,g_{i_1}])\Big\}^{-N+1,N-1}. & \label{eqb8}
\end{flalign}
}
\end{lemma}
\begin{proof}By a similar argument as in the proof of Lemma~\ref{brackets} and Lemma~\ref{lem_recur2}.
\end{proof}

The following lemma is an analogue of Lemma~\ref{lem_simp_expre} for $\alpha_i$ and $\beta_i$. In one word, the holomorphic deformation terms $\alpha_i$, $\beta_i$ in \eqref{alpha}, \eqref{beta} can be expressed in terms of the operators $(\theta,\theta^\st,\Ch^{1,0},\bar\p)$ and $\{\eta_i,g_i^{-1,1}\}_{i=1}^n$. The main idea of the proof is the same, but we have to treat some cases for grading pieces with small indices separately.
\begin{lemma}[Simplified expressions of $\alpha_n,\beta_n$]\label{lem_simpl_ab} Under the assumption of Theorem~\ref{thm_main_tru}, for any integer $N\geq 2$, we have 
{\fontsize{10}{12}\selectfont \begin{flalign}&\makebox[5mm]{}\alpha_n^{0,0}=-\eta_n(\theta)-\cc(\Ch^{1,0}g_n)^{0,0}+\cc\sum_{|I_2|=n}\eta_{i_1}(\Ch^{1,0}g_{i_2})^{0,0},& \label{alpha-0}\\
&\makebox[5mm]{}\alpha_n^{-N+1,N-1}=\frac{1}{2^{N}}\sum_{|I_{N}|=n}\sum_{l=1}^{N}(-1)^{l}\mathbb P(\xi_1<\cdots<\xi_l>\cdots>\xi_{N})(g_{i_1}\cdots (\Ch^{1,0}g_{i_l})\cdots g_{i_{N}})^{-N+1,N-1}& \notag\\
&\makebox[.2cm]{}+\frac{1}{2^N}\sum_{|I_{N+1}|=n}\sum_{l=1}^{N}(-1)^{l-1}\mathbb P(\xi_1<\cdots<\xi_l>\cdots>\xi_{N})(g_{i_1}\cdots (\eta_{i_{N+1}}(\Ch^{1,0}g_{i_l}))\cdots g_{i_{N}})^{-N+1,N-1},&\label{alphared}
\end{flalign}}
and
{\fontsize{10}{12}\selectfont \begin{flalign}
&\makebox[5mm]{}\beta_n^{1,-1}=-\cc[\theta^\st,g_n]^{1,-1};&\label{beta1}\\
&\makebox[5mm]{}\beta_n^{0,0}=\cc(\bar\p g_n)^{0,0}-\frac{1}{4}\sum_{|I_2|=n}[[\theta^\st,g_{i_1}],g_{i_2}]^{0,0}+\sum_{|I_2|=n}\eta_{i_1}(\Ch^{1,0}g_{i_2})^{0,0}.&\label{beta0}\\
&\makebox[5mm]{}
\beta_n^{-N+1,N-1}=&\notag\\
&\makebox[1cm]{}\frac{1}{2^{N+1}}\makebox[-2mm]{}\sum_{|I_{N+1}|=n}\makebox[-1mm]{}\sum_{l=1}^{N+1}(-1)^{l}\mathbb P(\xi_1<\cdots<\xi_l>\cdots>\xi_{N+1})(g_{i_1}\cdots [\theta^\st,g_{i_l}]\cdots g_{i_{N+1}})^{-N+1,N-1}.&\label{betared}
\end{flalign}}

\end{lemma}

\begin{proof}
By using \eqref{alpha}, \eqref{beta} and \eqref{thmg}, we get the simplified expressions \eqref{alpha-0} and \eqref{beta1}, \eqref{beta0} for grading pieces with small indices.  
We shall use the same strategy as in Lemma~\ref{lem_simp_expre} to prove the rest cases. Substituting \eqref{eqb6}, \eqref{eqb5} and \eqref{eqb8} into \eqref{alpha}, we have for $N\geq 2$
\begin{align*}&\alpha_n^{-N+1,N-1}=\cc\Big\{\sum_{|I_{N-1}|=n} b_{i_1,\dots,i_{N-1}} (\mathcal L_{\Lambda_{N-1}}[\theta,g_{i_1}]) - \sum_{|I_{N}|=n} b_{i_1,\dots,i_N} (\mathcal L_{\Lambda_N}\Ch^{1,0}g_{i_1})\\
&\makebox[2cm]{}+\sum_{|I_{N+1}|=n} b_{i_1,\dots,i_{N}}
\eta_{i_{N+1}} \bigl( \mathcal L_{\Lambda_N}\Ch^{1,0}g_{i_1}\bigr)-\sum_{|I_N|=n}b_{i_1,\cdots,i_{N-1}}\eta_{i_N}(\mathcal L_{\Lambda_{N-1}}[\theta,g_{i_1}])\Big\}^{-N+1,N-1}\\&\overset{\eqref{vanisheqst_1}\eqref{vanisheqst_2}}=\cc\sum_{|I_{N}|=n} (\cc\cdot\frac{i_2}{i_2+i_1}\cdot  b_{i_1+i_2,i_3,\cdots,i_N}-b_{i_1,\cdots,i_N})(\mathcal  L_{\Lambda_N}\Ch^{1,0}g_{i_1})^{-N+1,N-1}\\
&\makebox[1cm]{}+\cc\sum_{|I_{N+1}|=n} \bigl( b_{i_1,\dots,i_{N}} - \cc\cdot  \frac{i_2}{i_2+i_1}\cdot b_{i_1+i_2,i_3,\dots,i_{N}} \bigr) \,
\eta_{i_{N+1}} \bigl(\mathcal  L_{\Lambda_N}\Ch^{1,0}g_{i_1} \bigr)^{-N+1,N-1}\\
&\stackrel{\eqref{abN}}{=}\frac{1}{2^N}\Big\{-\sum_{|I_{N}|=n} \mathbb P(\xi_1>\cdots>\xi_N)\mathcal  L_{\Lambda_N}\Ch^{1,0}g_{i_1}+\sum_{|I_{N+1}|=n} \mathbb P(\xi_1>\cdots>\xi_N)\eta_{i_{N+1}}(\mathcal  L_{\Lambda_N}\Ch^{1,0}g_{i_1})\Big\}^{-N+1,N-1}.
\end{align*}
By expanding the iterated Lie brackets above and using \eqref{2ineq}, we get \eqref{alphared}.

Substituting \eqref{eqb7} and \eqref{eqb8} into \eqref{alpha} and using partial equations \eqref{vanisheqst_1}, \eqref{vanisheqst_2} and using \eqref{recursiveb}, we have for $N\geq 2$
\begin{align*}\beta_n^{-N,N}=\cc\sum_{|I_{N+1}|=n}(\cc\cdot \frac{i_2}{i_2+i_1}\cdot b_{i_1+i_2,i_3,\cdots,i_{N+1}}-b_{i_1,\cdots,i_{N+1}})(\mathcal L_{\Lambda_{N+1}}[\theta^\st, g_{i_1}])^{-N+1,N-1}.
\end{align*}
By expanding the iterated Lie brackets above and using \eqref{2ineq}, we get \eqref{alphared}.\end{proof}

\begin{proof}[Proof of Theorem~\ref{thm_main_tru}]Let \begin{align}\label{eq_gauge_gr}
\mathscr{U} = \operatorname{id} + \sum_{i=1}^n \bar{t}^i u_i + \sum_{i=0}^{n-1} \sum_{j=1}^{n-i} t^j \bar{t}^i u_{ij}\in\A^0(\End\E)\otimes B_n,\end{align}
such that $(\E,\mathscr U^{-1}\circ \bar\p_t\circ\mathscr U,\mathscr U^{-1}\circ \theta_t\circ\mathscr U)$ satisfies \eqref{eq_hol} (the existence of such $\mathscr U$ is guaranteed by the holomorphicity assumption).
We may choose $u_{i\bar 0}$ freely without affecting the validity of \eqref{eq_hol}. Thus we take $u_{i\bar 0}\in\mathcal \bigoplus_{l>0}A^0((\End\E)^{-l,l})$ and for any positive integer $N$
\begin{align}\label{thmv}u_{i\bar 0}^{-N,N}:=v_i^{-N,N}=\frac{(-1)^N}{2^N}\sum_{|I_N| = i} \mathbb P(\xi_1<\xi_2<\cdots<\xi_N)( g_{i_1} g_{i_2}\cdots g_{i_N})^{-N,N}.
\end{align}

We aim to \textbf{prove} that the following $\alpha'_k\in \A^1((\End\E)^{-1,1})$ and $\beta'_k\in \A^{0,1}((\End\E)^{0,0})$ for any $k=1,\cdots,n$:
\begin{align*}\bar\p_t\circ\mathscr U&=\mathscr U\circ\Big(\pi_\eta''\Ch+\sum_{k=1}^nt^k\beta'_k-\bar\eta(\sum_{k=1}^nt^k\beta'_k)\Big);\\
\theta_t\circ\mathscr U&=\mathscr U\circ\Big(P_\eta'\theta+\sum_{k=1}^nt^k\alpha_k'\Big).
\end{align*}
If so, the gauge transformed bundle $(\E,\mathscr U^{-1}\circ\bar\p_t\circ\mathscr U,\mathscr U^{-1}\circ\theta_t\circ\mathscr U)$ is family of graded stable Higgs bundles and this proves Theorem~\ref{thm_main_tru}. By definition of $\mathscr U$ and $\alpha_k',\ \beta_k'$, we have 
\begin{align}\label{grinduc}\begin{cases}\beta_k+\bar\p v_k+\sum\limits_{l=1}^{k-1}\eta_l(\Ch^{1,0}v_{k-l})-\sum\limits_{l=1}^{k-1}v_{k-l}\beta_l'+\sum\limits_{l=1}^{k-1}\beta_lv_{k-l}=\beta_k';\\
\alpha_k+[\theta,v_k]-\sum\limits_{l=1}^{k-1}v_{k-l}\alpha_l'+\sum\limits_{l=1}^{k-1}\alpha_lv_{k-l}=\alpha_k'.
\end{cases}
\end{align}

In particular, we prove for any positive integer $N$
\begin{align}\label{beta'}\beta_k'^{1,-1}=\beta_k^{1,-1},\text{ and }\ \beta_k'^{-N+1,N-1}=0\end{align} and for any integer $N\geq 2$
\begin{align}\label{alpha'}\alpha_k'^{-1,1}=\alpha_k^{0,0},\text{ and }\ \alpha_k'^{-N+1,N-1}=0.\end{align}

When $k=1$, by \eqref{thmv}, \eqref{grinduc} and \eqref{beta}, we have for any $N\geq 1$
\begin{align*}\beta_1'^{1,-1}&=(\beta_1+\bar\p v_1)^{1,-1}=\beta_1^{1,-1};\\
\beta_1'^{-N+1,N-1}&=(\beta_1+\bar\p v_1)^{-N+1,N-1}=(-\cc[\theta^\st,g_1])^{-N+1,N-1}=0.
\end{align*}

When $k=1$, by \eqref{thmv}, \eqref{grinduc} and \eqref{alpha}, we have for any $N\geq 2$
\begin{align*}\alpha_1'^{0,0}&=(\alpha_1+[\theta,v_1])^{0,0}=\alpha_1^{0,0};\\
\alpha_1'^{-N+1,N-1}&=(\alpha_1+[\theta,v_1])^{-N+1,N-1}=(-\eta_1(\theta)-\cc\Ch^{1,0}g_1)^{-N+1,N-1}=0.
\end{align*}
Hence \eqref{beta'} and \eqref{alpha'} hold for $k=1$. Now we inductively assume \eqref{beta'} and \eqref{alpha'} hold for $k=1,\cdots,n-1$ and prove them for $k=n$.

By \eqref{eq_wt_alphabeta}, \eqref{thmv}, \eqref{grinduc} and the induction, we have 
\begin{align*}\beta_n'^{1,-1}=&\beta_n^{1,-1}+(\bar\p v_n+\sum_{|I_2|=n} \eta_{i_1}(\Ch^{1,0}v_{i_2}))^{1,-1}+\sum_{|I_2|=n} (-v_{i_1}\beta_{i_2}'+\beta_{i_1}v_{i_2})^{1,-1}=\beta_n^{1,-1};\\
\beta_n'^{0,0}=&\beta_n^{0,0}+(\bar\p v_n+\sum_{|I_2|=n} \eta_{i_1}(\Ch^{1,0}v_{i_2}))^{0,0}+\sum_{|I_2|=n} (-v_{i_1}\beta_{i_2}'+\beta_{i_1}v_{i_2})^{0,0}=0.
\end{align*}
We aim to prove the following is zero for $N\geq 2$
\begin{align*}\beta_n'^{-N+1,N-1}=&\beta_n^{-N+1,N-1}+(\bar\p v_n+\sum_{|I_2|=n} \eta_{i_1}(\Ch^{1,0}v_{i_2}))^{-N+1,N-1}+\sum_{|I_2|=n} (-v_{i_1}\beta_{i_2}'+\beta_{i_1}v_{i_2})^{-N+1,N-1}.
\end{align*}
We divide this expression into two parts as follow:
\begin{align*}\textbf{Eq4}:&=\beta_n^{-N+1,N-1}+\sum\limits_{|I_2|=n} (-v_{i_1}\beta_{i_2}'+\beta'_{i_1}v_{i_2})^{-N+1,N-1}+\sum\limits_{j=2}^{N-1}\beta_{i_1}^{-j+1,j-1}v_{i_2}^{-N+j,N-j}.\\
\textbf{Eq5}:&=(\bar\p v_n)^{-N+1,N-1}+\sum\limits_{|I_2|=n} \eta_{i_1}(\Ch^{1,0}v_{i_2})^{-N+1,N-1}+\sum\limits_{|I_2|=n}(\beta_{i_1})^{0,0}(v_{i_2})^{-N+1,N-1}
\end{align*}
 Then by the simplified expressions \eqref{beta1}, \eqref{betared} in Lemma~\ref{lem_simpl_ab}, and \eqref{thmv} and the induction
\begin{align*}
&\textbf{Eq4}=\frac{(-1)^N}{2^{N+1}}\sum_{|I_{N+1}|=n}\big\{\sum_{l=1}^{N+1}(-1)^{-N+l}\mathbb P(\xi_1<\cdots<\xi_l>\cdots>\xi_{N+1})(g_{i_1}\cdots[\theta^\st,g_{i_l}]\cdots g_{i_{N+1}})\\
&+ \mathbb P(\xi_1<\cdots<\xi_N)(g_{i_1}\cdots g_{i_N}[\theta^\st, g_{i_{N+1}}])-\mathbb P(\xi_2<\cdots<\xi_{N+1})([\theta^\st,g_{i_1}]g_{i_2}\cdots g_{i_{N+1}})\\
&+\sum_{j=2}^{N-1}\sum_{l=1}^{j+1}(-1)^{-j+l}\mathbb P(\xi_1<\cdots<\xi_l>\cdots>\xi_{j+1})\mathbb P(\xi_{j+2}<\cdots<\xi_{N+1})
(g_{i_1}\cdots[\theta^\st,g_{i_l}]\cdots g_{i_{N+1}})\big\}^{-N+1,N-1}.
\end{align*}
By a similar argument as in \eqref{altersumprob}, the coefficients of $(g_{i_1}\cdots g_{i_{l-1}}[\theta^\star,g_{i_l}]g_{i_{l+1}}\cdots g_{i_{N+1}})^{-N+1,N-1}$ in \textbf{Eq4} are given by
\[
\frac{(-1)^{N+1}}{2^{N+1}}\cdot
\begin{cases}
\mathbb P(\xi_2<\cdots<\xi_{N+1})+\mathbb P(\xi_1>\xi_2>\xi_3<\cdots<\xi_{N+1}), & l=1,\\[6pt]
\mathbb P(\xi_1<\cdots<\xi_{N+1})-\mathbb P(\xi_1<\xi_2)\mathbb P(\xi_3<\cdots<\xi_{N+1}), & l=2,\\[6pt]
\mathbb P(\xi_1<\cdots<\xi_{N+1}), & 3\le l\le N,\\[6pt]
\mathbb P(\xi_1<\cdots<\xi_{N+1})-\mathbb P(\xi_1<\cdots<\xi_N), & l=N+1.
\end{cases}
\]

Note that by the simplified expression \eqref{beta0} in Lemma~\ref{lem_simpl_ab} and \eqref{thmv}
\begin{align*}
&\textbf{Eq5}=\frac{(-1)^{N}}{2^N}\Big\{\sum_{|I_{N}|=n}\big\{ \mathbb P(\xi_1<\cdots<\xi_N)\bar\p(g_{i_1}\cdots g_{i_N})-\mathbb P(\xi_2<\cdots<\xi_N)(\bar\p g_{i_1})g_{i_2}\cdots g_{i_N}\big\}\\
&+\sum_{|I_{N+1}|=n}\big\{\mathbb P(\xi_1<\cdots<\xi_N)\eta_{i_{N+1}}(\Ch^{1,0} (g_{i_1}g_{i_2}\cdots g_{i_{N}})\big)\\&-\mathbb P(\xi_2<\cdots<\xi_N)\eta_{i_{N+1}}((\Ch^{1,0} g_{i_1})g_{i_2}\cdots g_{i_{N}})\big)+\cc \mathbb P(\xi_3<\cdots<\xi_{N+1})[[\theta^\st,g_{i_1}],g_{i_2}]g_{i_3}\cdots g_{i_{N+1}}\big\}\Big\}^{-N+1,N-1}.
\end{align*}
Using the partial eqaution \eqref{vanisheqst_1}, \textbf{Eq5} can be expressed as a linear combination of $$(g_{i_1}\cdots g_{i_{l-1}}[\theta^\st,g_{i_{l}}]g_{i_{l+1}}\cdots g_{i_{N+1}})^{-N+1,N-1},\quad l=1,2,\cdots,N+1,$$
and we place this complicated computation in Lemma~\ref{betaterms} at the end of this section.
By Lemma~\ref{betaterms} and \textbf{Eq4} above, one can prove \textbf{Eq4}$+$\textbf{Eq5} is zero as desired.

\bigskip

By \eqref{eq_wt_alphabeta}, \eqref{grinduc} and \eqref{thmv}, we have 
\begin{align*}\alpha_n'^{0,0}=\alpha_n^{0,0}+[\theta, v_n]^{0,0}+\sum_{|I_2|=n} (-v_{i_1}\alpha_{i_2}'+\alpha_{i_1}v_{i_2})^{0,0}=\alpha_n^{0,0}.
\end{align*}

We aim to prove the following is zero for $N\geq 2$
\begin{align*}\alpha_n'^{-N+1,N-1}=&\alpha_n^{-N+1,N-1}+[\theta, v_n]^{-N+1,N-1}+\sum_{|I_2|=n} (-v_{i_1}\alpha_{i_2}'+\alpha_{i_1}v_{i_2})^{-N+1,N-1}.\end{align*}
Note that 
\begin{align*}&[\theta,v_n]^{-N+1,N-1}\overset{\eqref{thmv}}=
\frac{(-1)^{N-1}}{2^{N-1}}\sum_{|I_{N-1}|=n}\sum_{l=1}^{N-1} \mathbb P(\xi_1<\cdots<\xi_{N-1})(g_{i_1}\cdots[\theta,g_{i_l}]\cdots g_{i_{N-1}})^{-N+1,N-1}\\
&\overset{\eqref{vanisheqst_2}}=\frac{(-1)^{N-1}}{2^{N}}\sum_{|I_{N}|=n}\sum_{l=1}^{N-1} \mathbb P(\xi_1<\cdots<\max(\xi_l,\xi_{l+1})<\xi_{l+2}<\cdots<\xi_{N})\cdot \frac{i_{l+1}}{i_l+i_{l+1}}\times\\&\makebox[4cm]{}(g_{i_1}\cdots[\Ch^{1,0}g_{i_l},g_{i_{l+1}}]g_{i_{l+2}}\cdots g_{i_{N}})^{-N+1,N-1}.
\end{align*}
Expanding the Lie brackets in the above and using \eqref{probmountain'}, we get
\begin{align*} [\theta,v_n]^{-N+1,N-1}&=\frac{(-1)^{N-1}}{2^N}\sum_{|I_{N}|=n} \sum_{l=1}^{N-1} \mathbb P(\xi_1<\cdots<\xi_N)(g_{i_1}\cdots(\Ch^{1,0}g_{i_l})\cdots g_{i_{N}})^{-N+1,N-1}\\
&+\frac{(-1)^N}{2^N}\sum_{|I_{N}|=n}\mathbb P(\xi_1<\cdots<\xi_{N-1}>\xi_N)(g_{i_1}\cdots g_{i_{N-1}}(\Ch^{1,0}g_{i_N}))^{-N+1,N-1}.
\end{align*}

Using \eqref{thmv}, \eqref{alpha'}, \eqref{alpha-0} and \eqref{vanisheqst_2}, for $N\geq 2$, we have
\begin{align*}&\sum_{|I_2|=n}[\alpha'_{i_1},v_{i_2}]^{-N+1,N-1}=\frac{(-1)^N}{2^N}\Big\{\sum_{|I_{N+1}|=n}\big\{\sum_{l=1}^{N}\mathbb P(\xi_1<\cdots<\xi_{N})\eta_{i_{N+1}}(g_{i_1}\cdots(\Ch^{1,0}g_{i_l})\cdots g_{i_{N}})\\
&-\mathbb P(\xi_2<\cdots<\xi_{N})\eta_{i_{N+1}}((\Ch^{1,0}g_{i_1})g_{i_2}\cdots g_{i_{N}})\big\}+\sum_{|I_{N}|=n}[\Ch^{1,0} g_{i_1},\mathbb P(\xi_2<\cdots<\xi_N)g_{i_2}\cdots g_{i_N}]\Big\}^{-N+1,N-1}.
\end{align*}
Using the above two equations and \eqref{alphared} and \eqref{thmv}, we have, for $N\geq 2$, the $(0,1)$-part and $(1,0)$-part of \begin{align*}\alpha_n'^{-N+1,N-1}=&\alpha_n^{-N+1,N-1}+[\theta, v_n]^{-N+1,N-1}+\sum_{|I_2|=n} (-v_{i_1}\alpha_{i_2}'+\alpha_{i_1}v_{i_2})^{-N+1,N-1}\\
=&\alpha_n^{-N+1,N-1}+[\theta,v_n]^{-N+1,N-1}+\sum_{|I_2|=n}[\alpha_{i_1}',v_{i_2}]^{-N+1,N-1}+\sum_{|I_2|=n}\sum_{j=2}^{N-1}\alpha_{i_1}^{-j+1,j-1}v_{i_2}^{-(N-j),N-j}\end{align*} are 
\begin{align*}&\frac{(-1)^{N+1}}{2^{N}}\sum_{|I_{N+1}|=n}\eta_{i_{N+1}}\Big\{\sum_{l=1}^{N}(-1)^{-N+l}\mathbb P(\xi_1<\cdots<\xi_l>\cdots>\xi_{N})g_{i_1}\cdots(\Ch^{1,0}g_{i_l})\cdots g_{i_{N}}\\
&-\sum_{l=1}^{N}\mathbb P(\xi_1<\cdots<\xi_{N})g_{i_1}\cdots(\Ch^{1,0}g_{i_l})\cdots g_{i_{N}}+\mathbb P(\xi_2<\cdots<\xi_{N})(\Ch^{1,0}g_{i_1})g_{i_2}\cdots g_{i_{N}}\\
&+\sum_{j=2}^{N-1}\sum_{l=1}^{j} (-1)^{l-j}\mathbb P(\xi_1<\cdots<\xi_l>\cdots>\xi_{j})\mathbb P(\xi_{j+1}<\cdots<\xi_{N})g_{i_1}\cdots(\Ch^{1,0}g_{i_l})\cdots g_{i_{N}}\Big\}^{-N+1,N-1}\end{align*} 
and
\begin{align*}
&\frac{(-1)^N}{2^{N}}\sum_{|I_{N}|=n}\Big\{\sum_{l=1}^{N}(-1)^{-N+l}\mathbb P(\xi_1<\cdots<\xi_l>\cdots>\xi_{N})g_{i_1}\cdots(\Ch^{1,0}g_{i_l})\cdots g_{i_{N}}\\
&- \sum_{l=1}^{N-1} \mathbb P(\xi_1<\cdots<\xi_N)g_{i_1}\cdots(\Ch^{1,0}g_{i_l})\cdots g_{i_{N}}\\
&+ \mathbb P(\xi_1<\cdots<\xi_{N-1}>\xi_N)g_{i_1}\cdots g_{i_{N-1}}(\Ch^{1,0}g_{i_N})+ [\Ch^{1,0} g_{i_1},\mathbb P(\xi_2<\cdots<\xi_N)g_{i_2}\cdots g_{i_N}]\\
&+ \sum_{j=2}^{N-1}\sum_{l=1}^{j} (-1)^{l-j}\mathbb P(\xi_1<\cdots<\xi_l>\cdots>\xi_{j})\mathbb P(\xi_{j+1}<\cdots<\xi_{N})g_{i_1}\cdots(\Ch^{1,0}g_{i_l})\cdots g_{i_{N}})\Big\}^{-N+1,N-1}
\end{align*}
respectively, and those terms cancel out by \eqref{altersumprob}. Thus $\alpha_n'^{-N+1,N-1}=0$ for $N\geq2$.
\end{proof}

For $l = 1, \dots, N$, define 
\[
B_l := (\xi_1 < \dots < \xi_{l-1} < \widehat{\xi_l} < \xi_{l+1} < \dots < \xi_N),
\] 
where $\widehat{\xi_l}$ indicates that $\xi_l$ is omitted from the inequality, and let
\begin{align*}D_l:=(\xi_1 < \dots < \xi_{l-1} < \widehat{\xi_l} <\widehat{\xi_{l+1}} <\xi_{l+2}< \dots < \xi_{N+1}).
\end{align*}

\begin{lemma}\label{betaterms}For any integer $N\geq 2$, we have
\begin{equation}\label{beta'II}\begin{aligned}&\Big\{\sum_{|I_{N}|=n}\Big( \mathbb P(\xi_1<\cdots<\xi_N)\bar\p(g_{i_1}\cdots g_{i_N})-\mathbb P(\xi_2<\cdots<\xi_N)(\bar\p g_{i_1})g_{i_2}\cdots g_{i_N}\Big)\\
&+\sum_{|I_{N+1}|=n}\Big(\mathbb P(\xi_1<\cdots<\xi_N)\eta_{i_{N+1}}\big(\Ch^{1,0} (g_{i_1}g_{i_2}\cdots g_{i_{N}})\big)\\&-\mathbb P(\xi_2<\cdots<\xi_N)\eta_{i_{N+1}}(\Ch^{1,0} g_{i_1})g_{i_2}\cdots g_{i_{N}}\Big)\Big\}^{-N+1,N-1}\\
&=-\sum_{|I_{N}|=n}\sum_{l=1}^{N-1} \mathbb P(B_l\land(\xi_{l+1}<\xi_{l}))\big\{g_{i_1}\cdots [\bar\p g_{i_{l}},g_{i_{l+1}}]\cdots g_{i_{N}}\big\}^{-N+1,N-1}\\
&-\sum_{|I_{N+1}|=n}\sum_{l=1}^{N-1} \mathbb P(B_l\land(\xi_{l+1}<\xi_{l}))\big\{g_{i_1}\cdots [\eta_{i_{N+1}}(\Ch^{1,0} g_{i_{l}}),g_{i_{l+1}}]\cdots g_{i_{N}}\big\}^{-N+1,N-1}.
\end{aligned}\end{equation}
Using the partial equations \eqref{vanisheqst_1} and \eqref{vanisheqst_2}, the last two lines in \eqref{beta'II} can be expressed as a linear combination of $$(g_{i_1}\cdots g_{i_{l-1}}[\theta^\st,g_{i_{l}}]g_{i_{l+1}}\cdots g_{i_{N+1}})^{-N+1,N-1},\quad l=1,2,\cdots,N+1,$$
with coefficient \[
\cc\cdot
\begin{cases}
-\mathbb P(D_1\land(\xi_2>\max(\xi_1,\xi_3))), & l=1,\\[6pt]
\mathbb P(D_1\land(\xi_1>\max(\xi_2,\xi_3)))  + \mathbb P((\xi_1<\xi_3<\cdots<\xi_{N+1})\land(\xi_3>\xi_2)), & l=2,\\[6pt]
\mathbb P(\xi_1<\cdots<\xi_{N+1}), & 3\le l\le N,\\[6pt]
\mathbb P(\xi_1<\cdots<\xi_{N+1})-\mathbb P(\xi_1<\cdots<\xi_N), & l=N+1.
\end{cases}
\]
\end{lemma}
\begin{proof}By expanding the Lie brackets in \eqref{beta'II} and reordering the summation indices, the coefficient of $((\bar\p g_{i_1})g_{i_2}\cdots g_{i_N})^{-N+1,N-1}$ and $\eta_{i_{N+1}}((\Ch^{1,0} g_{i_1})g_{i_2}\cdots g_{i_N})^{-N+1,N-1}$ on the right hand side of \eqref{beta'II} is 
\begin{align*}-\mathbb P(B_1\land(\xi_2<\xi_1))=\mathbb P(\xi_1<\cdots<\xi_N)-\mathbb P(\xi_2<\cdots<\xi_N).
\end{align*}
For $l=2,\cdots,N$, the coefficient of $(g_{i_1}\cdots (\bar\p g_{i_1})\cdots g_{i_N})^{-N+1,N-1}$ and $\eta_{i_{N+1}}(g_{i_1}\cdots (\Ch^{1,0} g_{i_1})\cdots g_{i_N})^{-N+1,N-1}$ on the right hand side of \eqref{beta'II} is 
\begin{align*}&-\mathbb P(B_l\land(\xi_{l+1}<\xi_l))
+\mathbb P(B_{l}\land(\xi_{l-1}<\xi_{l}))
=\mathbb P(\xi_1<\cdots<\xi_N).
\end{align*}
This proves \eqref{beta'II}. For $l=1,\cdots,N-1$, by \eqref{vanisheqst_1}, we have
\begin{align*}&\sum_{|I_{N}|=n} \mathbb P(B_l\land(\xi_{l+1}<\xi_{l}))\big\{ g_{i_1}\cdots g_{i_{l-1}}[\bar\p g_{i_{l}},g_{i_{l+1}}]g_{i_{l+2}}\cdots g_{i_N}\big\}^{-N+1,N-1}\\
&=\sum_{|I_{N}|=n} \mathbb P(B_l\land(\xi_{l+1}<\xi_{l}))\big\{g_{i_1}\cdots g_{i_{l-1}}[2\eta_{i_l}(\theta),g_{i_{l+1}}]g_{i_{l+2}}\cdots g_{i_N}\big\}^{-N+1,N-1}\\
&\sum_{|I_{N+1}|=n} \Big\{-\mathbb P(B_l\land(\xi_{l+1}<\max(\xi_{l},\xi_{N+1})))\big\{g_{i_1}\cdots g_{i_{l-1}}\eta_{i_{N+1}}([\Ch^{1,0} g_{i_{l}},g_{i_{l+1}}])g_{i_{l+2}}\cdots g_{i_N}\\
&+\cc \mathbb P(B_l\land(\xi_{l+1}<\max(\xi_{l},\xi_{N+1})))\frac{i_l}{i_l+i_{N+1}}\cdot g_{i_1}\cdots g_{i_{l-1}}[[[\theta^\st,g_{i_{N+1}}],g_{i_l}],g_{i_{l+1}}]g_{i_{l+2}}\cdots g_{i_N}\Big\}^{-N+1,N-1},
\end{align*}
denoted by \textbf{Eq6}, \textbf{Eq7}, \textbf{Eq8} for each summation on the right hand side of the above. Hence
\begin{align*}&\textbf{Eq6}\overset{\eqref{vanisheqst_1}}=\sum_{|I_{N+1}|=n} \mathbb P((\xi_1<\cdots<\xi_{l-1}<\max(\xi_{l+1},\xi_{N+1})<\xi_{l+2}<\cdots<\xi_N)\land(\max(\xi_{l+1},\xi_{N+1})<\xi_{l}))\cdot \\
&\makebox[2.6cm]{}\frac{i_{l+1}}{i_{l+1}+i_{N+1}}\cdot\big\{g_{i_1}\cdots\eta_{i_l}([\Ch^{1,0}g_{i_{N+1}},g_{i_{l+1}}])g_{i_{l+2}}\cdots g_{i_N}\big\}^{-N+1,N-1}\\
&=\sum_{|I_{N+1}|=n} \mathbb P((\xi_1<\cdots<\xi_{l-1}<\max(\xi_{l+1},\xi_{l})<\xi_{l+2}<\cdots<\xi_N)\land(\max(\xi_{l+1},\xi_{l})<\xi_{N+1}))\cdot \\
&\makebox[2cm]{}\frac{i_{l+1}}{i_{l}+i_{l+1}}\cdot\big\{g_{i_1}\cdots\eta_{i_{N+1}}([\Ch^{1,0}g_{i_{l}},g_{i_{l+1}}])g_{i_{l+2}}\cdots g_{i_N}\big\}^{-N+1,N-1}\\
&=\sum_{|I_{N+1}|=n} \mathbb P(B_l\land(\xi_l<\xi_{l+1}<\xi_{N+1}))\big\{g_{i_1}\cdots\eta_{i_{N+1}}([\Ch^{1,0}g_{i_{l}},g_{i_{l+1}}])g_{i_{l+2}}\cdots g_{i_N}\big\}^{-N+1,N-1},
\end{align*}
where the last equality follows from \eqref{probmountain'}. Next,
\begin{align*}\textbf{Eq6}+\textbf{Eq7}
=&\sum_{|I_{N+1}|=n} (\mathbb P(B_l\land(\xi_l<\xi_{l+1}<\xi_{N+1}))-\mathbb P(B_l\land(\xi_{l+1}<\max(\xi_{l},\xi_{N+1})))\cdot\\
&\makebox[1.2cm]{}\big\{g_{i_1}\cdots \eta_{i_{N+1}}([\Ch^{1,0} g_{i_{l}},g_{i_{l+1}}])g_{i_{l+2}}\cdots g_{i_N}\big\}^{-N+1,N-1}\\
=&-\sum_{|I_{N+1}|=n} (\mathbb P(B_l\land(\xi_{l+1}<\xi_{l}))\big\{g_{i_1}\cdots \eta_{i_{N+1}}([\Ch^{1,0} g_{i_{l}},g_{i_{l+1}}])g_{i_{l+2}}\cdots g_{i_N}\big\}^{-N+1,N-1}.
\end{align*}

By performing the permutation $(N+1\quad l\quad l+1\quad l+2\quad \cdots\quad N)$ on the summation index of \textbf{Eq8}, 
\begin{align*}&\textbf{Eq8}\\
&=\cc \sum_{|I_{N+1}|=n} \mathbb P(D_l\land(\xi_{l+2}<\max(\xi_{l},\xi_{l+1})))\frac{i_{l+1}}{i_l+i_{l+1}}\cdot\big\{g_{i_1}\cdots [[[\theta^\st,g_{i_{l}}],g_{i_{l+1}}],g_{i_{l+2}}]g_{i_{l+3}}\cdots g_{i_{N+1}}\big\}^{-N+1,N-1}\\
&=\cc\sum_{|I_{N+1}|=n} \mathbb P(D_l\land(\xi_{l+1}>\max(\xi_l,\xi_{l+2})))(g_{i_1}\cdots g_{i_{l-1}}[[[\theta^\st,g_{i_l}],g_{i_{l+1}}],g_{i_{l+2}}]g_{i_{l+3}}\cdots g_{i_{N+1}})^{-N+1,N-1}.
\end{align*}
By summing \textbf{Eq6}, \textbf{Eq7} and \textbf{Eq8}, we get the right hand side of \eqref{beta'II} is \begin{align*}
-\cc\sum_{|I_{N+1}|=n}\sum_{l=1}^{N-1}\mathbb P(D_l\land(\xi_{l+1}>\max(\xi_l,\xi_{l+2})))(g_{i_1}\cdots g_{i_{l-1}}[[[\theta^\st,g_{i_l}],g_{i_{l+1}}],g_{i_{l+2}}]g_{i_{l+3}}\cdots g_{i_{N+1}})^{-N+1,N-1}.
\end{align*}
By expanding the iterated Lie brackets above and reordering the summation indices, we get, in the above,
\item[\ding{172}] the coefficient of $([\theta^\st,g_{i_1}]g_{i_2}\cdots g_{i_{N+1}})^{-N+1,N-1}$ is $$-\cc \mathbb P(D_1\land (\xi_2>\max(\xi_1,\xi_3)));$$
\item[\ding{173}] the coefficient of $(g_{i_1}[\theta^\st,g_{i_2}]g_{i_3}\cdots g_{i_{N+1}})^{-N+1,N-1}$ is 
\begin{align*}&-\cc\big\{-\mathbb P(D_1\land(\xi_1>\max(\xi_2,\xi_3))+\mathbb P(D_2\land(\xi_3>\max(\xi_2,\xi_4))-\mathbb P(D_2\land(\xi_3>\max(\xi_1,\xi_2)))\big\}\\
&=-\cc\big\{-\mathbb P(D_1\land(\xi_1>\max(\xi_2,\xi_3)))-\mathbb P((\xi_1<\xi_3<\cdots<\xi_{N+1})\land(\xi_3>\xi_2))\big\};\end{align*}
\item[\ding{174}] the coefficient of $(g_{i_1}\cdots g_{i_{l-1}}[\theta^\st,g_{i_l}]g_{i_{l+1}}\cdots g_{i_{N+1}})^{-N+1,N-1}$, where $3\leq l\leq N-1$ is 
\begin{align*}&-\cc\big\{\mathbb P(D_l\land(\xi_{l+1}>\max(\xi_l,\xi_{l+2})))-\mathbb P(D_l\land(\xi_{l+1}>\max(\xi_l,\xi_{l-1})))\\
&\quad\quad\quad+\mathbb P(D_{l-1}\land(\xi_{l-1}>\max(\xi_l,\xi_{l-2})))-\mathbb P(D_{l-1}\land(\xi_{l-1}>\max(\xi_l,\xi_{l+1})))\big\}\\
&=-\cc\big\{-\mathbb P((\xi_1<\cdots<\xi_{l-1}\widehat {\xi_l}<\xi_{l+1}<\cdots<\xi_{N+1})\land(\xi_{l+1}>\xi_l))\\
&\quad\quad\quad+\mathbb P((\xi_1<\cdots<\xi_{l-1}\widehat {\xi_l}<\xi_{l+1}<\cdots<\xi_{N+1})\land(\xi_{l-1}>\xi_l))\big\}\\
&=-\cc (-\mathbb P(\xi_1<\cdots<\xi_{N+1}));
\end{align*}
\item[\ding{175}] the coefficient of $(g_{i_1}\cdots g_{i_{N-1}}[\theta^\st,g_{i_N}]g_{i_{N+1}})^{-N+1,N-1}$ is 
\begin{align*}&-\cc\big\{\mathbb P(D_{N-1}\land(\xi_{N-1}>\max(\xi_N,\xi_{N-2})))-\mathbb P(D_{N-1}\land(\xi_{N-1}>\max(\xi_N,\xi_{N+1})))\\
&\quad\quad\quad-\mathbb P(D_{N}\land(\xi_{N+1}>\max(\xi_N,\xi_{N-1})))\big\}\\
&=-\cc\big\{\mathbb P((\xi_1<\cdots<\xi_{N-1}<\xi_{N+1})\land(\xi_{N-1}>\xi_N))-\mathbb P((\xi_1<\cdots<\xi_{N-1})\land(\xi_{N+1}>\max(\xi_N,\xi_{N-1}))\big\}\\
&=-\cc (-\mathbb P(\xi_1<\cdots<\xi_{N+1}));
\end{align*}
\item[\ding{176}] the coefficient of $(g_{i_1}\cdots g_{i_{N}}[\theta^\st,g_{i_{N+1}}])^{-N+1,N-1}$ is 
\[
-\cc \mathbb P(\xi_1<\cdots<\xi_N>\xi_{N+1}).\qedhere 
\]
\end{proof}

\appendix
\newpage

\section{Some combinatorial lemmas}
Let $\xi_{1},\cdots,\xi_{N}$ be independent random variables with distributions $$\operatorname{Beta}(i_1,1),\cdots,\operatorname{Beta}(i_N,1)$$ respectively, where $i_1,\cdots,i_N$ are positive integers. We define the sequence $b_{i_1,\cdots,i_N}$ as in \eqref{bN}. 

\begin{definition}\label{AkN}
Let $N$ be a positive integer.
\begin{enumerate}\item For any integer $1 \le k < N$. Denote 
 \[A_{k;N} := \left\{\sigma \in S_N \,\left|\, 
\begin{aligned}
&\sigma^{-1}(k) <\sigma^{-1}(k-1) <\cdots <\sigma^{-1}(1), \text{ and }\\
&\sigma^{-1}(k+1)<\sigma^{-1}(k+2)<\cdots<\sigma^{-1}(N)
\end{aligned}
\right.\right\}.\]
\item Since, for any $\sigma\in A_{k;N}$, $\sigma^{-1}(k)=1$ or $\sigma^{-1}(k+1)=1$, there is a partition 
\[A_{k;N} = A'_{k;N}\coprod A''_{k;N},\] where
$A'_{k;N} = \{\sigma \in A_{k;N} \mid \sigma(1)=k\}$ and 
$A''_{k;N} = \{\sigma \in A_{k;N} \mid \sigma(1)=k+1\}$. In addition, we define
\begin{align*}A_{N;N}':=\{\sigma\in S_N\mid\sigma(l)=N+1-l\ \text{ for any }\ l=1,2,\cdots,N\}.
\end{align*}
\item For any integer $1 \le k \le N$. Denote $V_{k;N} :=\{\sigma \mid \sigma^{-1} \in A'_{k;N}\}$. More precisely, 
\[V_{k;N} = \left\{\sigma \in S_N \,\left|\, 
\begin{aligned}
&\sigma(1) >\sigma(2) >\cdots >\sigma(k-1),\ \sigma(k)=1, \\
&\text{and }\ \sigma(k+1)<\sigma(k+2)<\cdots<\sigma(N)
\end{aligned}
\right.\right\}.\]\end{enumerate}
\end{definition}

It is easy to see that the cardinality \(\# A_{k;N}=\binom{N}{k}\).

\begin{lemma}\label{partition}
For \(1\leq j\leq N\), the event \(\xi_{1}<\cdots<\xi_{j}>\cdots>\xi_{N}\) can be decomposed into \(\binom{N-1}{j-1}\) mutually exclusive events, and each of these events is a complete ordering inequality (a total order) of the \(N\) random variables. More precisely, we have
\begin{align}\label{2ineq}(\xi_1<\cdots<\xi_j>\cdots>\xi_N)=\lor_{\sigma\in A'_{j;N}}(\xi_{\sigma(1)}>\cdots>\xi_{\sigma(N)}),
\end{align}
(since equality occurs with probability zero, we may neglect such cases.)
Consequently, \(2^{N-1}\cdot b_{i_1,\cdots,i_N}\) is the sum of probabilities of \(2^{N-1}\) mutually exclusive events.
\end{lemma}

\begin{proof}
Combining the inequalities \(\xi_{j}>\xi_{{j-1}}>\cdots>\xi_{1}\) and \(\xi_{j}>\xi_{{j+1}}>\cdots>\xi_{N}\) into a single inequality yields exactly \(\binom{N-1}{j-1}\) possibilities. They are exactly \eqref{2ineq}.
\end{proof}

\begin{proposition}\label{bkbN-k}
Let \(1\leq k<N\). Then we have the identity
\begin{align*}
\sum_{\sigma\in A_{k;N}}b_{i_{\sigma(1)},i_{\sigma(2)},\cdots,i_{\sigma(N)}}= b_{i_1,\cdots,i_k}b_{i_{k+1},\cdots,i_N}.
\end{align*}
\end{proposition}
\begin{proof}
We first study \(2^{N-2}b_{i_1,\cdots,i_k}b_{i_{k+1},\cdots,i_N}\). By Lemma~\ref{partition}, expanding this product gives terms of the form
\begin{align*}
\mathbb P(\text{some total order inequality of }\xi_{1},\cdots,\xi_{k})\cdot \mathbb P(\text{some total order inequality of }\xi_{k+1},\cdots,\xi_{N}).
\end{align*}
Combining the two inequalities into one yields exactly \(\binom{N}{k}\) mutually exclusive events. Hence \(2^{N-2}b_{i_1,\cdots,i_k}b_{i_{k+1},\cdots,i_N}\) is the sum of probabilities of \(2^{N-2}\binom{N}{k}\) mutually exclusive events, each of which is a total order inequality of \(\xi_{1},\cdots,\xi_{N}\). From now on, the “events” for \(b_{\cdots}\) refer to the events appearing in the probability sum of its expansion, see Lemma~\ref{partition}.

We shall prove that \(2^{N-2}\sum_{\sigma\in A_{k;N}}b_{i_{\sigma(1)},i_{\sigma(2)},\cdots,i_{\sigma(N)}}\) also consists of \(2^{N-2}\binom{N}{k}\) mutually exclusive events, each being a total order inequality of \(\xi_{1},\cdots,\xi_{N}\), and each such event also appears exactly in the expansion of \(2^{N-2}b_{i_1,\cdots,i_k}b_{i_{k+1},\cdots,i_N}\).

For \(\sigma\in A_{k;N}\), consider \(2^{N-1}b_{i_{\sigma(1)},i_{\sigma(2)},\cdots,i_{\sigma(N)}}\). By Lemma~\ref{partition}, this is the sum of probabilities of \(2^{N-1}\) mutually exclusive events. Pick any one of these events, say coming from the inequality
\begin{align}\label{ineqbN}
\xi_{{\sigma(1)}}<\cdots<\xi_{{\sigma(j)}}>\cdots>\xi_{{\sigma(N)}}
\end{align}
and for brevity we restate its properties as follows:
\begin{enumerate}
\item \(\sigma\in A_{k;N}\);
\item This event is a total order inequality of \(\xi_{1},\cdots,\xi_{N}\) obtained by combining \(\xi_{{\sigma(j)}}>\cdots>\xi_{{\sigma(1)}}\) and \(\xi_{{\sigma(j)}}>\cdots>\xi_{{\sigma(N)}}\). Write it as \(\xi_{\tau(1)}>\xi_{\tau(2)}>\cdots> \xi_{\tau(N)}\), then \(\tau(1)=\sigma(j)\).
\end{enumerate}
We want to show that there exists exactly one \(\sigma'\ne \sigma\in A_{k;N}\) such that \(2^{N-1}b_{i_{\sigma'(1)},i_{\sigma'(2)},\cdots,i_{\sigma'(N)}}\) also contains this event.

\noindent\textbf{Case 1: }\(\sigma(j)\in\{k,k-1,\cdots,1\}\). Consider the positive integer \(l\) such that \(\sigma(l)\notin\{k,k-1,\cdots,1\}\) and \(\xi_{{\sigma(l)}}=\max_{t=k+1,\cdots,N} \xi_{t}\). When we speak of elements to the left or right of \(\sigma(\cdot)\), we mean the order in the permutation \((\sigma(1),\cdots,\sigma(N))\). If \(l<j\), we define a permutation that moves \(\sigma(l)\) to some position to the right of \(\sigma(j)\); specifically, we successively swap \(\sigma(l)\) with the elements to its right, subject to:
\begin{enumerate}
\item We must move \(\sigma(l)\) to the right of \(\sigma(j)\);
\item During this process we cannot perform swaps of the form \((\sigma(l),\sigma(l'))\) where \(\sigma(l')\in \{k+1,\cdots,N\}\);
\item After these swaps we obtain a permutation \(\gamma\in A_{k;N}\) different from \(\sigma\) (not necessarily unique). Clearly, in the expansion \eqref{bN} of \(2^{N-1}b_{i_{\gamma(1)},i_{\gamma(2)},\cdots,i_{\gamma(N)}}\) there is exactly one event that has \(\xi_{{\sigma(j)}}\) as the maximum;
\item The event corresponds to an inequality, and we require that this inequality be compatible with the inequality \(\xi_{\tau(1)}>\xi_{\tau(2)}>\cdots> \xi_{\tau(N)}\), i.e., no contradiction. This requirement restricts the choice of \(\gamma\).
\end{enumerate}
By definition there exists a unique \(\gamma\) satisfying the above conditions; then \(\gamma\) is the desired \(\sigma'\). We now verify the uniqueness of \(\sigma'\). Since \(\sigma'\ne\sigma\), \(\sigma'\) must either move some element from the left of \(\sigma(j)\) to the right of \(\sigma(j)\), or move some element from the right of \(\sigma(j)\) to the left of \(\sigma(j)\) (it is impossible to keep left elements on the left, otherwise the inequality \eqref{ineqbN} would not hold; similarly for the right side). By the definition of \(l\), \(\sigma'\) must move \(\sigma(l)\), therefore \(\sigma'\) satisfies the four conditions above, proving uniqueness.

If \(l>j\), a similar argument works.

\noindent\textbf{Case 2: }\(\sigma(j)\in\{k+1,\cdots,N\}\). Similar to Case 1.

Thus we have shown that \(2^{N-2}\sum_{\sigma\in A_{k;N}}b_{i_{\sigma(1)},i_{\sigma(2)},\cdots,i_{\sigma(N)}}\) also consists of \(2^{N-2}\binom{N}{k}\) mutually exclusive events, each being a total order inequality of \(\xi_{1},\cdots,\xi_{N}\). Next we prove that each such event also appears in the expansion of \(2^{N-2}b_{i_1,\cdots,i_k}b_{i_{k+1},\cdots,i_N}\).

For \(\sigma\in A_{k;N}\), consider any event \eqref{ineqbN} of \(2^{N-1}b_{i_{\sigma(1)},i_{\sigma(2)},\cdots,i_{\sigma(N)}}\). Deleting those \(X\) whose indices lie in \(\{1,2,\cdots,k\}\) from the inequality \eqref{ineqbN} yields an inequality consequence; this event obviously lies among the events of \(b_{i_{k+1},\cdots,i_N}\). Similarly, deleting those \(X\) whose indices lie in \(\{k+1,\cdots,N\}\) yields an inequality consequence that lies among the events of \(b_{i_{1},\cdots,i_k}\). This completes the proof of the proposition.
\end{proof}

\begin{corollary}\label{altersum}
\begin{enumerate}\item If \(N\) is even, then
\begin{align} \label{eq_alt_sum_1}
2b_{i_1,\cdots,i_N}=b_{i_1,\cdots,i_{N-1}}b_{i_N}-b_{i_1,\cdots,i_{N-2}}b_{i_{N-1},i_N}+b_{i_1,\cdots,i_{N-3}}b_{i_{N-2},i_{N-1},i_N}-\cdots+b_{i_1}b_{i_2,\cdots,i_{N}}.
\end{align}
\item If \(N\) is odd, then
\begin{align}  \label{eq_alt_sum_2}
0=b_{i_1,\cdots,i_{N-1}}b_{i_N}-b_{i_1,\cdots,i_{N-2}}b_{i_{N-1},i_N}+b_{i_1,\cdots,i_{N-3}}b_{i_{N-2},i_{N-1},i_N}-\cdots-b_{i_1}b_{i_2,\cdots,i_{N}}.
\end{align}
\item For any $j=1,2,\cdots,N$, we have
\begin{align}\label{redeqb1}
(-1)^j\sum_{\sigma\in A'_{j;N}}b_{i_{\sigma(1)},\cdots,i_{\sigma(N)}}=-b_{i_1,\cdots,i_N}+\sum_{l=1}^{j-1}(-1)^{l-1}b_{i_1,\cdots,i_l}b_{i_{l+1},\cdots,i_N}.
\end{align}
\end{enumerate}
\end{corollary}
\begin{proof}
1. First consider the case when \(N\) is even. By Proposition~\ref{bkbN-k}, we have
\begin{align*}
\sum_{k=1}^{N-1}(-1)^{k-1}\sum_{\sigma\in A_{k;N}}b_{i_{\sigma(1)},i_{\sigma(2)},\cdots,i_{\sigma(N)}}= \sum_{k=1}^{N-1}(-1)^{k-1}b_{i_1,\cdots,i_k}b_{i_{k+1},\cdots,i_N}.
\end{align*}
Recall \(A_{k;N}\) can be partitioned into two parts \(A'_{k;N}\coprod A''_{k;N}\) in Definition~\ref{AkN}. It is easy to see that \(A'_{k;N}= A''_{k-1;N}\). Therefore
\begin{align*}
&\sum_{k=1}^{N-1}(-1)^{k-1}\sum_{\sigma\in A_{k;N}}b_{i_{\sigma(1)},i_{\sigma(2)},\cdots,i_{\sigma(N)}}\\
=&\Bigl(\sum_{\sigma\in A'_{1;N}}+\sum_{\sigma\in A''_{1;N}}-\sum_{\sigma\in A'_{2;N}}-\sum_{\sigma\in A''_{2;N}}+\cdots+\sum_{\sigma\in A'_{N-1;N}}+\sum_{\sigma\in A''_{N-1;N}}\Bigr)b_{i_{\sigma(1)},i_{\sigma(2)},\cdots,i_{\sigma(N)}}\\
=&\Bigl(\sum_{\sigma\in A'_{1;N}}+\sum_{\sigma\in A''_{N-1;N}}\Bigr)b_{i_{\sigma(1)},i_{\sigma(2)},\cdots,i_{\sigma(N)}}=b_{i_1,\cdots,i_N}+b_{i_N,i_{N-1},\cdots,i_1}=2b_{i_1,\cdots,i_N}.
\end{align*}
2. The case when \(N\) is odd is proved similarly.

\noindent 3. We prove \eqref{redeqb1} by an induction on $j$. For $j=1$, \eqref{redeqb1} is $-b_{i_1,\cdots,i_N}=-b_{i_1,\cdots,i_N}$. Assume \eqref{redeqb1} holds for $1\leq j<N$, we prove \eqref{redeqb1} holds for $j+1$. By induction and Proposition~\ref{bkbN-k}, we have
\begin{align*}&-(-1)^{j+1}\sum_{\sigma\in A'_{j+1;N}}b_{i_{\sigma(1)},\cdots,i_{\sigma(N)}}-b_{i_1,\cdots,i_N}+\sum_{l=1}^{j}(-1)^{l-1}b_{i_1,\cdots,i_l}b_{i_{l+1},\cdots,i_N}\\=&(-1)^{j}\sum_{\sigma\in A'_{j+1;N}}b_{i_{\sigma(1)},\cdots,i_{\sigma(N)}}+(-1)^j\sum_{\sigma\in A'_{j;N}}b_{i_{\sigma(1)},\cdots,i_{\sigma(N)}}+(-1)^{j-1}b_{i_1,\cdots,i_j}b_{i_{j+1},\cdots,i_N}\\=&(-1)^{j}\sum_{\sigma\in A_{j;N}}b_{i_{\sigma(1)},\cdots,i_{\sigma(N)}}+(-1)^{j-1}b_{i_1,\cdots,i_j}b_{i_{j+1},\cdots,i_N}=0.\qedhere 
\end{align*}
\end{proof}

\begin{lemma}For $l=1,\cdots,N$, we have
\begin{equation}\label{probmountain}\begin{aligned}&\mathbb P(\xi_1<\cdots<\xi_{l-1}<\xi_l>\xi_{l+1}>\cdots>\xi_N)\\
&=(\prod_{j=2}^{l-1}\frac{i_j}{i_j+i_{j-1}+\cdots+i_1})\frac{i_l}{i_1+\cdots+i_N}(\prod_{k=l+1}^{N-1}\frac{i_k}{i_k+i_{k+1}+\cdots+i_N}).
\end{aligned}\end{equation}
For $l=3,\cdots,N$ and $j=2,\cdots,l-1$, we have
\begin{equation}\label{probmountain'}\begin{aligned}&\frac{i_{j}}{i_{j-1}+i_{j}}\cdot\mathbb P(\xi_1<\cdots<\xi_{j-2}<\max(\xi_{j-1},\xi_j)<\xi_{j+1}<\cdots<\xi_l>\cdots>\xi_N)\\
&=\mathbb P((\xi_1<\cdots<\xi_{j-2}<\xi_j<\xi_{j+1}<\cdots<\xi_l>\cdots>\xi_N)\land(\xi_j>\xi_{j-1})).
\end{aligned}\end{equation}
\end{lemma}
\begin{proof}Firstly, we have
\begin{align*}&\mathbb P(\xi_1<\cdots<\xi_{l-1}<\xi_l>\xi_{l+1}>\cdots>\xi_N)\\
=&(\prod_{j=2}^{l-1}\mathbb P(\max(\xi_1,\cdots,\xi_{j-1})<\xi_j))\mathbb P(\xi_l>\max(\xi_1,\cdots,\xi_{l-1},\xi_{l+1},\cdots,\xi_N))\cdot\\
&(\prod_{k=l+1}^{N-1}\mathbb P(\xi_k>\max(\xi_{k+1},\cdots,\xi_N))).
\end{align*}
Then we get \eqref{probmountain} by the following facts
\begin{align*}\max(\xi_1,\cdots,\xi_{j-1})&\sim \operatorname{Beta}(i_1+\cdots+i_{j-1},1);\quad 
\mathbb P(\xi_2>\xi_1)=\frac{i_2}{i_2+i_1}.
\end{align*}
One may derive \eqref{probmountain'} from \eqref{probmountain} directly.
\end{proof}
\begin{lemma}For any positive integers $i_1,\cdots,i_{N+1}$, we have
\begin{align}\label{recursiveb}
b_{i_{N+1}+i_1,i_2,\cdots,i_{N}}-{b}_{i_1,i_2,\cdots,i_{N}}+ \frac{i_2}{i_2+i_1}\cdot(\cc\cdot b_{{i_1+i_2},i_3,\cdots,i_{N}}-b_{i_{N+1},i_1+i_2,i_3,\cdots,i_{N}})=0.
\end{align}
\end{lemma}
\begin{proof}Let $\xi_{1},\cdots,\xi_{N+1}$ be independent random variables with distributions $$\operatorname{Beta}(i_1,1),\cdots,\operatorname{Beta}(i_{N+1},1)$$ respectively, where $i_1,\cdots,i_{N+1}$ are positive integers. Then
\begin{align*}&2^{N-1}\frac{i_2}{i_2+i_1}\cdot(\cc\cdot b_{{i_1+i_2},i_3,\cdots,i_{N}}-{b}_{i_{N+1},i_1+i_2,i_3,\cdots,i_{N}})\\
=&\sum_{l=3}^{N}\frac{i_2}{i_2+i_1}\cdot(\mathbb P(\max(\xi_1,\xi_2)<\xi_3<\cdots<\xi_l>\cdots>\xi_{N})\\
&\makebox[.8cm]{}-\mathbb P(\xi_{N+1}<\max(\xi_1,\xi_2)<\xi_3<\cdots<\xi_l>\cdots>\xi_{N}))\\
\stackrel{\eqref{probmountain'}}{=}&\sum_{l=3}^{N}(\mathbb P(\xi_1<\xi_2<\xi_3<\cdots<\xi_l>\cdots>\xi_{N})-\mathbb P(\max(\xi_{N+1},\xi_1)<\xi_2<\cdots<\xi_l>\cdots>\xi_{N}))\\
=&2^{N-1}(b_{i_1,\cdots,i_{N}}-b_{i_{N+1}+i_1,i_2,\cdots,i_{N}}).
\end{align*}

As a byproduct of this proof, we have
\begin{equation}\begin{aligned}\label{abN}&\frac{i_2}{i_2+i_1}\cdot 2^{N-2}b_{i_1+i_2,i_3,\cdots,i_{N}}\\
=&\sum_{l=3}^{N}\mathbb P(\xi_1<\xi_2<\cdots<\xi_l>\cdots>\xi_{N})+\frac{i_2}{i_2+i_1}\cdot\mathbb P(\max(\xi_1,\xi_2)>\xi_3>\cdots>\xi_{N})\\
=&b_{i_1,i_2,\cdots,i_{N}}-\mathbb P(\xi_2>\cdots>\xi_{N})+\frac{i_2}{i_2+i_1}\cdot\mathbb P(\max(\xi_1,\xi_2)>\xi_3>\cdots>\xi_{N})\\
=&b_{i_1,i_2,\cdots,i_{N}}-\mathbb P(\xi_1>\xi_2>\cdots>\xi_{N}),
\end{aligned}\end{equation}
because $\frac{i_2}{i_2+i_1}\cdot\mathbb P(\max(\xi_1,\xi_2)>\xi_3>\cdots>\xi_{N})=\mathbb P((\xi_2>\xi_1)\land(\xi_2>\xi_3>\cdots>\xi_{N}))$.
\end{proof}
\begin{lemma}For $l=2,\cdots,N$, we have
\begin{align}\label{altersumprob}\mathbb P(\xi_1<\cdots<\xi_{N+1})=\sum_{j=l-1}^N (-1)^{j-l+1}\mathbb P(\xi_1<\cdots<\xi_l>\cdots>\xi_{j+1})\mathbb P(\xi_{j+2}<\cdots<\xi_{N+1}).
\end{align}
(When $j+2\geq N+1$, we denote $\mathbb P(\xi_{j+2}<\cdots<\xi_{N+1}):=1$.)
\end{lemma}
\begin{proof}Note that
\begin{align*}&(\xi_1<\cdots<\xi_l)\land(\xi_{l+1}<\cdots<\xi_{N+1})\\=&(\xi_1<\cdots<\xi_{N+1})\lor \big\{(\xi_1<\cdots<\xi_l>\xi_{l+1})\land(\xi_{l+1}<\cdots<\xi_{N+1})\big\};\\
&(\xi_1<\cdots<\xi_l>\xi_{l+1})\land(\xi_{l+2}<\cdots<\xi_{N+1})\\
=&\big\{(\xi_1<\cdots<\xi_l>\xi_{l+1})\land(\xi_{l+1}<\cdots<\xi_{N+1})\big\}\lor\\&\big\{ (\xi_1<\cdots<\xi_l>\xi_{l+1}>\xi_{l+2})\land(\xi_{l+2}<\cdots<\xi_{N+1})\big\},\quad\cdots,\\
&(\xi_1<\cdots<\xi_l>\cdots>\xi_{N})\\
=&\big\{(\xi_1<\cdots<\xi_l>\cdots>\xi_N)\land(\xi_N<\xi_{N+1})\big\}\lor(\xi_1<\cdots<\xi_l>\cdots>\xi_{N+1}).
\end{align*}
Take the alternative sum of their probabilities and we get the desired identity.
\end{proof}

\section{NHC of relative moduli spaces is real analytic} 
Let
\[
  \pi:X\longrightarrow S
\]
be a smooth projective morphism of complex manifolds.

Fix a topological type of complex vector bundle such that the usual non-abelian Hodge correspondence applies.  In our applications, this means that the rational Chern classes vanish.  Let
\[
  M_{\Dol}(X/S)
  \qquad\text{and}\qquad
  M_{\Betti}(X/S)
\]
be respectively the relative Dolbeault moduli space and the relative Betti moduli space.

In this section, we prove the following result, which has been proven by \cite[Theorem 4.23]{CTW} when any fiber of $\pi$ is a compact Riemann surface.

\begin{thm}[Relative real analyticity]
\label{R-analyticity}
The relative non-abelian Hodge correspondence
\[
  \NHC:M_{\Dol}(X/S)
  \longrightarrow
  M_{\Betti}(X/S)
\]
is a real analytic isomorphism near their smooth points. Moreover, for any holomorphic family of flat bundles over $X$, the corresponding family of Higgs bundles over $X$ by taking $NHC$ is a real analytic family of Higgs bundles.
\end{thm}

The key analytic input used in the proof is the analytic regularity theorem of Morrey (cf. \cite{MorreyNonlinear}), which is reviewed here in the form we need:

\begin{proposition}[Analytic dependence for elliptic systems]
\label{prop:analytic-dependence}
Let $M$ be a compact real analytic manifold and let $Q$ be a finite-dimensional real analytic parameter space.  Suppose
\[
  \mathcal F(q,u)=0
\]
is a second-order nonlinear strongly elliptic system for a section $u$ of a real analytic vector bundle over $M$.  Assume that, in real analytic local coordinates and trivializations, the coefficients of $\mathcal F$ are real analytic in
\[
  (q,x,u,Du,D^2u).
\]
Then, near any solution $(q_0,u_0)$, the solution $u_q$ is real analytic as a function of $q$, and the corresponding section
\[
  (q,x)\longmapsto u_q(x)
\]
is real analytic on $Q\times M$.
\end{proposition}

\begin{proof}[Proof of Theorem~\ref{R-analyticity}]

Note that the assertion is local on the source and target.
Let $\Delta\subset S$ be a sufficiently small coordinate polydisc. After shrinking $\Delta$, choose a real analytic Ehresmann trivialization
\[
  X_\Delta:=\pi^{-1}(\Delta)\cong X_0\times \Delta
\]
where $X_0$ is a fixed compact smooth manifold. Let $J_s$ denote the induced complex structure on $X_0$ and let $\omega_s$ be a fiberwise K\"ahler form obtained from a relative polarization.  We regard both $J_s$ and $\omega_s$ as real analytic functions of $s\in \Delta$.

Now fix a point
\[
  (s_0,E_0,\theta_0)\in M_{\Dol}(X_\Delta/\Delta)
\]
and work after shrinking $\Delta$.  Choose a smooth bundle $\E\to X_0$ representing the fixed topological type, and choose a reference Hermitian metric $h_0$ on $\E$.

Without loss of generality we may assume that $(E_0,\theta_0)$ is stable. There is a finite-dimensional real analytic slice $\Sigma$ for the relative Dolbeault moduli problem near $(s_0,E_0,\theta_0)$ (i,e, $\Sigma$ is transversal to the fiber of $M_{\Dol}(X_\Delta/\Delta) \to \Delta$).  A point $q\in\Sigma$ determines
\[
  q=(s,\bar\partial_q,\theta_q),
\]
where $s\in \Delta$, $\bar\partial_q$ is a holomorphic structure on the fixed smooth bundle over $(X_0,J_s)$, and $\theta_q$ is a Higgs field satisfying
\[
  \bar\partial_q\theta_q=0,
  \qquad
  \theta_q\wedge\theta_q=0.
\]

By construction of the Kuranishi slice, and by the real analytic trivialization of the family, the coefficients of $J_s$, $\omega_s$, $\bar\partial_q$ and $\theta_q$ depend real analytically on $q$ and on the fiber variable.

For each $q\in\Sigma$, non-abelian Hodge theory supplies a harmonic metric $h_q$.  Write
\[
  h_q=h_0 k_q,
\]
where $k_q$ is a positive $h_0$-self-adjoint endomorphism.  If the structure group has a nontrivial center, fix the determinant metric, or equivalently work modulo the central scalar ambiguity, so that $k_q$ is uniquely determined near $q_0$.

Recall the harmonic metric equation for $h_q$ (cf. \cite{Hit,Simp92})
\begin{equation}
\label{eq:higher-dimensional-HYM-Higgs}
  \sqrt{-1}\Lambda_{\omega_s}
  \left(
    F_{h_q,\bar\partial_q}
    +[\theta_q,\theta_q^{\star_{h_q}}]
  \right)
  = \lambda_q\,\id_{\E}.
\end{equation}

Here $\lambda_q$ is the topological constant, which is zero in the vanishing-Chern-class case.  Rewriting \eqref{eq:higher-dimensional-HYM-Higgs} in terms of $k_q$ and $h_0$ gives
\begin{equation}
\label{eq:k-equation}
  \sqrt{-1}\Lambda_{\omega_s}
  \left(
    \bar\partial_q\bigl(k_q^{-1}\partial_{0,q}k_q\bigr)
    +F_{h_0,\bar\partial_q}
    +[\theta_q,k_q^{-1}\theta_q^{\st}k_q]
  \right)
  =\lambda_q\,\id_{\E},
\end{equation}
where $\partial_{0,q}$ is the $(1,0)$ part of the Chern connection determined by $(\bar\partial_q,h_0)$.

Equation \eqref{eq:k-equation} is a second-order nonlinear elliptic system for $k_q$.  Its principal part is the Laplacian-type operator
\[
  \sqrt{-1}\Lambda_{\omega_s}\bar\partial_q\partial_{0,q},
\]
acting on Hermitian endomorphisms, and is strongly elliptic because $\omega_s$ is positive.  Moreover, the coefficients of \eqref{eq:k-equation} are real analytic in $(q,x)$ and analytic in $k_q$ and its derivatives as long as $k_q$ remains positive and invertible.  By uniqueness of the harmonic metric, Proposition~\ref{prop:analytic-dependence} applies.  Therefore
\[
  (q,x)\longmapsto k_q(x)
\]
is real analytic.  Hence the harmonic metric $h_q=h_0k_q$ is real analytic in $q$ and in the fiber variable.

The flat connection associated with the harmonic bundle is
\begin{equation}
\label{eq:flat-connection-from-Higgs}
  D_q
  =\bar\partial_q+\partial_{h_q,\bar\partial_q}
   +\theta_q+\theta_q^{\star_{h_q}}.
\end{equation}
Since $\bar\partial_q$, $\theta_q$ and $h_q$ depend real analytically on $q$, the connection one-form of $D_q$ in any real analytic trivialization depends real analytically on $(q,x)$.

The inverse direction is proved by the same argument.  Start with a real analytic slice in the Betti moduli space, represented by a real analytic family of flat connections $D_q$ on the fixed smooth bundle over $X_0$.  The Corlette--Simpson harmonic metric equation for $D_q$ (cf. \cite{Simp92, Corl}), with respect to the complex structure $J_s$ and Kahler metric $\omega_s$, is again a second-order nonlinear strongly elliptic analytic system for the metric endomorphism $k_q$.  After the same determinant or central normalization, uniqueness of the harmonic metric and Proposition~\ref{prop:analytic-dependence} imply that $k_q$ depends real analytically on $q$.  Decomposing $D_q$ using this harmonic metric and the complex structure $J_s$ gives
\[
  D_q=(\bar\partial_{E,q}+\theta_q)+(\partial_{E,q}^{h_q}+\theta_q^{\star_{h_q}}),
\]
so the resulting holomorphic structure and Higgs field are real analytic in $q$.
\end{proof}

\section{Real analytic manifolds and real analytic deformations}
\label{sec_complexification}

In this appendix we spell out the convention used in the paper for
real analytic deformations of holomorphic objects.  The guiding principle is
that a real analytic family is obtained by restricting a holomorphic family on
the complexification of the real analytic base to the diagonal.

\subsection{Real analytic functions and complexification}

Let \(\mathbb D_{\epsilon}\subset \mathbb R^2\) be a sufficiently small disk
centered at the origin.  A function
\[
  f:\mathbb D_{\epsilon}\longrightarrow \C
\]
is said to be real analytic near \(0\) if, after writing
\[
  z=x+\sqrt{-1}y,\qquad
  \bar z=x-\sqrt{-1}y,
\]
it admits a convergent expansion
\[
  f(x,y)=\sum_{i,j\geq 0}a_{ij}z^i\bar z^j
\]
near \(0\).  Two such functions define the same germ at \(0\) if they agree on
some smaller neighborhood of \(0\).

Let $\mathcal R_0$ be the ring of germs of \(\C\)-valued real analytic functions at
\(0\in\mathbb R^2\), and let \(\mathcal R\) be the corresponding sheaf on
\(\mathbb R^2\).

We regard $\C=\mathbb R^2$ as the complex line with its standard complex structure, and $\overline{\C}$ as the same real vector space with the opposite complex structure.  We write
\(z\) for the holomorphic coordinate on \(\C\), and write \(\zeta\) for the
holomorphic coordinate on \(\overline{\C}\).  Under the diagonal embedding
\[
  i:\mathbb R^2\hookrightarrow \C\times\overline{\C},
  \qquad
  (x,y)\longmapsto (z,\bar z),
\]
one has
\[
  i^*z=z,\qquad i^*\zeta=\bar z.
\]

\begin{lemma}
\label{lem:real-analytic-functions-complexification}
There is a natural isomorphism of sheaves of \(\C\)-algebras
\[
  \mathcal R\simeq i^*\mathcal O_{\C\times\overline{\C}} .
\]
Equivalently, every real analytic germ in the variables \((z,\bar z)\) is the
restriction of a holomorphic germ in the independent variables \((z,\zeta)\).
\end{lemma}

\begin{proof}
The assertion is local.  A holomorphic function on a sufficiently small
polydisc in \(\C\times\overline{\C}\) has a convergent expansion
\[
  F(z,\zeta)=\sum_{i,j\geq 0}a_{ij}z^i\zeta^j.
\]
Restricting to the diagonal \(\zeta=\bar z\) gives the real analytic function
\[
  i^*F(z,\bar z)=\sum_{i,j\geq 0}a_{ij}z^i\bar z^j .
\]
Conversely, if
\[
  f(z,\bar z)=\sum_{i,j\geq 0}a_{ij}z^i\bar z^j
\]
is a convergent real analytic germ, then
\[
  F(z,\zeta)=\sum_{i,j\geq 0}a_{ij}z^i\zeta^j
\]
is a convergent holomorphic germ on a sufficiently small polydisc in
\(\C\times\overline{\C}\), and \(i^*F=f\).  These local constructions are
compatible with restrictions, hence glue to the desired isomorphism.

By a similar argument, this complexification lemma also holds similarly for real analytic germs at $0\in\mathbb R^{2k}$.
\end{proof}

Let \(M\) be a complex manifold and let \(\overline M\) be the conjugate
complex manifold.  We write \(M^\circ\) for the underlying real analytic
manifold.  The diagonal embedding is
\[
  i_{M^\circ}:M^\circ\hookrightarrow M\times\overline M .
\]
We define the sheaf of \(\C\)-valued real analytic functions on \(M^\circ\) by
\[
  \mathcal R_{M^\circ}:=i_{M^\circ}^*\mathcal O_{M\times\overline M}.
\]
Thus the associated real analytic space is the locally ringed space
\[
  (M^\circ,\mathcal R_{M^\circ}).
\]

The complexified tangent bundle of \(M^\circ\) decomposes as
\[
  \C TM^\circ
  =
  TM^\circ\otimes_{\mathbb R}\C
  =
  T^{1,0}M^\circ\oplus T^{0,1}M^\circ .
\]
Under the above complexification, this is naturally identified with
\[
  \C TM^\circ
  \cong
  i_{M^\circ}^*T_M\oplus i_{M^\circ}^*T_{\overline M}.
\]

\subsection{Infinitesimal real analytic disks}

For \(n\geq 1\), define
\[
  A_n:=\C[t]/(t^{n+1}),
  \qquad
  \overline A_n:=\C[\bar t]/(\bar t^{n+1}),
\]
and
\[
  B_n:=\C[t,\bar t]/(t,\bar t)^{n+1}.
\]
Here \(\bar t\) is a formal variable.  Geometrically, \(t\) is the holomorphic
coordinate on the first factor of the complexified disk, while \(\bar t\) is
the holomorphic coordinate on the conjugate factor.

For \(n=1\), we have
\[
  B_1=\C[t,\bar t]/(t,\bar t)^2
      =\C\oplus \C t\oplus \C \bar t,
\]
and its maximal ideal
\[
  I_{B_1}:=(t,\bar t)
\]
satisfies \(I_{B_1}^2=0\).  Thus
\[
  I_{B_1}\cong \C t\oplus \C\bar t.
\]
Equivalently,
\[
  B_1\simeq A_1\times_{\C}\overline A_1.
\]

\begin{proposition}
\label{prop:real-analytic-first-jets}
Let \(x\in M^\circ\).  Then the set of first order real analytic arcs in
\(M^\circ\) through \(x\) is naturally identified with the complexified tangent
space:
\[
  \mathrm{Hom}_{\C\text{-alg}}
  \bigl(\mathcal R_{M^\circ,x},B_1\bigr)
  \cong
  \C T_xM^\circ .
\]
More explicitly,
\[
  \mathrm{Hom}_{\C\text{-alg}}
  \bigl(\mathcal R_{M^\circ,x},B_1\bigr)
  \cong
  T_{M,x}\oplus T_{\overline M,x}.
\]
For general \(n\), we define the space of \(n\)-jets of real analytic arcs
through \(x\) by
\[
  J_n^{\mathbb R\mathrm{an}}(M^\circ)_x
  :=
  \mathrm{Hom}_{\C\text{-alg}}
  \bigl(\mathcal R_{M^\circ,x},B_n\bigr).
\]
\end{proposition}

\begin{proof}
Choose holomorphic coordinates \(z_1,\ldots,z_m\) on \(M\) centered at \(x\),
and let \(\zeta_1,\ldots,\zeta_m\) be the corresponding holomorphic
coordinates on \(\overline M\).  Then
\[
  \mathcal R_{M^\circ,x}
  \cong
  \C\{z_1,\ldots,z_m,\zeta_1,\ldots,\zeta_m\}.
\]
A \(\C\)-algebra homomorphism
\[
  \phi:\mathcal R_{M^\circ,x}\longrightarrow B_1
\]
whose residue is \(x\) is determined by
\[
  \phi(z_i)=a_i t+b_i\bar t,
  \qquad
  \phi(\zeta_i)=c_i t+d_i\bar t.
\]
Restricting to the diagonal real analytic structure amounts to recording the
two independent tangent directions along the two factors \(M\) and
\(\overline M\).  Hence the first order part is precisely an element of
\[
  T_{M,x}\oplus T_{\overline M,x}
  \cong
  \C T_xM^\circ.
\]
The statement for \(B_n\) is the same construction with higher order
coefficients retained up to total degree \(n\).
\end{proof}

\subsection{Real analytic deformations of complex manifolds}

Let \(X\) be a complex manifold.  A real analytic deformation of \(X\) over
\(M^\circ\) is a diagram
\[
\begin{tikzcd}
X \arrow[r, hook] \arrow[d]
&
\mathcal X^\circ \arrow[r, hook] \arrow[d]
&
\mathcal X \arrow[d]
\\
\operatorname{Spec}\C \arrow[r, hook]
&
M^\circ \arrow[r, hook, "i_{M^\circ}"]
&
M\times\overline M
\end{tikzcd}
\]
such that \(\mathcal X\to M\times\overline M\) is a holomorphic deformation,
and
\[
  \mathcal X^\circ
  =
  \mathcal X\times_{M\times\overline M}M^\circ .
\]
Equivalently, a real analytic deformation is the restriction to the diagonal
of a holomorphic deformation over the complexification of the base.

For an infinitesimal real analytic deformation over \(B_n\), we have a
cartesian diagram
\[
\begin{tikzcd}
X \arrow[r, hook] \arrow[d]
&
X_n \arrow[d, "\pi"]
\\
\operatorname{Spec}\C \arrow[r, hook]
&
\operatorname{Spec}B_n .
\end{tikzcd}
\]

Assume now \(n=1\).  The cotangent sequence for
\(\pi:X_1\to\operatorname{Spec}B_1\), after restricting to the central fiber,
gives
\[
  0
  \longrightarrow
  I_{B_1}^{\vee}\otimes_{\C}\mathcal O_X
  \longrightarrow
  \Omega^1_{X_1}\big|_X
  \longrightarrow
  \Omega_X^1
  \longrightarrow
  0.
\]
Its extension class is the real analytic Kodaira--Spencer class
\[
  \mathrm{KS}^{\mathbb R\mathrm{an}}(X_1)
  \in
  \mathrm{Ext}^1_X
  \bigl(\Omega_X^1,I_{B_1}^{\vee}\otimes_{\C}\mathcal O_X\bigr).
\]
Since \(X\) is smooth,
\[
  \mathrm{Ext}^1_X(\Omega_X^1,\mathcal O_X)
  =
  H^1(X,T_X).
\]
Therefore
\[
  \mathrm{KS}^{\mathbb R\mathrm{an}}(X_1)
  \in
  H^1(X,T_X)\otimes_{\C}I_{B_1}^{\vee}.
\]
Using
\[
  I_{B_1}^{\vee}\cong \C\,dt\oplus\C\,d\bar t,
\]
we obtain a decomposition
\[
  \mathrm{KS}^{\mathbb R\mathrm{an}}(X_1)
  =
  \mathrm{KS}_t(X_1)\,dt+\mathrm{KS}_{\bar t}(X_1)\,d\bar t
\]
with
\[
  \mathrm{KS}_t(X_1),\mathrm{KS}_{\bar t}(X_1)\in H^1(X,T_X).
\]

If one records the \(\bar t\)-direction after conjugating the central fiber,
then the same decomposition is written as
\[
  \mathrm{KS}^{\mathbb R\mathrm{an}}(X_1)
  \in
  H^1(X,T_X)
  \oplus
  H^1(\overline X,T_{\overline X}).
\]
The first summand is the holomorphic Kodaira--Spencer direction, and the
second summand is the anti-holomorphic direction.

\begin{proposition}
\label{prop:KS-real-analytic-manifold}
First order real analytic deformations of \(X\) over
\(\operatorname{Spec}B_1\) are classified by
\[
  H^1(X,T_X)\otimes_{\C}I_{B_1}^{\vee}
  \cong
  H^1(X,T_X)\oplus H^1(X,T_X).
\]
Equivalently, after conjugating the \(\bar t\)-direction, this classification
may be written as
\[
  H^1(X,T_X)\oplus H^1(\overline X,T_{\overline X}).
\]
\end{proposition}

\begin{proof}
Take a sufficiently fine Stein open cover \(\{U_i\}\) of \(X\).  A first order
deformation over \(B_1\) is obtained by gluing the trivial thickenings
\[
  U_i\times\operatorname{Spec}B_1
\]
by transition functions of the form
\[
  z_i
  =
  f_{ij}(z_j)
  +t\,v_{ij}(z_j)
  +\bar t\,w_{ij}(z_j),
\]
where \(f_{ij}\) are the original transition functions of \(X\), and
\(v_{ij},w_{ij}\) are holomorphic vector fields on \(U_{ij}\).  The cocycle
condition modulo \((t,\bar t)^2\) says precisely that
\[
  \{v_{ij}\}\in Z^1(\{U_i\},T_X),
  \qquad
  \{w_{ij}\}\in Z^1(\{U_i\},T_X).
\]
Changing the local trivializations modifies these cocycles by coboundaries.
Thus the isomorphism class of the deformation is determined by
\[
  \bigl([\{v_{ij}\}],[\{w_{ij}\}]\bigr)
  \in
  H^1(X,T_X)\oplus H^1(X,T_X).
\]
Conversely, any such pair of cocycles defines a first order deformation by the
above gluing formula.  This proves the classification.  The final formulation
with
\[
  H^1(\overline X,T_{\overline X})
\]
is obtained by applying complex conjugation to the \(\bar t\)-part.
\end{proof}

\subsection{Real analytic deformations of coherent sheaves on a fixed space}

Let \(X\) be a fixed complex manifold or, more generally, a fixed complex
scheme, and let \(\mathcal F\) be a coherent \(\mathcal O_X\)-module.

\begin{definition}
\label{def:ra-sheaf-fixed}
A real analytic deformation of \(\mathcal F\) over \(M^\circ\), with the space
\(X\) fixed, is a diagram
\[
\begin{tikzcd}
\mathcal F \arrow[r] \arrow[d]
&
\mathcal F^\circ \arrow[r] \arrow[d]
&
\mathcal F^{\mathbb C} \arrow[d]
\\
X\times \operatorname{Spec}\C \arrow[r, hook]
&
X\times M^\circ \arrow[r, hook]
&
X\times M\times\overline M
\end{tikzcd}
\]
such that \(\mathcal F^{\mathbb C}\) is a coherent sheaf on
\(X\times M\times\overline M\), flat over \(M\times\overline M\), and
\[
  \mathcal F^\circ
  =
  \mathcal F^{\mathbb C}
  \big|_{X\times M^\circ}.
\]
\end{definition}

For the first order base \(B_1\), this is the same as a coherent
\[
  \mathcal O_X\otimes_{\C}B_1
\]
module \(\mathcal F_1\), flat over \(B_1\), together with an isomorphism
\[
  \mathcal F_1\otimes_{B_1}\C\simeq \mathcal F.
\]
Because \(I_{B_1}^2=0\), every such deformation fits into an exact sequence of
\(\mathcal O_X\)-modules
\[
  0
  \longrightarrow
  \mathcal F\otimes_{\C}I_{B_1}
  \longrightarrow
  \mathcal F_1
  \longrightarrow
  \mathcal F
  \longrightarrow
  0.
\]
The extension class is
\[
  \mathrm{KS}^{\mathbb R\mathrm{an}}(\mathcal F_1)
  \in
  \mathrm{Ext}^1_X
  \bigl(
    \mathcal F,
    \mathcal F\otimes_{\C}I_{B_1}
  \bigr).
\]
Since
\[
  I_{B_1}\cong \C t\oplus \C\bar t,
\]
we have
\[
  \mathrm{Ext}^1_X
  \bigl(
    \mathcal F,
    \mathcal F\otimes_{\C}I_{B_1}
  \bigr)
  \cong
  \mathrm{Ext}^1_X(\mathcal F,\mathcal F)
  \oplus
  \mathrm{Ext}^1_X(\mathcal F,\mathcal F).
\]

\begin{proposition}
\label{prop:ra-fixed-sheaf-KS}
First order real analytic deformations of a coherent sheaf \(\mathcal F\) on a
fixed \(X\) over \(\operatorname{Spec}B_1\) are classified by
\[
  \mathrm{Ext}^1_X(\mathcal F,\mathcal F)\otimes_{\C}I_{B_1}
  \cong
  \mathrm{Ext}^1_X(\mathcal F,\mathcal F)
  \oplus
  \mathrm{Ext}^1_X(\mathcal F,\mathcal F).
\]
The two components are the \(t\)- and \(\bar t\)-Kodaira--Spencer classes of
the real analytic family.
\end{proposition}

\begin{proof}
The standard deformation theory of coherent sheaves over a square-zero
extension gives the classification by
\[
  \mathrm{Ext}^1_X
  \bigl(
    \mathcal F,
    \mathcal F\otimes_{\C}I_{B_1}
  \bigr).
\]
For completeness, we recall the elementary construction.  Choose an open cover
\(\{U_i\}\) on which the deformation is locally trivial.  Then the local
trivial deformations glue by automorphisms of the form
\[
  1+t\,a_{ij}+\bar t\,b_{ij},
\]
where \(a_{ij},b_{ij}\) are local endomorphisms of \(\mathcal F\).  The gluing
condition modulo \((t,\bar t)^2\) says that
\[
  \{a_{ij}\},\{b_{ij}\}
\]
are \(1\)-cocycles with values in \(\mathcal End(\mathcal F)\), or more
generally represent classes in
\[
  \mathrm{Ext}^1_X(\mathcal F,\mathcal F)
\]
when \(\mathcal F\) is not locally free.  Changing the local trivializations
changes these cocycles by coboundaries.  Therefore the isomorphism class of the
first order deformation is determined by the pair
\[
  \bigl([\{a_{ij}\}],[\{b_{ij}\}]\bigr)
  \in
  \mathrm{Ext}^1_X(\mathcal F,\mathcal F)
  \oplus
  \mathrm{Ext}^1_X(\mathcal F,\mathcal F).
\]
Conversely, a pair of such extension classes defines the required gluing data,
hence a first order real analytic deformation.
\end{proof}

\begin{remark}
If one also conjugates the fixed holomorphic space \(X\) in the
\(\bar t\)-direction, then the second summand may equivalently be written as
\[
  \mathrm{Ext}^1_{\overline X}
  (\overline{\mathcal F},\overline{\mathcal F}).
\]
In the fixed-space convention of Definition~\ref{def:ra-sheaf-fixed},
however, the two summands are both naturally Ext-groups on \(X\); the second
one is anti-holomorphic only with respect to the parameter.
\end{remark}

\subsection{Real analytic deformations of coherent sheaves on a moving space}\label{sec_joint_realanalytic}

We now discuss the case in which the ambient complex space also varies.

Let
\[
  f:\mathcal X\longrightarrow S
\]
be a holomorphic family of complex manifolds or smooth complex schemes, and let
\(0\in S\).  Put
\[
  X_0:=f^{-1}(0).
\]
Let \(\mathcal G\) be a coherent \(\mathcal O_{X_0}\)-module.

A real analytic deformation of \(\mathcal G\) along the family
\(\mathcal X\to S\) is obtained by restricting a holomorphic deformation on the
complexification.  Thus it is represented by a diagram
\[
\begin{tikzcd}
\mathcal G \arrow[r] \arrow[d]
&
\mathcal G^\circ \arrow[r] \arrow[d]
&
\mathcal G^{\mathbb C} \arrow[d]
\\
X_0 \arrow[r, hook] \arrow[d]
&
\mathcal X^\circ \arrow[r, hook] \arrow[d]
&
\mathcal X\times\overline{\mathcal X}
  \arrow[d, "{(f,\bar f)}"]
\\
\operatorname{Spec}\C \arrow[r, hook]
&
S^\circ \arrow[r, hook]
&
S\times\overline S ,
\end{tikzcd}
\]
where \(\mathcal G^{\mathbb C}\) is flat over \(S\times\overline S\), and
\[
  \mathcal G^\circ
  =
  \mathcal G^{\mathbb C}\big|_{\mathcal X^\circ}.
\]

For first order deformation theory, let \(I\) be a finite-dimensional
\(\C\)-vector space with \(I^2=0\).  Let $X_I$ be a square-zero deformation of \(X_0\) with ideal \(I\otimes_{\C}\mathcal O_{X_0}\).
Its Kodaira--Spencer class is
\[
  \kappa(X_I)
  \in
  \Ext^1_{X_0}
  \bigl(
    \Omega^1_{X_0},
    I\otimes_{\C}\mathcal O_{X_0}
  \bigr)
  \cong
  H^1(X_0,T_{X_0})\otimes_{\C}I.
\]

Let
\[
  \operatorname{At}(\mathcal G)
  \in
  \mathrm{Ext}^1_{X_0}
  \bigl(
    \mathcal G,
    \mathcal G\otimes\Omega^1_{X_0}
  \bigr)
\]
be the Atiyah class of \(\mathcal G\).  Contracting the Atiyah class with the
Kodaira--Spencer class gives
\[
  \operatorname{At}(\mathcal G)\cup\kappa(X_I)
  \in
  \mathrm{Ext}^2_{X_0}
  \bigl(
    \mathcal G,
    \mathcal G\otimes_{\C} I
  \bigr).
\]

\begin{thm}
\label{thm:ra-moving-sheaf-KS}
Let \(X_I\) be a square-zero deformation of \(X_0\) with ideal
\(I\otimes_{\C}\mathcal O_{X_0}\).

\begin{enumerate}
\item The obstruction to lifting \(\mathcal G\) to a coherent sheaf
\(\mathcal G_I\) on \(X_I\), flat over \(\C\oplus I\), is
\[
  \operatorname{At}(\mathcal G)\cup\kappa(X_I)
  \in
  \mathrm{Ext}^2_{X_0}
  \bigl(
    \mathcal G,
    \mathcal G\otimes_{\C}I
  \bigr).
\]

\item If this obstruction vanishes, the set of isomorphism classes of such
lifts is a torsor under
\[
  \mathrm{Ext}^1_{X_0}
  \bigl(
    \mathcal G,
    \mathcal G\otimes_{\C}I
  \bigr)
  \cong
  \mathrm{Ext}^1_{X_0}(\mathcal G,\mathcal G)\otimes_{\C}I.
\]

\item The infinitesimal automorphisms of a fixed lift are given by
\[
  \mathrm{Ext}^0_{X_0}
  \bigl(
    \mathcal G,
    \mathcal G\otimes_{\C}I
  \bigr).
\]
\end{enumerate}
\end{thm}

\begin{proof}
The statement is the standard Atiyah--Kodaira--Spencer obstruction theory for
a sheaf on a varying space.  We recall the construction.

The square-zero deformation \(X_I\) is classified by the extension
\[
  0
  \longrightarrow
  I\otimes_{\C}\mathcal O_{X_0}
  \longrightarrow
  \Omega^1_{X_I}\big|_{X_0}
  \longrightarrow
  \Omega^1_{X_0}
  \longrightarrow
  0,
\]
whose class is \(\kappa(X_I)\).  The Atiyah class of \(\mathcal G\) is the
extension class of the first jet sequence
\[
  0
  \longrightarrow
  \mathcal G\otimes\Omega^1_{X_0}
  \longrightarrow
  P^1(\mathcal G)
  \longrightarrow
  \mathcal G
  \longrightarrow
  0.
\]
Splicing these two extensions gives the Yoneda product
\[
  \operatorname{At}(\mathcal G)\cup\kappa(X_I)
  \in
  \mathrm{Ext}^2_{X_0}
  \bigl(
    \mathcal G,
    \mathcal G\otimes_{\C}I
  \bigr).
\]
This product measures precisely the failure of the transition functions of
\(\mathcal G\) to be lifted compatibly to the deformed structure sheaf
\(\mathcal O_{X_I}\).  Hence it is the obstruction to the existence of a lift.

If the obstruction vanishes, choices of lifted transition data differ by
\(1\)-cocycles with values in
\[
  \mathcal Hom(\mathcal G,\mathcal G\otimes_{\C}I),
\]
or, for an arbitrary coherent sheaf, by classes in
\[
  \mathrm{Ext}^1_{X_0}
  \bigl(
    \mathcal G,
    \mathcal G\otimes_{\C}I
  \bigr).
\]
Thus the set of lifts is a torsor under this Ext-group.  Automorphisms of a
fixed lift are similarly given by \(0\)-cocycles, namely by
\[
  \mathrm{Ext}^0_{X_0}
  \bigl(
    \mathcal G,
    \mathcal G\otimes_{\C}I
  \bigr).
\]
\end{proof}

Applying Theorem~\ref{thm:ra-moving-sheaf-KS} to
\[
  I=I_{B_1}=\C t\oplus \C\bar t
\]
gives the real analytic first order deformation theory.

Let
\[
  X_1^{\mathbb R\mathrm{an}}
  \longrightarrow
  \operatorname{Spec}B_1
\]
be the first order real analytic deformation of \(X_0\).  Its
Kodaira--Spencer class decomposes as
\[
  \kappa(X_1^{\mathbb R\mathrm{an}})
  =
  \kappa_t\,t+\kappa_{\bar t}\,\bar t,
\]
where
\[
  \kappa_t,\kappa_{\bar t}\in H^1(X_0,T_{X_0}).
\]
Equivalently, after conjugating the \(\bar t\)-direction, one writes
\[
  \kappa_t\in H^1(X_0,T_{X_0}),
  \qquad
  \kappa_{\bar t}\in H^1(\overline{X_0},T_{\overline{X_0}}).
\]

\begin{corollary}
\label{cor:ra-moving-sheaf-first-order}
The obstruction to a first order real analytic deformation of
\(\mathcal G\) over \(X_1^{\mathbb R\mathrm{an}}\) is the pair
\[
  \bigl(
    \operatorname{At}(\mathcal G)\cup\kappa_t,\,
    \operatorname{At}(\mathcal G)\cup\kappa_{\bar t}
  \bigr)
  \in
  \mathrm{Ext}^2_{X_0}(\mathcal G,\mathcal G)
  \oplus
  \mathrm{Ext}^2_{X_0}(\mathcal G,\mathcal G).
\]
If this pair vanishes, then the set of first order real analytic deformations
of \(\mathcal G\) over the fixed deformation
\(X_1^{\mathbb R\mathrm{an}}\) is a torsor under
\[
  \mathrm{Ext}^1_{X_0}(\mathcal G,\mathcal G)
  \oplus
  \mathrm{Ext}^1_{X_0}(\mathcal G,\mathcal G).
\]
\end{corollary}

\begin{proof}
This is Theorem~\ref{thm:ra-moving-sheaf-KS} applied to the square-zero ideal
\[
  I_{B_1}=\C t\oplus\C\bar t.
\]
The obstruction group splits according to this decomposition:
\[
  \mathrm{Ext}^2_{X_0}
  \bigl(
    \mathcal G,
    \mathcal G\otimes_{\C}I_{B_1}
  \bigr)
  \cong
  \mathrm{Ext}^2_{X_0}(\mathcal G,\mathcal G)
  \oplus
  \mathrm{Ext}^2_{X_0}(\mathcal G,\mathcal G).
\]
The same splitting holds for the torsor of lifts, giving the Ext\(^1\)
statement.
\end{proof}

It is often useful to package the deformation theory of the pair
\((X_0,\mathcal G)\) into a single complex.  The Atiyah class induces a
morphism in the derived category
\[
  T_{X_0}
  \longrightarrow
  R\mathcal Hom_{X_0}(\mathcal G,\mathcal G)[1].
\]
Define the Atiyah--Kodaira--Spencer complex of the pair by
\[
  \mathcal K_{\mathcal G}
  :=
  \operatorname{Cone}
  \left(
    T_{X_0}
    \longrightarrow
    R\mathcal Hom_{X_0}(\mathcal G,\mathcal G)[1]
  \right)[-1].
\]
Then first order deformations of the pair \((X_0,\mathcal G)\) over a
square-zero ideal \(I\) are governed by
\[
  \mathbb H^1(X_0,\mathcal K_{\mathcal G})\otimes_{\C}I.
\]
The natural long exact sequence contains
\[
\begin{aligned}
  \mathrm{Ext}^1_{X_0}(\mathcal G,\mathcal G)\otimes I
  &\longrightarrow
  \mathbb H^1(X_0,\mathcal K_{\mathcal G})\otimes I
  \longrightarrow
  H^1(X_0,T_{X_0})\otimes I
  \\
  &\xrightarrow{\operatorname{At}(\mathcal G)\cup -}
  \mathrm{Ext}^2_{X_0}(\mathcal G,\mathcal G)\otimes I .
\end{aligned}
\]
Thus the image of a deformation of the pair in \(H^1(X_0,T_{X_0})\otimes I\)
is the Kodaira--Spencer class of the moving space, and the connecting map is
the Atiyah obstruction to lifting the sheaf.

For \(I=I_{B_1}\), this becomes
\[
  \mathbb H^1(X_0,\mathcal K_{\mathcal G})
  \oplus
  \mathbb H^1(X_0,\mathcal K_{\mathcal G}).
\]
After conjugating the \(\bar t\)-direction, one may equivalently write the
real analytic tangent space of the pair as
\[
  \mathbb H^1(X_0,\mathcal K_{\mathcal G})
  \oplus
  \mathbb H^1(\overline{X_0},\mathcal K_{\overline{\mathcal G}}).
\]

\begin{definition}
\label{def:ra-KS-map-pair}
Let
\[
  f:\mathcal X\to S
\]
be a holomorphic family and let \(\mathcal G^\circ\) be a real analytic
deformation of \(\mathcal G\) along \(S^\circ\).  The real analytic
Kodaira--Spencer map of the pair \((\mathcal X^\circ,\mathcal G^\circ)\) at
\(0\in S\) is the linear map
\[
  \KS^{\mathbb R\mathrm{an}}_{\mathcal G}
  :
  \C T_0S^\circ
  \longrightarrow
  \mathbb H^1(X_0,\mathcal K_{\mathcal G})
\]
obtained by pulling the family back to first order real analytic arcs
\[
  \operatorname{Spec}B_1\longrightarrow S^\circ.
\]
Under the splitting
\[
  \C T_0S^\circ
  \cong
  T_{S,0}\oplus T_{\overline S,0},
\]
this map decomposes into its holomorphic and anti-holomorphic components.
\end{definition}

\begin{remark}
When the ambient space \(X\) is fixed, the Kodaira--Spencer class of the space
is zero.  Therefore the Atiyah obstruction vanishes automatically, and the
complex \(\mathcal K_{\mathcal F}\) reduces to
\(R\mathcal Hom_X(\mathcal F,\mathcal F)\).  Hence the moving-space theory
recovers Proposition~\ref{prop:ra-fixed-sheaf-KS}.
\end{remark}

\end{document}